\documentclass{amsart}
\usepackage{amsmath,amsthm,amssymb,amsfonts}
\usepackage{pdfsync}
\usepackage{fourier}
\usepackage[colorlinks=true, linktocpage=true]{hyperref}
\hypersetup{citecolor=blue}
\usepackage{booktabs, comment}

\title[Torsion growth and cycle complexity]{Torsion homology growth and cycle complexity of arithmetic manifolds}

\usepackage{xypic}
\usepackage{pdfsync}
\usepackage{amssymb}
\usepackage{color}
\usepackage{longtable}
\author{Nicolas Bergeron, Mehmet Haluk \c{S}eng\"un and Akshay Venkatesh} 

\address{Institut de Math\'ematiques de Jussieu \\
Unit\'e Mixte de Recherche 7586 du CNRS \\
Universit\'e Pierre et Marie Curie \\
4, place Jussieu 75252 Paris Cedex 05, France \\}
\email{bergeron@math.jussieu.fr}
\urladdr{http://people.math.jussieu.fr/~bergeron}
\thanks{The first author is a member of the Institut Universitaire de France.}

\address{Mathematics Institute \\
University of Warwick \\
Coventry, CV4 7AL, UK \\}
\email{M.H.Sengun@warwick.ac.uk}
\urladdr{http://warwick.ac.uk/haluksengun}
\thanks{The second author is funded by a Marie Curie Intra-European Fellowship.}

\address{Department of Mathematics \\
            Stanford University \\
            Stanford CA 94304 USA}
\email{akshay@math.stanford.edu}
\urladdr{http://math.stanford.edu/~akshay/}
\thanks{The third author is funded by a Packard foundation fellowship
and an NSF fellowship.}

 \DeclareFontEncoding{OT2}{}{} 
  \newcommand{\textcyr}[1]{%
    {\fontencoding{OT2}\fontfamily{wncyr}\fontseries{m}\fontshape{n}%
     \selectfont #1}}
\newcommand{\Sha}{{\mbox{\textcyr{Sh}}}}
 \DeclareFontFamily{OT1}{rsfs}{}
 \newcommand{\tors}{\mathrm{tors}}
\DeclareFontShape{OT1}{rsfs}{n}{it}{<-> rsfs10}{}
\DeclareMathAlphabet{\mathscr}{OT1}{rsfs}{n}{it}

\newcommand{\denom}{\mathrm{denom}}
\newcommand{\Spec}{\mathrm{Spec}}

\newcommand{\trunc}{\mathrm{tr}}

\newcommand{\C}{\mathbb{C}}

\newcommand{\Res}{\mathrm{Res}}

\newcommand{\Z}{\mathbb{Z}}

\newcommand{\frakn}{\mathfrak{n}}

\newcommand{\T}{\mathbf{T}}
\newcommand{\Ad}{\mathrm{Ad}}

\newcommand{\adele}{\mathbb{A}}
\newcommand{\Afinite}{\mathbb{A}_{\mathrm{f}}}
\newcommand{\HH}{\mathbf{H}}

\newcommand{\ab}{\mathrm{ab}}

\newcommand{\p}{\mathfrak{p}}

\newcommand{\Stor}{S-\mathrm{modf}}

\newcommand{\Gal}{\mathrm{Gal}}

\renewcommand{\H}{\mathbf{H}}

\newcommand{\g}{\mathfrak{g}}

\newcommand{\Norm}{\mathrm{N}}
\newcommand{\OO}{\mathcal{O}}

\DeclareFontFamily{OT1}{rsfs}{}

\DeclareFontShape{OT1}{rsfs}{n}{it}{<-> rsfs10}{}
\DeclareMathAlphabet{\mathscr}{OT1}{rsfs}{n}{it}
\newcommand{\Q}{\mathbb{Q}}
\newcommand{\Hom}{\mathrm{Hom}}

\newcommand{\G}{\mathbf{G}}

\newcommand{\R}{\mathbb{R}}

\newcommand{\A}{\mathbb{A}}
\newcommand{\ZZ}{\mathbf{Z}}
\newcommand{\TT}{\mathbb{T}}

\swapnumbers
\newtheorem{thm}[subsection]{Theorem}  

\newtheorem{lem}[subsection]{Lemma}         
\newtheorem*{lem*}{Lemma}         
\newtheorem{prop}[subsection]{Proposition}
\newtheorem*{prop*}{Proposition}

\newtheorem{conj}[subsection]{Conjecture}
 \newtheorem*{quest}{Question}

\theoremstyle{definition}

\numberwithin{equation}{subsection}

\newcommand{\N}{\mathbb N}

\newcommand{\bm}{\mathrm{BM}}

\renewcommand{\H}{\mathbb H}  %

\newcommand{\GL}{\mathrm{GL}}
\newcommand{\SL}{\mathrm{SL}}

\newcommand{\Lie}{\mathrm{Lie}}
\newcommand{\torus}{\mathbf{A}}

\newcommand{\PGL}{\mathrm{PGL}}

\newcommand{\vol}{\mathrm{vol}}

\newcommand{\PSL}{\mathrm{PSL}}

\begin{document}

\begin{abstract}
Let $M$ be an arithmetic hyperbolic $3$-manifold, such as a Bianchi manifold.
We conjecture that there is a basis for the second homology
of $M$, where each  basis element is represented by a surface of `low' genus, and give evidence for this. 
We explain the relationship between this conjecture  and  the study of torsion homology growth. 
\end{abstract}

\maketitle
\tableofcontents

\section{Introduction}

In this paper we   formulate and discuss a conjecture about topological complexity of arithmetic manifolds, 
i.e. locally symmetric spaces associated to arithmetic groups. This conjecture is closely related
to studying growth of torsion in homology. Roughly speaking, the conjecture is  that
\begin{quote}
{\em homology classes on arithmetic manifolds are represented by cycles of low complexity. }
\end{quote}
From a strictly arithmetic perspective, what may be most interesting is that
our proofs suggest that the {\em topological}  complexity of these cycles reflect the {\em arithmetic } complexity of the  
(Langlands-)associated varieties (i.e. the height of equations needed to define the varieties). 

We will study this in detail in a simple interesting case, namely, 
that of arithmetic hyperbolic $3$-manifolds.
To simplify matters as far as possible,
we study  only sequences  that are coverings of a fixed base manifold $M_0$.  

\begin{conj} \label{conj}
There is a constant $C=C(M_0)$ such that,  for any  arithmetic congruence hyperbolic $3$-manifold $M \rightarrow M_0$ of volume $V$,   
there exist immersed surfaces $S_i$ of genus $\leq V^C$
such that the $[S_i]$ span $H_2(M, \R)$. 
\end{conj}

Thus the conjecture is related to understanding the Gromov-Thurston norm on $H_2$; 
it can also be  phrased in  terms of a `harmonic' norm on $H_2$ whose definition uses the hyperbolic metric.
See \S \ref{GHT}. It follows from Gabai's generalization \cite[p. 3]{Gabai} of Dehn's lemma to higher genus that we may as well
ask the $S_i$ to be {\it embedded} in Conjecture \ref{conj}.

 It is plausible, although we are not sure, that this conjecture is really a special feature of arithmetic manifolds. For the purpose
 of this paper, ``arithmetic manifold'' means more properly ``arithmetic congruence manifold.''
Firstly, our proofs certainly use number theory heavily.   Secondly,  it seems that any `naive' 
analysis yields only an exponential bound on $[S_i]$ in terms
 of $V$ or the topological complexity of $M$ -- indeed,   work in progress of Jeff Brock and Nathan Dunfield
indeed strongly suggests that this exponential bound
 cannot be improved.  Finally, numerical data (see e.g. \cite{Dunfield} or \cite{sengun-tetrahedral}),
 although far from conclusive, also appears to differ between nonarithmetic and arithmetic cases.  See \S \ref{arith} for a little further discussion. 

This conjecture is motivated by the study of torsion classes, and indeed
in trying to understand the obstruction to extending previous results (see \cite{BV}) on `strongly acyclic' coefficient
systems to the case of the trivial local system. We will prove: 
\begin{thm} \label{torsiontheorem}
 Let $(M_i \rightarrow M_0)_{i \in \N}$ be a sequence of arithmetic congruence hyperbolic $3$-manifolds s.t. $M_0$ is compact and $V_i = \vol (M_i )$ goes to infinity.
Assume the following two conditions are satisfied:
\begin{itemize}
\item[(i)] `Few small eigenvalues':  For every $\varepsilon>0$ there exists some positive real number $c$ such that  
\begin{equation} \label{Strongcondition} \limsup_{i \to \infty} \frac{1}{V_i} \sum_{0 < \lambda \leq c} | \log \lambda| \leq \varepsilon. \end{equation}
Here $\lambda$ ranges over eigenvalues of the $1$-form Laplacian $\Delta$ on $M_i$.
Indeed we may even  replace the condition by 
  the condition that \begin{equation} \label{Weakcondition} \lim_{i \to \infty} \frac{1}{V_i} \sum_{0 < \lambda \leq V_i^{-\delta}} | \log \lambda| = 0 \end{equation}
  for every $\delta > 0$. 
   
%
\item[(ii)]  `Small Betti numbers': $b_1(M_i, \Q) = o (\frac{V_i}{\log V_i})$.
\end{itemize}

Then, if Conjecture \ref{conj} holds, as $i \to \infty$, we have:  
\begin{equation} \label{eqgrowth}  \frac{ \log \# H_1(M_i, \Z)_{\tors} }{V_i} \longrightarrow \frac{1}{6 \pi}. \end{equation}
\end{thm}
For the proof see \S \ref{sec:R1} (it also uses results from \S \ref{R1boundSec} and \S \ref{GHT}).  
Heuristically, we expect (i) to be valid with very few exceptions, and (ii) to be always valid; see \cite{CE,BLLS} for evidence,  and also \cite{Marshall} in a somewhat different direction.

 The proof of this Theorem also gives a partial converse. 
For instance, if we suppose \eqref{eqgrowth} and 
a strengthening of (ii) -- that the Betti numbers $b_1$ actually remain bounded --
then (i) must be true, and also 
a weak form of the Conjecture,   with ``polynomial''
replaced by ``subexponential,'' must hold.

Now the central result of our paper:

\begin{thm}
Conjecture \ref{conj} is true in the two following cases: 
\begin{itemize} 
\item[(i)] When $M_0$ arises from a division algebra $D \otimes F$ where $D$ is a quaternion algebra over $\Q$ and $F$ is an imaginary quadratic field, 
  $M$ is defined by a principal congruence subgroup\footnote{  This is not an onerous restriction; is easy to reduce the conjecture for other standard subgroup structures, such as $\Gamma_0$-structure, to this case. }, and all the 
cohomology of $M$ is of base-change type (\S \ref{BaseChange});
\item[(ii)] When $M_0$ is a Bianchi manifold  (for us: an adelic manifold whose components
are of the form $\Gamma_0(\mathfrak{n}) \backslash \H^3$),
and the cuspidal cohomology of $M$ is $1$-dimensional, associated to a non-CM elliptic curve of conductor $\mathfrak{n}$, 
for which we assume the equivariant BSD conjecture (see \eqref{eqBSD})  and the Frey--Szpiro conjecture (see \cite[F.3.2]{HS}). 
\end{itemize}
\end{thm}  

What the proof of (ii) really gives is a relationship between the complexity of $H_2$-cycles
and the height of the elliptic curve (i.e., the minimal size of $A,B$ so it can be expressed as $y^2=x^3+Ax+B$.) Thus, ``the topological complexity of cycles in $H_2$ reflect
the arithmetic height of $E$.''   This  may be a general phenomenon (it was also suggested in \cite{CV}). 

A few words on the conjectures which appear in (ii): The Frey--Szpiro conjecture is a conjecture in Diophantine analysis
which follows from the ABC conjecture  (and thus is very strongly expected from a heuristic viewpoint).  It asserts
that the height of an elliptic curve cannot be too large relative to its conductor. 
Moreover,  for the purposes of establishing growth of torsion, as in Theorem \ref{torsiontheorem},
we do not need the full strength of Conjecture \ref{conj};
a weaker version with sub-exponential bounds  would suffice,
and correspondingly a very weak ``sub-exponential'' version of Frey--Szpiro would do.

We note that both case (i) and case (ii) are quite common over imaginary quadratic fields!
For (i),  we present data in \S \ref{numerics}:
e.g. for the first 40 rational primes $p$ that are inert in $\Q(\sqrt{-7})$, the cohomology of $\Gamma_0(\mathfrak{p})$, where $\mathfrak{p}=( p)$, is entirely base change in all but
$6$ cases.  For (ii) we refer to \cite[p.17]{SengunExpMath}; in the data there, at prime level,
situation (ii) occurs in the majority of cases where $b_{1, !} > 0$, see also \S \ref{numerics2}.  

Also, (i) and (ii) illustrate two different extremes of the Theorem:

For (i)  it's  easy to think of candidate surfaces in $H_2$ ---  the challenge is, rather,  that the dimension of $H_2$
is increasing rapidly and it is not clear that the candidate surfaces span `enough' homology.
In fact, our result applies to all $M$, but bounds only the regulator of the `base-change part'
of cohomology.  One can see (i) as an effectivization of a result of Harder, Langlands and Rapoport \cite{HLR},
although they work with Hilbert modular surfaces rather than hyperbolic $3$-manifolds. The main global ingredient is a (polynomially strong) quantitative form of 
the `multiplicity one' theorem in the theory of automorphic forms but there is also (surprisingly) a nontrivial local ingredient:  one needs good control 
on (e.g.) support of matrix coefficients of supercuspidal representations. 
 In fact, one motivation to study example (i) is that our result
shows that the   regulator $R_2$
(see \S \ref{sec:R1}) grows subexponentially, whereas this was not at all clear by looking at numerical evidence! --- see \S \ref{numerics3}.
  (There is actually another setting where $H_2$ grows quickly for easily comprehensible reasons -- the setting of ``oldforms,''
  whereby one pulls back forms from a surface of lower level. In that case, it is not difficult to see
  that the complexity of the cycles remains controlled.)

For (ii) the challenge is instead that there are no obvious cycles in $H_2$;
we work with $H_1$ and modular symbols, and dualize; the main point is 
  to replace a modular symbol by the sum of two well-chosen others to avoid unpleasant dominators. 
  The equivariant BSD conjecture enters to compute cycle integrals over modular symbols. 
The Szpiro conjecture enters to give a lower bound on the period of an elliptic curve.
We note that this result is closely related to prior work of Goldfeld \cite{Goldfeld},
although the techniques of proof are necessarily different owing to the lack of an algebraic structure.

\subsection{The role of arithmeticity} \label{arith}
As we have mentioned, it seems plausible that Conjecture \ref{conj} is really specific to arithmetic. 
It would be desirable to have a specific counterexample in this direction, that is to say, 
exhibiting the behavior that Conjecture \ref{conj} disallows in the arithmetic case.

From the point of view of mirroring the situation of this paper, it would be ideal to have an answer to the following: 

\begin{quest} Can one produce a sequence of hyperbolic manifolds $M_i$
with the  following properties?
\begin{itemize}  \item[-] the volumes  of $M_i$ go to infinity (or, even better, the sequence $(M_i)$ BS-converges toward $\H^3$, see \S \ref{BS}),
\item[-] The injectivity radii of $M_i$ remain bounded below, and 
\item[-] in any basis for $H_2(M_i, \Z)$, at least one basis element {\em cannot}
be represented by a surface of genus $\leq (\vol M_i)^{i}$?
\end{itemize}
\end{quest}

Jeff Brock and Nathan Dunfield have made progress in constructing such a sequence.

Here is some intuition as to why arithmeticity might play a role: In general, generators for $H_2(M, \Z)$
 might be of exponential complexity. This comes down to analyzing the kernel of a matrix
 $M$ that expresses adjacency between $1$-cells and $2$-cells in a triangulation. 
Now, even given a matrix
 $A \in \mathrm{M}_n(\Z)$ of zeroes and ones, generators for the kernel of $A$  on $\Z^n$ could
 have exponentially large (in $n$) entries. However, in the arithmetic case, 
 the existence of Hecke operators means that the `adjacency matrix' $A$  is (heuristically speaking)
 forced to commute with many other symmetries.   One might expect
this to reduce its effective size --- a phenomenon that is perhaps parallel
to the observed difference  between eigenvalue statistics in the arithmetic and nonarithmetic case 
(see \cite{GUE} for discussion). 

\subsection{Acknowledgements}
We would like to thank Ian Agol, Farrell Brumley, Nathan Dunfield,  Dipendra Prasad and Jonathan Pfaff for helpful comments and providing useful references. 
      
\section{Relationship to torsion and the proof of Theorem \ref{torsiontheorem}} \label{sec:R1}

In this section and the next, we will give the proof of Theorem \ref{torsiontheorem}. 
We first recall the definition of `regulators' from a prior paper \cite{BV} by the first- and last- named author (N.B. and A.V.)

\subsection{Regulators} 
Let $M$ be a compact Riemannian manifold of dimension $n$. 
We define
the {\it $H_j$-regulator} of $M$ as  the volume of $H_j(M, \Z)$  with respect to the metric on $H_j(M, \R)$ defined by harmonic forms --- 
the `harmonic metric.' That is, 
\begin{equation} \label{Rjdef} R_j(M) = \frac{ \det \left( \int_{\gamma_k} \omega_{\ell} \right)}{ \sqrt{ \det  \langle  \omega_k, \omega_{\ell}  \rangle}} \end{equation}
where $\gamma_k \in H_j (M, \Z )$ project to a basis for $H_j(M, \Z ) / H_j (M, \Z )_{\rm tors}$ 
and $\omega_{\ell}$ are a basis for the space of $L^2$ harmonic forms on $M$. Note that $R_0 (M) = \frac{1}{\sqrt{\vol (M)}}$, $R_n (M) = \sqrt{\vol (M)}$ and, by Poincar\'e duality, we have: 
\begin{equation*}
R_j (M) \cdot R_{n-j} (M) = 1.
\end{equation*}

\subsection{} A celebrated theorem of Cheeger and M{\"u}ller \cite{Cheeger, Mueller} relates the torsion homology groups and the regulators to the analytic torsion of $M$. In the special case $n=3$, the theorem of Cheeger and M{\"u}ller implies that 
$$|H_1(M, \Z)_{\tors}| \cdot  \frac{R_0 R_2}{R_1 R_3} = T_{\rm an}(M)^{-1},$$
where $T_{\rm an} (M)$ is the analytic torsion of the manifold $M$. We furthermore note that
$$\frac{R_0 R_2}{R_1 R_3} =  \frac{R_2^2}{\vol (M)} = \frac{1}{R_1^2 \vol (M)}.$$

\subsection{Benjamini-Schramm convergence} \label{BS}
For a hyperbolic manifold $M$ we define $M_{<R}$ to be the $R$-thin part of $M$, i.e. the part of $M$ where the local injectivity radius is $<R$. 
Now let $(M_i \to M_0)_{i \in \N}$ be a sequence of finite covers of a fixed compact hyperbolic $3$-manifolds. Following \cite{ABBGNRS}, we say that the sequence
$(M_i)_{i \in \N}$ {\it BS-converges} to $\H^3$ if for all $R>0$ one has $\vol (M_i )_{<R} / \vol (M_i ) \to 0$. It follows from \cite[Theorem 1.12]{ABBGNRS} that if $(M_i \to M_0)_{i \in \N}$ is a sequence of arithmetic congruence compact hyperbolic manifolds s.t. $V_i = \vol (M_i )$ goes to infinity then $(M_i)_{i \in \N}$ BS-converges to $\H^3$. 
The proof of Theorem \ref{torsiontheorem} then follows from the following three ingredients: 

\subsection{First ingredient} We shall show in the next section (Proposition \ref{P}) that there exists some constant $C$ s.t. 
\begin{equation} \label{R1bound} R_1 (M_i) \ll \vol(M_i)^{C b(M_i)}, \end{equation}
where $b(M)= b_1 (M,\Q)= b_2 (M,\Q)$ is the Betti number. In particular, so long as $b(M_i)$ grows as $o (\frac{V_i}{\log V_i})$,
the subexponential growth of $R_1 (M_i)$ follows.   (Here and below, subexponential means subexponential in $V_i$). 

\subsection{Second ingredient} We will also show in \S \ref{R2P} that,  assuming Conjecture \ref{conj}, there exists a constant $C$ such that \begin{equation} \label{R2bound} R_2(M_i) \ll \vol(M_i)^{C b(M_i)}. \end{equation}
So here again, as long as $b(M_i)$ grows as $o (\frac{V_i}{\log V_i})$,
the subexponential growth of $R_2(M_i)$ {\em follows} from Conjecture \ref{conj}.

\subsection{Third ingredient} Finally,  the condition `few small eigenvalues' from Theorem \ref{torsiontheorem} implies that 
\begin{equation} \label{Tcv}
\frac{\log T_{\rm an} (M_i)}{V_i} \to \tau_{\H^3}^{(2)} = -\frac{1}{6\pi}.
\end{equation}
It follows from the definition of analytic torsion and well known properties of the spectrum of the Laplace operators on Riemannian $3$-manifolds
(see e.g. in \cite{BV}) that it is enough to prove that 
\begin{equation} \label{lim1}
\frac{d}{ds} \big|_{s=0} \frac{1}{\Gamma (s)} \int_0^{+\infty} t^{s-1} \frac{1}{V_i} \int_{M_i} \left( \mathrm{tr} e^{-t \Delta^{(2)}} (\tilde x, \tilde x) -  (\mathrm{tr} e^{-t \Delta_i} (x,x) -b_1 (M_i))
\right) dx dt \to 0.
\end{equation}
Here $\Delta_i$, resp. $\Delta^{(2)}$, is the Laplace operator on square-integrable $1$-forms on $M_i$, resp. $\H^3$, and $\tilde x$ is an arbitrary lift of $x$ to $\H^3$.

Since $b(M_i)$ grows as $o (\frac{V_i}{\log V_i})$ the proof of the limit \eqref{lim1} follows the same lines as \cite[Theorem 4.5]{BV} under the assumptions that
\begin{enumerate}
\item the injectivity radius of $M_i$ goes to infinity; and
\item there exists some positive $c$ such that for all $M_i$
the lowest eigenvalue of $\Delta_i$ is bigger than $c$. 
\end{enumerate}
The first assumption is used to handle the `small $t$' contribution to the limit \eqref{lim1}. In fact the proof only uses the fact that the local injectivity radius 
is `almost everywhere' going to infinity, the condition is precisely that the sequence $(M_i)_{i \in \N}$ BS-converges to $\H^3$. We refer to \cite[\S 8 and 9]{ABBGNRS} for more details in particular on how to bound the size of the heat kernel at the bad points.

The second assumption is used to handle the `large $t$' contribution; it more precisely implies that for sufficiently large $t$ each individual term of the 
difference in \eqref{lim1} can be made arbitrary small. However this spectral gap assumption never holds for the trivial coefficient system; we 
replace that instead by assumption (i) of Theorem \ref{torsiontheorem}. 

 Let $\varepsilon$ and $c$ be as in  assumption (i) of Theorem \ref{torsiontheorem}.  Without loss of generality, $c<1$. Spectral expansion on the compact manifold $M_i$ and classical Sobolev estimates yield that for any $t\geq 1$ we have: 
\begin{eqnarray*} \int_{M_i} \mathrm{tr} e^{-t\Delta_i '} (x,x)  &=&  \sum_{0 < \lambda}  e^{-\lambda t} \\  & \ll & \sum_{0 < \lambda \leq c} e^{-t \lambda} + 
e^{-c(t-1)} \sum_{\lambda > c} e^{-\lambda}    \\
& \ll &  \sum_{0 < \lambda \leq c} e^{-t \lambda} +e^{-c(t-1)} V_i,  \end{eqnarray*}
where we have denoted by $\Delta_i '$ the restriction of $\Delta_i$ to the
orthogonal complement of its kernel and the implicit constant does not depend on $i$ and $x$; we used the fact that the trace of $e^{-\Delta_i '}$ on $M_i$ can be bounded
by a multiple of $V_i$.  To conclude the proof we just have to remark that 
for any $T \geq 1$ fixed 
\begin{eqnarray}  \nonumber \frac{d}{ds} \big|_{s=0} \frac{1}{\Gamma (s) }\int_T^{+\infty} t^{s-1} \sum_{0 < \lambda \leq c} e^{-t \lambda} dt  &=& \sum_{0 < \lambda \leq c} \int_{T}^{\infty} e^{-t\lambda} \frac{dt}{t}  
 \\ \nonumber    &=& \sum_{0 < \lambda \leq c} \int_{T}^{\infty} \frac{e^{-t}}{t} dt + \sum_{0 <\lambda \leq c} \int_{T}^{\infty} \frac{  e^{-t \lambda}-e^{-t}}{t} dt
\\  \nonumber  &<&  \left( \mbox{number of eigenvalues in $(0,c]$ }\right) e^{-T} +  \sum_{0 < \lambda \leq c}  \log| \lambda |. \end{eqnarray} 

There is a constant $A$ so that the number of eigenvalues in $(0,c]$ is $\leq A V_i$, and we may then
choose $T$ sufficiently large so that $A e^{-T} < \varepsilon$. Thus 
the integral above contributes at most $2 \varepsilon$ to the limit. 
Using \eqref{Strongcondition}, this holds for every $\varepsilon$, so  the proof of \eqref{lim1} follows as in \cite{BV}.

 We have now completed the proof of the Theorem, but assuming (i) in the stronger form \eqref{Strongcondition}.
 To see that \eqref{Weakcondition} suffices: 
  
 \begin{lem} \label{vp} Assume that \eqref{Weakcondition}. Then, 
for every $\varepsilon>0$ there exists some positive real number $c$ such that  
$$\lim_{i \to \infty} \frac{1}{V_i} \sum_{0 < \lambda \leq c} | \log \lambda| \leq \varepsilon.$$
Here $\lambda$ ranges over eigenvalues of  the $1$-form Laplacian $\Delta_i$ for $M_i$. 
\end{lem}
 \proof In \cite[Theorem 1.12]{ABBGNRS} a quantitative version of BS-convergence is proven; in particular there exist 
positive constants $c$ and $\delta$ such that for every $i$ one has
$$\vol (M_i)_{<c \log V_i} \leq V_i^{1-\delta}.$$
Fix $M=M_i$ and $V=V_i$.  Employing the trace formula --- as in \cite{Rudnick,Sarnak} --- with a test function supported 
in an interval of length $1/(c' \log V)$ for some positive constant $c'$,   and using the estimates
of \cite[Lemma 7.23]{ABBGNRS}, we can show that for every $k \in \N$ the number of eigenvalues of $\Delta$ between 
$\frac{k}{c'\log V}$ and $\frac{k+1}{c'\log V}$ is bounded by some uniform constant times $V/\log V$.
It follows that 
$$\prod_{\frac{k}{c'\log V} < \lambda \leq \frac{k+1}{c' \log V}} \lambda \gg \left( \frac{k}{\log V} \right)^{\frac{V}{\log V}}.$$
Taking a further product for $k=1 , \ldots , \alpha \log V$ for some positive $\alpha$ (and using Stirling's formula) we get that 
\begin{equation} \label{vp1}
\prod_{\frac{1}{c' \log V} < \lambda \leq \alpha} \lambda \gg \left( \frac{(\alpha \log V) !}{(\log V)^{\alpha \log V}} \right)^{\frac{V}{\log V}} \gg e^{-o(1) V}
\end{equation}
as $\alpha \to 0$.

Now given a positive real number $\delta$ we similarly have:
\begin{equation} \label{vp2}
\prod_{V^{-\delta} < \lambda \leq \frac{1}{\log V}} \lambda \gg \left( \frac{1}{V^\delta } \right)^{\frac{V}{\log V}} = e^{-\delta V}.
\end{equation}
The lemma follows from \eqref{vp1}, \eqref{vp2} and the `few eigenvalues' assumption.
\qed

\section{Bounding $R_1(M)$} \label{R1boundSec}

Here we prove \eqref{R1bound} that was used in the proof of Theorem \ref{torsiontheorem}.
Let $M_0$ be a complete Riemannian $n$-dimensional manifold of pinched nonpositive sectional curvature. We  
more generally prove the following:

\begin{prop} \label{P}
If $M$ varies through a sequence of finite coverings of a fixed compact manifold $M_0$, 
we have:
$$|R_1 (M)|  \ll \mathrm{vol} (M )^{C b (M )}.$$
Here the implicit constants only depend on $M_0$.
\end{prop}

The following is a consequence of Sobolev estimates: 
\begin{lem} \label{A}
Let $M$ be as in Proposition \ref{P}, let $S\subset M$ be a $k$-submanifold of  (Riemannian) volume $v$ and let $\omega$ be an $L^2$-normalized harmonic differential $k$-form on $M$. Then:   
$$\int_S \omega \ll v.$$
\end{lem}

We now explain how to prove Proposition \ref{P} using Lemma \ref{A}.

\subsection{} Fix $M_0$ and let $\Gamma_0$ be the fundamental group of $M_0$,
let $S$ be a set of generators of $\Gamma_0$ and let $d_0$ be the cardinality of $S$.

To any finite covering $M \to M_0$ --- corresponding to a finite index subgroup $\Gamma < \Gamma_0$ --- we associate the Schreier graph 
$\mathcal{G} (\Gamma_0 / \Gamma , S)$; it is a finite cover of degree $[\Gamma_0 : \Gamma]$ of the wedge
product of $d_0$ circles. Computing the Euler characteristic we conclude that $\mathcal{G} (\Gamma_0 / \Gamma , S)$ 
has the homotopy type of the wedge product of $d$ circles where:
$$(d -1) = [\Gamma_0 : \Gamma ] (d_0-1).$$
The group $\Gamma$ is therefore generated by at most $d$ elements; moreover each of these
elements has length at most $[\Gamma_0 : \Gamma ]$ in the $S$-word metric of $\Gamma_0$. 

Since $\Gamma_0$ with the $S$-word metric is quasi-isometric to the universal cover $\widetilde{M}$ of
$M$ with its induced Riemannian metric we have the following:

\begin{lem} \label{B}
There exists a constant $c=c(M_0)$ such that $\Gamma$ is generated by at most $c[\Gamma_0 : \Gamma]$ elements which can be represented by closed geodesics of length $\leq c[\Gamma_0 : \Gamma ]$. 
\end{lem}

Note that up to a constant (depending only on $M_0$) $\mathrm{vol} (M)$ equals $[\Gamma_0 : \Gamma ]$. Hadamard's inequality, Lemma \ref{A} and Lemma \ref{B} therefore imply Proposition \ref{P}.
(Note that, in the definition \eqref{Rjdef} of the regulator, replacing the $\gamma_j$ by elements $\gamma_j'$ that generate a finite index sublattice
of homology only increases the regulator.) 

\subsection{} Assuming Conjecture \ref{conj} we can apply a similar scheme to bound $R_2 (M)$, but we now need to compare two different norms on 
$H_2(M, \R)$. This is the purpose of the next section.

\section{Relationship
 of the harmonic norm and  the Gromov--Thurston norm} \label{GHT}
 
In this section, $M$ will denote a compact hyperbolic $3$-manifold.
The second homology group $H_2(M, \R)$  is equipped with
 two natural norms:  the Gromov--Thurston norm, which measures the number of simplices
needed to present a cycle, and the harmonic norm, which arises from the identification of $H_2(M, \R) \simeq H^1(M, \R)$ with harmonic $1$-forms on $M$.   
We will relate the two norms and use it to prove \eqref{R2bound}, used in the proof of Theorem \ref{torsiontheorem}.

 More precisely: if $\delta \in H_2(M, \R)$ we set
$$ \|\delta \|_{GT} = \inf \left\{ \sum |n_k | \; \big| \; [\sum n_k \sigma_k ] = \delta \mbox{ where } \sum n_k \sigma_k \mbox{ is a singular chain} \right\}.$$
Note that Gabai \cite[Corollary 6.18] {Gabai} shows that if $\delta \in H_2 (M, \Z)$ then
\begin{multline} \label{Gabai}
\|\delta \|_{GT} \\ = 2 \min \left\{ \sum_{i , \ \chi (S_i ) <0} |\chi (S_i )| \; \Big| \; \begin{array}{l} S = \cup_i S_i, \mbox{ where } S_i \mbox{ is a properly embedded}  \\
\mbox{connected surface in } M
\mbox{ and } [S] = \delta \mbox{ in } H_2(M, \Z) \end{array}\right\}.
\end{multline}

  Note that, since $M$ is compact hyperbolic, we may suppose that each $S_i$ actually a  surface of genus $\geq 2$,
 since if $S$ is either a sphere or a torus the image of $H_2(S, \Z)$ in $H_2(M, \Z)$ will be trivial. In particular, to prove
 the theorem, it is enough to exhibit a set in $H_2(M, \Z)$ of full rank, and with polynomially bounded Gromov-Thurston norm. 
   
We also define $ \|\delta \|_{L^2} =  \|\omega \|_{L^2}$ where $\omega$ is the $L^2$ harmonic $1$-form on $M$ which is dual to $\delta$, i.e.
$$\int_M \omega \wedge \alpha = \int_{\delta} \alpha,$$
for every closed $2$-form $\alpha$ on $M$. Note in particular that 
\begin{equation}
 \|\delta \|_{L^2}^2 = \left| \int_{\delta} * \omega \right|.
\end{equation}
In this section we compare $\|\cdot\|_{L^2}$ and $\|\cdot\|_{GT}$.
In particular, we prove the following:

\begin{prop} \label{P1}
If $M$ varies through a sequence of finite coverings of a fixed manifold $M_0$, 
we have:
 $$  \frac{1}{\vol (M)} \|\cdot\|_{GT} \ll  \|\cdot\|_{L^2} \ll \|\cdot\|_{GT}.$$
\end{prop}

\begin{proof}
The proof occupies \S \ref{proofstart}--\S\ref{proofend} below.
 
\subsection{}  \label{proofstart}
 Given a cycle $\delta \in H_2(M, \R)$ with $\|\delta\|_{GT} \leq 1$, 
we may write (see e.g. \cite[Theorem 11.4.2 and the Remark following it]{Rattcliffe}) 
 $$\delta = \sum_k n_k \sigma_k$$ where each $\sigma_k$ is a {\it straight} simplex (or triangle), i.e. the image of the convex hull of $3$ points 
 in $\H^3$, and $\sum_k |n_k | \leq 1$. 
Now if $\alpha$ is a harmonic $2$-form, then 
$$\int_{\sigma_k} \alpha \ll ||\alpha ||_{\infty} \mathrm{area} (\sigma_k ) \leq \pi || \alpha ||_{\infty} \ll  || \alpha ||_{2}$$
is uniformly bounded so that $\int_{\delta} \alpha \ll   \|\alpha\|_2$.
Since we can compute the harmonic  norm of $\delta$ as the operator norm of $\alpha \mapsto \int_{\delta} \alpha$, 
this has shown the second inequality of Proposition \ref{P1};
we pass now to the first inequality.

\subsection{} 
In the reverse direction,  suppose given an element $\delta \in H_2(M, \R)$
of harmonic norm $\leq 1$; equivalently, its image under
$H_2(M, \R) \simeq H^1(M, \R)$
is represented by a harmonic $1$-form $\omega$ of $L^2$-norm $\leq 1$. 

 Fix a triangulation $K$ of $M$
by lifting a triangulation $K_0$ of $M_0$. We can suppose
that  every edge has length $\leq 1$ and every triangle has area $\leq 1$. 
Let $K'$ be the dual cell subdivision. We denote by 
$$\langle , \rangle : C_i (K , \Z) \times C_{3-i} (K' , \Z) \to \Z$$
the (integer) intersection number; it canonically identifies $C_{3-i} (K' , \Z)$ with the dual $C^i (K , \Z) = C_i (K , \Z)^*$
of $C_i (K , \Z)$. Furthermore, the boundary homomorphism 
$$\partial : C_{3-i} (K', \Z) \to C_{3-i-1} (K' , \Z)$$
is (up to sign) dual to the corresponding boundary homomorphism $C_{i+1} (K , \Z) \to C_i (K , \Z)$,
in other words $\partial$ identifies (up to sign) with the coboundary homomorphism $C^i (K , \Z) \to C^{i+1} (K, \Z)$. 
Both complexes $C_{\bullet} (K , \Z)$ and $C_{\bullet} (K' , \Z)$ compute $H_{\bullet} (M , \Z)$. Now the latter identifies with $C^{3-\bullet} (K , \Z)$
and computes $H^{3-\bullet} (M , \Z)$. This realizes the Poincar\'e duality.

\subsection{} Consider, then, the two-cycle 
 $$Z:= \sum_{e} \left( \int_e \omega \right) e^* \in C^1 (K , \R) = C_2 (K' , \R).$$
Since $\omega$ is closed, it follows from Stokes formula that
$$ \partial Z = \pm \sum_{t}  \left( \int_{\partial t} \omega \right) t^* = 0.$$
On the other hand, $Z$ represents the image of the class of $[\omega]$ under 
the Poincar{\'e} duality pairing
$$H^1(M, \R) \stackrel{\sim}{\rightarrow} H_{2}(M, \R).$$
 
\subsection{}  \label{proofend}
Subdivide $K_0'$ to get a triangulation $T_0$ of $M_0$ and lift this triangulation to a triangulation $T$ of $M$. 
There exists a constant $c$ which only depends on $M_0$ (and $T_0$) such that the number of triangles of $T_0$ in each cell of $K_0'$ dual to an edge of $K_0$ is bounded by $c$. Then the number of
triangles of $T$ in each cell of $K'$ is $\leq c [M:M_0]$, where $[M:M_0]$ is the degree of the cover $M\to M_0$. By definition of the Gromov--Thurston norm 
we conclude that 
$$ \| [Z] \|_{GT} \ll ||\omega ||_{\infty} \vol (M) \ll   \vol(M),$$
the last by the Sobolev inequality. 
Proposition \ref{P1} now follows. \end{proof}

\subsection{Relation with $R_2 (M)$}  \label{R2P} Let us now assume Conjecture \ref{conj}. Then each $[S_i]$ has Gromov--Thurston norm --- and therefore, by Proposition \ref{P1}, harmonic norm --- which is bounded by a polynomial in $\vol(M)$. 
Thus here again Hadamard's inequality shows that
$$R_2(M) \ll \vol(M)^{C b(M)},$$
where $b(M)$ is the Betti number.  

We have now concluded the proof of Theorem \ref{torsiontheorem}.

\section{Arithmetic manifolds} 
\label{sec:arithmanifolds} 
Let $F$ be an imaginary quadratic field. We consider arithmetic manifolds associated to an 
algebraic group $\G$ over $\Q$ such that $\G(\R)  = \PGL_2(\C)$.  In this paper we are interested in the two examples: 

\begin{itemize}
\item  $\G_1 = \mathrm{Res}_{F/\Q} \GL_2$ mod center;
\item $\G_2 = \mathrm{Res}_{F/\Q} \GL_1(D')$ mod center, 
where $D'$ is a  division algebra over $F$ of the form $D'=D \otimes F$, with $D$  a quaternion algebra over $\Q$. 
\end{itemize}

Let $\A$ and $\A_F$ be the ring of ad\`eles of $\Q$ and $F$ respectively. We denote by $\A_{f}$ and $\A_{F,f}$ the corresponding rings of finite ad\`eles. 
We also write $F_{\infty} = F \otimes \R \simeq \C$. 
In remaining part of this paper $\G$ stands for either $\G_1$ or $\G_2$.  

In the second case $\G_2$ admits a $\Q$-subgroup which will be of importance to us: Let $\HH = \GL_1(D)$ modulo center,
considered as a subgroup of $\G_2$.  Thus, $\HH(\R) = \PGL_2(\R)$. 

\subsection{The arithmetic manifold $X(\mathfrak{n})$} \label{arithX}
Let $\mathfrak{n}$ be an ideal of the ring of integers $\mathcal{O}$ of $F$. We associate to $\mathfrak{n}$ a compact open subgroup 
$$K(\mathfrak{n}) = \prod_v K_v (\mathfrak{n})  \subset \G(\A_f)$$
in the following way. If $\G = \G_1$ as usual we define $K (\mathfrak{n})=K_1 (\mathfrak{n})$ as the subgroup corresponding --- after restriction of scalars and mod center --- to 
\begin{equation} \label{Kn} \{ g \in \GL_2 (\widehat{\mathcal{O}}) \; : \; g \equiv I_2 \ ( \mathfrak{n} \widehat{\mathcal{O}} ) \}.\end{equation} 
Here $\widehat{\mathcal{O}}$ is the closure of $\mathcal{O}$ in $\A_{F,f}$. In this case,  we
also define $K_0(\mathfrak{n})$ in the usual way
$$K_0(\mathfrak{n}) = \{ g \in \GL_2 (\widehat{\mathcal{O}}) \; : \; g \equiv  \left(\begin{array} {cc} * & * \\ 0 & * \end{array} \right)  \ ( \mathfrak{n} \widehat{\mathcal{O}} ) \}.$$

Now if $\G = \G_2$ we make a corresponding definition
of $K(\mathfrak{n})$ as follows: 
Regard $\G_2(\adele_f)$
as the $\adele_{F,f}$-poitns of $\GL_1(D')$ mod center. In  the paragraph that follows, 
products over places $v$ will be over places of $F$. 
First make an arbitrary choice $K = \prod_v K_v \subset \G_2(\adele_f)$ of a compact open subgroup such that $K_v$ is hyperspecial at each unramified place $v$.
At those places we may then fix isomorphisms $\phi_v : \G_2 (\Q_v) \to \PGL_2(F_v)$
such that $\phi_v(K_v) = \PGL_2(\OO_v)$, where $\OO_v \subset F_v$ is the ring of integers. We finally define 
$K(\mathfrak{n}) = \prod_v K_v (\mathfrak{n})$ by setting $K_v (\mathfrak{n}) = K_v$ at each ramified place and 
$K_v (\mathfrak{n}) = \phi_v^{-1} (K_{1,v} (\mathfrak{n}))$ at each unramified place, where $K_{1,v}(\mathfrak{n})$ is 
the local analog of \eqref{Kn}.  We will also suppose, for at least one ramified place $v$, 
the subgroup $K_v$ is sufficiently small so as to force any group $\G(\Q) \cap K(\mathfrak{n})$ to be torsion-free. 

Given any compact open subgroup $K \subset \G (\A_f)$, we define the  arithmetic manifold  $X(K)$
by 
$$X(K)  = \G(F) \backslash (\H^3 \times \G(\A_f)) / K .$$
We simply denote by $X(\mathfrak{n})$ the arithmetic manifold $X(K(\mathfrak{n}))$.

\subsection{Connected components of $X(\mathfrak{n})$.}

The connected components of $X(\mathfrak{n})$ can be described as follows. Write 
$\G (\A_f ) = \sqcup_j \G (\Q) g_j K(\mathfrak{n})$; then 
$$X(\mathfrak{n}) = \sqcup_j \Gamma_{j} \backslash \H^3,$$
where $\Gamma_j$ is the image in $\PGL_2 (\C)$ of 
$\Gamma_j ' = \G(\Q)\cap g_j K(\mathfrak{n}) g_j^{-1}.$
We let 
$$Y(\mathfrak{n}) = \Gamma \backslash \H^3$$
denote the connected component of $X(\mathfrak{n})$ associated to the class $g_j =e$ of the identity element
so that $\Gamma$ is the image in $\PGL_2 (\C)$ of 
$\G (\Q) \cap K(\mathfrak{n})$.

Note that $X(K(\mathfrak{n}))$ is the quotient of
$$\G(\Q) \backslash (\PGL_2 (\C) \times \G(\A_f)) / K(\mathfrak{n}) ,$$
by 
$$K_{\infty} = \mbox{ image in $\PGL_2(\C)$ of } \left\{ \left( \begin{array}{cc} a & b \\ -\bar{b} & \bar{a} \end{array} \right) : |a|^2+|b|^2=1 \right\}$$ 
the maximal compact subgroup at infinity. Here we have chosen an identification of $\G(\C)$ with $\PGL_2(\C)$; 
in case (ii), we require that this identification carry $(D')^{\times}$
into $\PGL_2(\R)$, so that in particular $K_{\infty}$ intersects  $\HH(\R)$ 
in a maximal compact subgroup. 

In the $\G_2$ case both $Y(\mathfrak{n})$ and $X(\mathfrak{n})$ are compact manifolds. In the $\G_1$ case both are noncompact of finite volume.

\subsection{Hecke operators}  \label{sec:Hecke} Suppose $g \in \G (\A_f)$,  that $K'$ is another compact open subgroup of $\G (\A_f)$ and $K' \subset g K(\mathfrak{n}) g^{-1}$. 
The map $\G( \A) \to \G (\A)$ given by $ h \mapsto hg$ defines a continuous mapping
$$r(g) : X(K') \to X(\mathfrak{n}).$$
Taking $K'= K(\mathfrak{n}) \cap gK(\mathfrak{n}) g^{-1}$ we define the {\it Hecke operator} 
$\TT(g) : H^{\bullet} (X(\mathfrak{n})) \to H^{\bullet} (X(\mathfrak{n}))$ to be the composition 
$$H^{\bullet} (X(\mathfrak{n})) \stackrel{r(g)^*}{\to} H^{\bullet} (X(K')) \stackrel{r(1)_*}{\to} H^{\bullet} (X(\mathfrak{n})).$$

For suitable choice of $g$ this gives rise to the usual Hecke operators $T_{\mathfrak{m}}$,
which are attached to any ideal $\mathfrak{m}$ of $F$ 
which is ``relatively prime to ramification,'' i.e. no prime divisor of $\mathfrak{m}$
lies above any place $v$ of $\Q$ where $K_v$ is non-maximal.

\subsection{The truncation in the Bianchi case} 
 Now assume that $\G = \G_1$, so that we are in the noncompact case. 
 We denote by $X(\mathfrak{n})_{\trunc}$ a `truncation' of $X(\mathfrak{n})$, where
 we ``chop off the cusps.'' Thus $X(\mathfrak{n})_{\trunc}$ is a manifold with boundary, and 
 up to homeomorphism it does not depend on the height at which the cusps were cut off. 
 
Connected components of 
$Y(\mathfrak{n})_{\trunc}$ are homeomorphic to the compact quotient $\Gamma \backslash \H^3_* $ 
where $$\H^3_* := \H^3 \ \backslash \bigcup_{\sigma \in \mathbb{P}^1(F)} B(\sigma)$$
where the $B(\sigma)$ are a disjoint collection of horospheres in $\H^3$ tangent to the rational
boundary point $\sigma$. In particular, $\Gamma \backslash \H^3_* $ looks like 
a thickening of the 2-skeleton of $\Gamma \backslash \H^3$.   
\section{Complexity of base-change cohomology classes} \label{BaseChange}

In this section $\G = \G_2$ and $M=Y(\mathfrak{n})$ is an associated congruence arithmetic manifold.
We address Conjecture \ref{conj} for base-change cohomology classes of $M$. 
We recall below the definition of the 
base-change part $H^2_{\rm bc} (M)$ of the cohomology $H^2 (M)$.  Note that
in this section, when we write $H^*(M)$ etc. without coefficients, we always mean complex cohomology.
Here we prove: 

\begin{thm} \label{Tbc}
There is a constant $C=C(F)$ such that, for any arithmetic hyperbolic manifold $M=Y(\mathfrak{n})$ of volume $V$, there exist compact immersed surfaces
$S_i$ of genus $\leq V^C$ such that the  homology classes  $[S_i] \in H_{2} (M)$ span $H_2^{{\rm bc}} (M)$.
\end{thm}

Note that \S \ref{Haluk}  gives evidence that `often' we actually have $H^2 (M) =H^2_{\rm bc}(M)$.

It is enough to prove this theorem for the non-connected $X(\mathfrak{n})$ rather than for $M=Y(\mathfrak{n})$.
Since $X(\mathfrak{n})$ is compact, we can compute $H^2 (X(\mathfrak{n}))$ by means of $L^2$-cohomology, and indeed
there is a Hecke-equivariant isomorphism (Matsushima's theorem, see \cite{BW}):
$$H^2 (X(\mathfrak{n})) = H^2 (\g , K_{\infty} ; L^2 ([\G]))^K$$
where $L^2 ([\G])$ is the Hilbert space of measurable functions $f$ on $\G(\Q) \backslash \G (\A)$ such that $|f|$ is square-integrable on $\G(\Q)   \backslash \G(\A)$ and
we abridge by $K=\prod_v K_v$ the compact open subgroup $K(\mathfrak{n})$. 

Since $\G$ is anisotropic the space $L^2 ([\G])$ decomposes as a direct sum of irreducible unitary representations of $\G(\A)$ with finite multiplicities (in fact equal to $1$). A representation $\sigma$ which occurs in this way is called an {\it automorphic representation} of $\G$; it is factorizable as a restricted tensor product of   admissible representations of $\G(\Q_p)$ (or more precisely, 
the unitary completions of these admissible representations).  In particular, 
$$\sigma = \sigma_{\infty} \otimes \sigma_f$$ where $\sigma_{\infty}$ is a unitary representation of $\G(\R)$ and $\sigma_f$ is a representation of $\G(\A_f)$.  
In the following we let $\mathcal{A}$ be the set of  all irreducible automorphic representations $(\sigma , V_{\sigma})$ of $\G (\A)$, realized on the subspace $V_{\sigma} \subset L^2 ([\G])$. 

\subsection{Representations with cohomology} \label{CR} Let $\g = \mathfrak{sl}_2 (\C)$ be the Lie algebra of the real Lie group $\G (\R)=\PGL_2 (\C)$. There exists a unique non-trivial irreducible $(\g, K_{\infty})$-module $(\pi , V_{\pi})$ such that $H^{\bullet} (\g , K_{\infty} ; V_{\pi}) \neq \{ 0 \}$. Furthermore:
$$H^q (\g , K_{\infty} ; V_{\pi} ) = \left\{ 
\begin{array}{ll}
0 & \mbox{ if } q\neq 1, \ 2 , \\
\C & \mbox{ if } q= 1, \ 2.
\end{array} \right. $$
If we let $\p = \mathfrak{sl}_2 (\C)/ \mathfrak{su}_2 $, the compact group $K_{\infty}$ acts by conjugation on $\wedge^q \p$; this yields an irreducible representation of
$K_{\infty}$.  There is a natural isomorphism
$$H^q (\g , K_{\infty} ; V_{\pi} ) \simeq \mathrm{Hom}_{K_{\infty}} (\wedge^q \p , V_{\pi}).$$

We denote by $\mathcal{C}$ the subset of $\mathcal{A}$ which consists of automorphic representation $\sigma = \sigma_{\infty} \otimes \sigma_f$ of $\G (\A)$ such that $\sigma_{\infty} \cong \pi$
(where, by a slight use of notation, we use $\pi$ also to denote the unitary completion of the $(\mathfrak{g}, K_{\infty})$-module described above). 

\subsection{} Let $\mathcal{H}_K$ be the Hecke algebra of finite $\Q$-linear combinations of $K$-double cosets in $\G (\A_f)$.  It is generated by the Hecke operators
$\TT (g) := K g K$. If $\sigma$ is a representation of 
$\G (\A )$, then $\mathcal{H}_K$ acts on the space of $K$-fixed vectors of $\sigma$.  On the other hand,  \S \ref{sec:Hecke} gives an action of $\mathcal{H}_K$ on $H^*(X(K))$. 

To summarize the prior discussion, then, we have a $\mathcal{H}_K$-isomorphism:
\begin{equation} \label{cow} H^q (X(K)) = \bigoplus_{\sigma \in \mathcal{C}} \mathrm{Hom}_{K_{\infty}} (\wedge^q \p , V_{\sigma}^K).  \end{equation}

\subsection{Base-change classes} \label{BC-definition}
Given an automorphic representation $\sigma$ of $\G (\A)$ we let  $\mathrm{JL} (\sigma)$ be the automorphic representation of $\mathrm{Res}_{F/\Q} (\GL_2 {}_{|F})$ with trivial central character associated to $\sigma$ by the Jacquet-Langlands correspondence. 

We  say that an automorphic representation $\sigma$ of $\G$ {\it comes from base-change} if 
 $\mathrm{JL}(\sigma)$
is isomorphic to $\mathrm{BC}(\sigma_0) \otimes \chi$, where $\sigma_0$ is a cuspidal
automorphic representation of $\GL_2{}_{|\Q}$, $\chi$ is an idele class character of $F$, and $\mathrm{BC}$ denotes base change. 

 We denote by $\mathcal{A}^{\rm bc}$ the set of all such representations $(\sigma , V_{\sigma})$ and define 
$$H^2_{{\rm bc}} (X(K)) \subset H^2 (X(K))$$
as the subspace corresponding, under \eqref{cow},
to those  $\sigma \in \mathcal{C}$ that  actually belong to $\mathcal{A}^{\rm bc}$.

{\em A priori}, this defines only a complex subspace, but it is actually defined over $\Q$, 
as one sees by consideration of Hecke operators. There is then a unique
Hecke-invariant splitting  
 $$H^2(X(K), \Q) = H^2_{{\rm bc}}(X(K), \Q) \oplus H^2_{{\rm else}}(X(K) \Q),$$
 and 
 we define the base-change subspace $H_{2}^{\rm bc}$ of homology
 as the orthogonal complement of $H^2_{{\rm else}}$.  As it turns out, $H_2^{\rm bc}$ is spanned by some special cycles that we now describe.

\subsection{Special cycles}   Let $\HH \subset \G_2$ be as in
\S \ref{sec:arithmanifolds}; recall that  $\HH (\R )   \simeq \PGL_2 (\R)$. Many
notions we have defined for $\G$ make similar sense for $\HH$, we won't recall definitions but just add $H$ as a subscript to avoid confusion.
For example, we write $\mathfrak{p}_H$ for the image inside $\mathfrak{p}$ of the Lie algebra of $\HH(\R)$ (recall
that $\mathfrak{p}$ is defined as a quotient of the Lie algebra of $\G_2(\R)$.)  

Let $L = K \cap \HH (\A_{f})$, the quotient 
$$Z (L) = \HH (\Q) \backslash (\C - \R) \times \HH (\A_f) /  L = \HH (\Q) \backslash \HH (\A ) / L_{\infty}^{\circ} L \quad (L_{\infty} = \left( \HH (\R) \cap K_{\infty}\right)),$$
is a union of (compact) Shimura curves. 
Here $L_{\infty}^{\circ}$ denotes the connected component of $L_{\infty}$, and $L_{\infty}/L_{\infty}^{\circ} \simeq \pm 1$
acts on $Z(L)$.

The inclusion $\HH \hookrightarrow \G$  defines a map $Z(L) \to X(K)$. 
Note that, since we are supposing $K$ is sufficiently small (\S \ref{arithX}) both $Z(L)$ and $X(K)$ are genuine manifolds and not merely orbifolds.
 The submanifold $Z(L)$ defines a class $[Z(L)]$ in $H_2(X(K))$. 

 More generally, for every $g \in \G(\Afinite)$ we
set $L_g = g K g^{-1} \cap \HH(\A_f)$; then right multiplication by $g$ gives a map
$Z(L_g) \rightarrow X(K)$, and by pushing forward the fundamental class 
from any component we obtain  a class in $H_2(X(K))$.
The components of $Z(L_g)$   are indexed by $\adele^{\times}/(\det L_g) (\adele^{\times})^2$,
where $\det$ denotes here the reduced norm. 
Accordingly, if $\mu: \adele^{\times}/(\det L_g) (\adele^{\times})^2 \rightarrow \mathbb{Z}$
is an integer-valued function we denote by $[Z(L)]_{g, \mu}$
the associated class in $H_2(X(K))$; in other words, $[Z(L)]_{g, \mu}$ is the image of
$$ \mu \in H^0(Z(L_g), \Z) \simeq H_2(Z(L_g), \Z) \rightarrow H_2(X(K),\Z),$$
where the first map is Poincar{\'e} duality.

We let $\mathcal{Z}_K$ be the subspace of $H_2 (X(K))$ spanned by  all such $[Z(L)]_{g, \mu}$.  Note that this subspace is spanned by classes of totally geodesic immersed surfaces that we call {\it special cycles}.  

\subsection{} \label{compatibility} 
We will need a precise description of the dual pairing $\langle -  ,  - \rangle : H_2 \times H^2 \to \C$: Choose a Haar measure $dh$ on $\HH (\Q) \backslash \HH (\A)$ and fix a generator  $\nu_H$ of the   line $\left( \wedge^2 \p_H\right)$.

Now let $T   \in \mathrm{Hom}_{K_{\infty}} (\wedge^2 \p , V_{\sigma}^K)$, for some $\sigma \in \mathcal{C}$. 
By \eqref{cow} we can identify $T $ with an element of $H^2(X(K))$.   We compute 
\begin{equation} \label{pairing}
\langle [Z(L)]_{g, \mu} , T \rangle =  c \int_{ \HH (\Q) \backslash \HH (\A ) / L_{\infty} L_g} T(\nu_H ) (hg) \ \mu(\det(h)) dh,
\end{equation}
where $c$ is a nonzero constant of proportionality, depending on $g$, the choice of measure $dh$ and the choice of $\nu_H$. 

\subsection{Distinguished representations}

Let $(\sigma , V_{\sigma}) \in \mathcal{A}$. A function 
$\varphi \in V_{\sigma}$ can then be seen as a function in $L^2 ([\G])$.  Let $\chi$ be a quadratic idele class character of $\adele^{\times}/\Q^{\times}$. We define the period integral
\begin{equation} \label{Pdef} P_{\chi}(\varphi ) = \int_{\HH (\Q)   \backslash \HH (\A )} \varphi (  h)   \chi( \det \ h) dh \end{equation}
where $dh = \otimes_v dh_v$ is a Haar measure on $\HH (\Q) \backslash \HH (\A)$ as above.  Let us  say that $\sigma$ is {\it  $\chi$-distinguished} if $P_{\chi}(\varphi) \neq 0$ for some $\varphi \in \sigma$. .  We say simply 
that $\sigma$ is {\it distinguished} if it is $\chi$-distinguished for some $\chi$.  For the following, see  \cite[Theorem 4.1]{PV}  (see also the discussion above that theorem,   \S 3 of {\em loc. cit.},
and  \cite[Proposition 3.4]{AP06}):

\begin{prop} \label{P10}An automorphic representation $\sigma \in \mathcal{A}$ is   distinguished  
  if and only if  $\sigma$ comes from base change.
\end{prop}

\medskip

\begin{prop} \label{P2}
Let $\sigma \in \mathcal{C}$ be such that $\sigma_f^K \neq \{ 0 \}$. 
Then $\sigma$ is  not  distinguished if and only if the subspace  
$$\mathrm{Hom}_{K_{\infty}} (\wedge^2 \p , V_{\sigma}^K)  \subset H^2 (X(K))$$
is orthogonal to the subspace $\mathcal{Z}_K$ spanned by the cycles $[Z(L)]_{g, \mu}$ . Equivalently, cycles
$[Z(L)]_{g, \mu}$ span $H_2^{\rm bc}$. 
\end{prop}
\begin{proof} The direct implication `only if' follows   immediately from \eqref{pairing},
together with the fact that $\mu$ lies in the span of functions $h \mapsto \chi \circ \det$.

In the converse direction: Suppose that $P_{\chi}$ is not identically zero
for some $\chi$, but 
  $\mathrm{Hom}_{K_{\infty}} (\wedge^2 \p , V_{\sigma}^K) $ 
is orthogonal to $\mathcal{Z}_K$. Then  \eqref{pairing} says at least that 
$P_{\chi}(\varphi)$ vanishes
  for every $\varphi$ of the form  $ \varphi_{\infty} \otimes g \varphi_f \in \sigma_{\infty} \otimes  \sigma_f$
where $\varphi_f$ is $K$-fixed, $g \in \G(\Afinite)$ is arbitrary,  and $\varphi_{\infty}$ is the image of $\nu_H$
under a nontrivial element of $\Hom_{K_{\infty}}(\wedge^2 \mathfrak{p}, \sigma_{\infty})$. 
Since such vectors $g \varphi_f$
span all of $\sigma_f$, we see that $P_{\chi}$ vanishes on $\varphi_{\infty} \otimes \sigma_f$.

 Now factor $P_{\chi}$ on $\sigma = \sigma_{\infty} \otimes \sigma_f$ as $P_{\infty} \otimes P_f$ (this can be done by multiplicity one, cf. \S \ref{step3}). 
 It remains to show that $P_{\infty} (\varphi_{\infty}) \neq 0$.

However,  $\chi_{\infty} $  is the nontrivial quadratic
character of $\R^*$. This is because,
if $\sigma_{\infty}  \simeq \pi$ were distinguished by the trivial character $\chi_{\infty}$, then $\sigma$ -- considered
as a representation of $\GL_2(\C)$ -- would be distinguished by $\GL_2(\R)$. It is known \cite[Theorem 7]{Flicker1}
that such representations of $\GL_2(\C)$ are the {\it unstable} base-changes of representations of $\mathrm{U} (1,1)$
but $\sigma$ is not such a representation; it is the {\it stable} base-change. 
of the weight $2$ discrete
series representation.

 Now, if $P_{\infty}(\varphi_{\infty}) = 0$, 
the above argument shows that $[Z(L)]_{g,\mu}$ would be zero for {\em every} choice of $L, \chi, g_f$ as above,
and this is not so as follows e.g. from \cite{Millson}.
\end{proof}

\subsection{Outline of the proof of Theorem \ref{Tbc}}  \label{Outline}
Write $V = \vol(Y(K))$.   Fix an embedding of $\iota: \G \hookrightarrow \SL_N$ over $F$.
For $g \in \G(\Q_v)$, we denote by $\|g\|_v$ the largest $v$-adic valuation
of any entry of $\iota(g)$.  For $g = (g_v) \in \G(\adele)$ we put
$\|g\| = \prod_v \|g_v\|_v$.

Let $\sigma_j$ (for $j$ in some index set $J$)
be all the  $\sigma \in \mathcal{C}$ such that $\sigma^K \neq 0$ and such that $\sigma$ comes from base change. 
Let $R$ be the set of ramified places, i.e.
the set of places at which $K_v \subset \G(\Q_v)$ is not maximal or where $K_v \cap \HH(\Q_v) \subset \HH(\Q_v)$ is not maximal, 
and let $\Q_R = \prod_{v \in R} \Q_v$. We decompose accordingly each $\sigma=\sigma_j$ ($j \in J$) as $$\sigma = \sigma_R \otimes \sigma^R$$ where
$\sigma_R$ is a representation of $\G (\Q_R)$ and $\sigma^R$ is a representation of $\G(\adele_F^{(R)})$, the group $\G$ over
the ``adeles omitting $R$.'' 
 
The proof now proceeds in 4 steps.   After giving the outline we discuss steps 1 and step 3 in more detail (\S \ref{step1} and \S \ref{step3}). 

 Fix $j_0 \in J$ and let $\sigma_0 = \sigma_{j_0}$. 
Let $\chi_0$ be so that $\sigma_0$ is $\chi_0$-distinguished. 
Factor $P_{\chi_0}$ on $\sigma_0 = \sigma_{0,R} \otimes \sigma_0^R$:
$$P_{\chi_0} =  P_R \otimes P^R.$$

\begin{itemize}

\item[(1)]  

We will show, first of all, that there exist  ideals $\mathfrak{p}_1, \dots, \mathfrak{p}_r$ of $F$  relatively prime to $R$ (that is, they do not lie above any place in $R$) whose norms $\Norm \mathfrak{p}_i$ are all bounded by $a V^b$ (with $a,b$  constants depending only on $F$)  and constants $c_i \in \C$ such that the Hecke operator 
$\sum_i c_i \TT_{\mathfrak{p}_i}$ 
is non-zero on $V_{\sigma_0}$ and trivial on $V_{\sigma_k}$ for every $k \in J$, $k \neq j_0$.

In other words, if $\lambda_{\mathfrak{q}}(\sigma)$ is the eigenvalue by which $T_{\mathfrak{q}}$
acts on $\sigma$, we have $\sum c_i \lambda_{\mathfrak{p}_i}(\sigma_0) \neq 0$
whereas $\sum c_i \lambda_{\mathfrak{p}_i}(\sigma_k) \neq 0$ for $k \neq j_0$. 

\item[(2)] If $v_{\rm sph}^R$ denotes the non trivial spherical vector in the space of $\sigma_0^R$, we have: 
$$P^R (v_{\rm sph}^R ) \neq 0.$$ 
We will  omit the proof; for discussion of this type of result, see  \cite[Corollary 8.0.4]{Sakellaridis}. 
 
 (We have not verified that the auxiliary conditions of \cite{Sakellaridis} apply here, although the method surely does.
 In any case,  this can be verified here by direct computation:  
Because of multiplicity one, it is enough for each $v \in R$ to show that there exists a function
on $\G(\Q_v)$ such that  $f(hgk) = \chi(h) f(g)$ for $k \in K_v$ and $h \in \mathbf{H}(\Q_v)$, such that $f(1) \neq 0$, 
and such that the  Hecke eigenvalue of $f$ is the same as $\sigma_v$. Now the cosets $\mathbf{H}(\Q_v) \backslash \G(\Q_v) / K$ are parameterized by non-negative integers
and one constructs the required $f$ as a solution to a linear recurrence.)

\item[(3)]  Now let $\varphi_1, \dots, \varphi_s$
be a basis for $V_{\sigma_0}^K$. Write $\varphi_j = \varphi_{j,R} \otimes v_{\rm sph}^R$. 
We will  show that there exist  $g_1, \dots, g_s \in \G(\Q_R)$ such that
$\|g_i\| \leq c V^d$ for constants $c,d$ depending only on $F$, and 
the matrix $(P_{R}(g_k  \cdot \varphi_{j,R}))_{1 \leq j,k \leq s}$ is nonsingular. 

\item[(4)] From the two first steps we conclude that for every $j=1 , \ldots , s$ we have:
\begin{eqnarray*}  P_{\chi_0} ( \sum_i c_i \TT_{\mathfrak{p}_i}  \varphi_j) &=& (\sum_i c_i \lambda_{\mathfrak{p}_i} (\sigma_0) ) P_R (\varphi_{j, R}) P^R ( v_{\rm sph}^R),  \\ 
P_{\chi_0} ( \sum_i c_i \TT_{\mathfrak{p}_i}  \psi)& =& 0, \ \psi \in \sigma_j \neq \sigma_0. \end{eqnarray*}

where --- according to steps 1 and 2 --- the scalars $\mu_1:= \sum_i c_i \lambda_{\mathfrak{p}_i} (\sigma_0)$ and $\mu_2 := P^R ( v_{\rm sph}^R)$ are both non zero.
Since the $g_k$ belong to $\G (\Q_R)$ and the ideals $\mathfrak{p}_i$ are  relatively prime to $R$, Step 3 finally implies that the matrix 
$$\left(  P_{\chi_0} (g_k \cdot \sum_i c_i \TT_{\mathfrak{p}_i} \cdot  \varphi_j) \right)_{ 1 \leq k \leq s ,1 \leq j \leq s }= \ \mu_1 \cdot \mu_2 \cdot \left(P_R(g_k  \cdot \varphi_{j,R}) \right)_{j,k} $$
is non singular.  \end{itemize}
 
Repeating the same reasoning for each $\sigma_j$ leads to the following refinement
of Proposition \ref{P2}:

\begin{quote} $H_2^{\rm bc}$ is spanned by cycles of the form
$\TT_{\mathfrak{p}} [Z(L)]_{g, \mu}$, where both $\Norm \mathfrak{p}$
and $\|g\|$ are bounded by a polynomial in $V$.  
\end{quote}
Using trivial estimates, we see all the cycles appearing in this statement  have volume bounded by a power of $V$.
That will conclude the proof of Theorem \ref{Tbc}.

In the following sections we provide details for steps 1 and 3. 
 
\subsection{Step 1 of \S \ref{Outline}: a quantitative `multiplicity one theorem'} \label{step1}
We first deal with automorphic representations of $\GL_2 {}_{|F}$. 
Recall the definition of the {\it analytic conductor} of Iwaniec-Sarnak:

Let $\pi = \otimes_v \pi_v$ be a cuspidal automorphic representation of $\GL_2 {}_{|F}$.
For each finite place $v$ we denote by $\mathrm{Cond}_v (\pi) = q_v^{m_v}$, where $m_v$ is the smallest non-negative integer such that
$\pi_v$ possesses a fixed vector under the subgroup of $\GL_2 (\mathfrak{o}_{F_v})$ consisting of matrices whose bottom row is congruent 
to $(0,0 , \ldots , 0, 1)$ modulo $\varpi_v^m$. Here $\varpi_v \in F_v$ is a uniformizer.  
For the infinite place $v$, let $\mu_{j,v} \in \C$ satisfy $L(s, \pi_v) = \prod (2\pi)^{-s-\mu_{j,v}} \Gamma (s+\mu_{j,v})$, and put 
$\mathrm{Cond}_v (\pi) = \prod (2+|\mu_{j,v} |)^2$. We then put $\mathrm{Cond} (\pi) = \prod_v \mathrm{Cond}_v (\pi)$ (this is within a 
constant factor of the Iwaniec-Sarnak definition). 

\begin{lem} \label{D} (Linear independence of Hecke eigenvalues) 
Given automorphic representations $\pi_1, \dots, \pi_r$ of $\GL_2 {}_{|F}$, all of which have analytic conductor at most $X$;
let $\mathcal{Q}$ be a set of prime ideals of $F$ of cardinality $\leq B \log X$ containing all ramified primes for the $\pi_i$; 
let $\{ \mathfrak{q}_j : j=1 , \ldots , s \}$ be the set of all ideals of $F$ of norm $<Y$  that are relatively prime to $\mathcal{Q}$.  Then the $r \times s$ matrix of Hecke eigenvalues
$$M_{ij}=(\lambda_{\mathfrak{q_j}}(\pi_i) )_{i, j} \quad (j=1, \ldots , s , \ i=1 , \ldots , r)$$ 
has rank $r$ so long as $Y \geq (r X)^{A}$, where $A$ is a constant depending only on $B$ and the field $F$. 
\end{lem}
Before the proof, we show how this gives Step 1:
The Jacquet-Langlands correspondence associates to any automorphic representation $\sigma_j$ as  in \S \ref{Outline},
 a cuspidal 
automorphic representation $\pi_j = \mathrm{JL} (\sigma_j)$ of $\GL_2 {}_{|F}$ with the same Hecke eigenvalues. 
Since $\sigma_f^{K(\mathfrak{n})} \neq \{0 \}$ then $\mathrm{Cond} ( \pi_j) \ll  \Norm(\mathfrak{n})^4$. (We do not know a reference
for this bound, but that such a polynomial bound exists can be  readily derived by reducing to the supercuspidal case
and using the relationship between depth and conductor; see \cite{RL} and references therein, especially \cite{Bushnell}). 
In  particular, the conductor is bounded by  a polynomial in $\vol(Y(\mathfrak{n}))$.  

To obtain Step 1, then, we 
apply the Lemma with $\mathcal{Q} = R$, the set of ``bad'' places -- i.e. the places
that are ramified for $D$ together with primes dividing $\mathfrak{n}$. 
Note that the number of primes dividing
$\mathfrak{n}$ is $\leq \log_2(\Norm \ \mathfrak{n})$; the desired result follows, since
(in the setting of Step 1) the integer $r$ is bounded by $\dim H^2(Y(\mathfrak{n}))$,
and thus by a linear function in  $\vol Y(\mathfrak{n})$.

\begin{proof} 
This is a certain strengthening of multiplicity one and will be deduced from the quantitative multiplicity one estimate of Brumley \cite{Brumley}.
(See also \cite{JPS2, Moreno} for earlier results in the same vein.)

Consider, instead of the matrix $M$, the smoothed matrix
$N$ wherein we multiply the matrix entry $M_{ij}$
by $h(\mathrm{Norm}(\mathfrak{q}_j/Y))$,
where $h$ is a smooth real-valued bump function on the positive reals
such that $h(x) = 0$ when $x >1$ and $h$ is positive for $x<1$.   Clearly the rank of $M$ and the rank of $N$ are the same. 
 
It is enough to show that the square  ($r \times r$) Hermitian matrix 
$$ N \cdot {}^t \overline{N}$$
is of full rank $r$. 
Its $(i,j)$  entry is equal to
$$ \sum  N_{ik} \overline{N}_{jk} =  \sum_{\mathfrak{q} } \lambda_{\mathfrak{q}}(\pi_i) \overline{ \lambda_{\mathfrak{q}}(\pi_j) }
h(\mathfrak{q}/Y)^2,$$
where the sum extends over the set of $\mathfrak{q}$ with norm $< Y$ and  prime to $\mathcal{Q}$. 

This is very close to \cite[page 1471, equation (23)]{Brumley}, with a minor wrinkle:
{\em loc. cit.} discusses the corresponding sum but with  $ \lambda_{\mathfrak{q}}(\pi_i) \overline{ \lambda_{\mathfrak{q}}(\pi_j)} $
replaced by $\lambda_{\mathfrak{q}}(\pi_i \times  \overline{\pi_j})$. 
But the proof of \cite{Brumley} applies word for word here, using  the equality
\begin{equation} \label{Brumz}\sum_{\mathfrak{q}}  \frac{ \lambda_{\mathfrak{q}}(\pi_i) \overline{ \lambda_{\mathfrak{q}}(\pi_j)} }{\Norm(\mathfrak{q})^s} = \frac{L^{\mathcal{Q}}(\pi_i \times\overline{\pi_j}, s)}{L^{\mathcal{Q}}( \omega_i \overline{\omega_j},2s)}.\end{equation} 
where  $\omega_i$ is the central character of $\pi_i$, and the superscript $\mathcal{Q}$ means we take the finite $L$-function and omit all factors at the set $\mathcal{Q}$. 
It leads to the corresponding bound: 
$$ \sum_{\mathfrak{q} } \lambda_{\mathfrak{q}}(\pi_i) \overline{ \lambda_{\mathfrak{q}}(\pi_j) }
h(\mathfrak{q}/Y)^2 =  \delta_{ij} Y  \cdot R_i+ O(Y^{1-\theta}   X^{B'}).$$
Here $R_i$ is a residue of the $L$-function on the right of \eqref{Brumz}, 
$\theta$ is a positive real number (one can take $\theta=1/2$) and $B'$ is a 
constant that depends only on the constant $B$ and the field $F$.   It moreover follows from \cite[equation (21)]{Brumley} that $R_i$ is bounded below by $X^{-C}$ for some
absolute (positive) constant $C$. 
 
Now the proof follows from `diagonal dominance': 
Given a square hermitian matrix $S=(S_{ij})$ such that, for every $\alpha$, 
\begin{equation} \label{DD}
S_{\alpha \alpha} > \sum_{j \neq \alpha} |S_{\alpha j}|
\end{equation}
then $S$ is nonsingular, by an elementary argument.

Now one may choose $A$, depending only on $B$ and $F$,  so that \eqref{DD} holds as long as $Y \geq (rX)^A$.
\end{proof}


\subsection{Step 3 of \S \ref{Outline}} \label{step3} 

Let $(\sigma, V^{\sigma})  \in \mathcal{C}$ and $\chi$ be such that the functional $P_{\chi}$ is not identically vanishing on $\sigma$.  
For $p$ a prime of $\Q$, let $H_p = \mathbf{H}(\Q_p)$ and $G_p = \G( \Q_p)$.    

The multiplicity one theorem shows that the functional $P_{\chi}$ factorizes over places:
%

\begin{lem}
For any irreducible $G_p$-module $\sigma_p$ we have:
$$\dim \mathrm{Hom}_{(H_p,\chi_p)} (\sigma_p  , \C) \leq 1.$$
\end{lem}
\begin{proof}  If $p$ is split in $F$ the result is easy.   If $D_p$ is split this amounts to \cite[Prop. 11]{Flicker} or \cite[Theorem A]{DP}
note that by twisting  one reduces to the case of $\chi_p=1$, at the cost of allowing $\sigma_p$ to have a central character,
so one can indeed apply Prop. 11.  In the nonsplit case, there does not appear to be a convenient reference:
one can also reduce to the results of \cite{Aizenbud} using an exceptional isomorphism, and see also \cite[Theorem B]{DP} for a closely related result;    \end{proof}
 
 For simplicity in what follows, we suppose that actually 
$\chi_p$ is trivial; the general case is a twisted case of what follows. 
So let $P_p$ be a nonzero $H_p$-invariant functional on $\sigma_p$.  Denote by $V_p$ the index
of $K_p$ inside a maximal compact subgroup.  
We will now sketch a proof of the following result, which implies step (3):

\begin{quote}
If $v_1, \dots, v_r$ form a basis for $\sigma_p^{K_p}$, 
then there exist $g_i \in G_p$ with $\| g_i^{-1} \| \leq  c V_p^d$ --
where $c,d$ are constants, depending only on the embedding $\iota$ used in the definition of $\|g\|$ -- 
such that the matrix $P_p(g_i^{-1} v_j)$ is nonsingular.
\end{quote}

Consider the functions $f_j$ on $X = G_p/H_p$
defined by the rule $g \mapsto P_p(g^{-1} v_j)$. We will show that,
when restricted to the compact set
$$  \Omega = \{ g H: \|g^{-1}\| \leq  c \cdot V_p^d \}$$
 the functions $f_j$ are linearly independent. 

Suppose to the contrary, i.e. there exists
$a_1, \dots, a_r$ not all $0$ such that $\sum a_j f_j$ is zero on $\Omega$. 
However, the asymptotics of $\sum a_j f_j$ can be
computed by the theory of asymptotics on spherical varieties or even symmetric varieties (see \cite{Lagier, KT} or \cite{SV});
this theory of asymptotics shows that if $\sum a_j f_j$ vanish identically on a sufficiently large
compact set, it must in fact identically vanish everywhere, contradiction. All that is needed
is to give a sufficiently effective version of this asymptotic theory, which we sketch:

 The wavefront lemma  (\cite[Proposition 3.2]{BO} or \cite[Corollary 5.3.2]{SV})  shows that there is a set $F \subset G$ such that $HF = G$ and 
$P_p(g v_j)$ coincides  for $g \in F$
with a {\em usual} matrix coefficient $\langle g u, v_j \rangle$,
where $u$ is a vector obtained by `smoothing' $P_p$. 
The desired asymptotics  for $\sum a_j f_j$ then follow from
known asymptotics of matrix coefficients, see e.g. \cite{Casselman}; but what is needed is an explicit
control on when matrix coefficients follow their asymptotic expansion.  For supercuspidal
representations of $\GL_n$  a sufficiently strong bound has been given by Finis, Lapid and M{\"u}ller: \cite[Corollary 2]{FLM}.
In  our case of $\GL_2$ the remaining possibilities of principal series  (and their subrepresentations) can be verified by direct computation. (An alternate approach that treats the two together is to compute in the Kirillov model, using the local functional
equation to control support near $0$).

\section{The noncompact case} \label{Noncompact}

 \subsection{The main result} \label{assump}
 In this section $\G = \G_1$. 
 
  If $M$ is a noncompact manifold we define, as usual, 
$H^i_{!}$ to be the image of compactly supported cohomology $H^i_c$
inside cohomology $H^i$; and $H_{i,  !}$
to be the image of usual homology $H_i$ inside Borel--Moore homology $H_{i, \bm}$. 
 All these definitions make sense with any coefficients, in particular, either integral or complex.  If we do not specify the coefficients
 we will understand them to be $\C$. 
 
 We now suppose that
 \begin{itemize}
 \item[(i)] $ K = K_0(\mathfrak{n})$
 where $\mathfrak{n}$ is a squarefree ideal, i.e.
 $K = \prod K_v$ where 
 $$K_v= \{ \left( \begin{array}{cc} a & b \\ c & d \end{array}\right) \in \PGL_2(\OO_v):  c \in \mathfrak{n} \OO_v \}.$$
   \item[(ii)] The corresponding (possibly disconnected) symmetric space $X_0(\mathfrak{n}) = X(K)$ 
 satisfies 
 $ \dim H^1_{!}(  X_0(\mathfrak{n}) , \C) = 1$;
 let $\pi$ be  the associated automorphic representation (i.e. the unique representation 
 whose Hecke eigenvalues coincide with those of a class in this $H^1_!$). 
 As before we let $Y_0(\mathfrak{n})$ be the identity component of $X_0(\mathfrak{n})$. 
  
 \item[(iii)]$\pi$ is associated to an elliptic curve $E$  of conductor $\mathfrak{n}$ over $F$, which we moreover assume
 to not have complex multiplication.   \footnote{Over $\Q$ condition (ii)
 automatically means that $\pi$ must be associated to an elliptic curve. Over $F$,
$\pi$ is still  conjecturally associated to a rank $2$ motive over $F$ with Hodge numbers
 $(0,1), (1,0)$ and coefficient field equal to $\Q$.  Such a motive
 might arise from an abelian variety $A/F$ admitting a quaternion algebra of endomorphisms.
 We anticipate that the same method would work in this case also.}
 \end{itemize}

 Under these assumptions our main result is:
 
 \begin{thm} \label{alltogether}
 There exists
 a $L^2$ harmonic  $1$-form $\omega$ representing a nonzero class in $H^1_!(Y_0(\mathfrak{n}), \C)$, 
 with integral periods (i.e. $\int_{\gamma} \omega \in \Z$ for
 every $\gamma \in H_1$) and moreover 
 \begin{equation} \label{Endgoal} \langle \omega, \omega \rangle  \ll   A(\mathrm{Norm} \ \mathfrak{n})^{B} \end{equation}  
 for some   constants $A$ and $B$ depending only on $F$. 
 \end{thm}

By methods similar to \S \ref{GHT} this proves Conjecture \ref{conj} in this case, i.e. 
 \begin{quote} 
If $Y_0(\mathfrak{n})$ is as above, 
 there exist immersed compact surfaces  $S_i$ of genus $\ll   \vol(Y_0(\mathfrak{n}))^C$ 
such that the images of $[S_i]$ span $H_{2} (Y_0(\mathfrak{n}), \Z)$. 
\end{quote} 

Passing $\omega$ through  Poincar{\'e}-Lefschetz duality 
$$H^1_{!} (Y_0(\mathfrak{n})) \simeq \mathrm{Im} (  H_2 (Y_0 (\mathfrak{n} )) \to H_{2, \bm} (Y_0 (\mathfrak{n} ))$$ 
one obtains a generator for the image of $H_2 (Y_0 (\mathfrak{n}) )$ in $H_{2, \bm} (Y_0 (\mathfrak{n} )) \cong H_2 (Y_0 (\mathfrak{n})_{\rm tr} , \partial Y_0 (\mathfrak{n})_{\rm tr} )$.
One can represent this generator as in \S \ref{GHT}, and we just outline the process: 
%
%

Fix a triangulation of $Y_0 (1)_{\rm tr}$ such that the boundary is a full subcomplex; it then follows that the boundary is a deformation retract of the subcomplex which consists of all simplices that intersect the boundary. We will also assume that all the edges of the dual cell subdivision have length $\leq 1$. We
finally lift this triangulation to a triangulation $K$ of $Y_0 (\mathfrak{n})_{\rm tr}$, denote by $\partial K$ the full subcomplex corresponding to the (tori) boundary components and let $K'$ denote the subcomplex of the first barycentric subdivision of $K$ consisting of all simplices that are disjoint from $\partial K$. Then the $2$-cycle
\begin{equation} \label{Zdef} Z := \sum_e \left( \int_e \omega \right) e^* \in 
C_2 (K , \partial K , \mathbb{R}).\end{equation} 
represents the image of the class $[\omega]$ under the Poincar\'e-Lefschetz duality.
 
 Thus, as in in \S \ref{GHT}, we have 
\begin{multline} \label{eqnGT}
\mathrm{inf} \{ \sum |n_k| \; \big| \; [ \sum n_k \sigma_k ] = [Z] \mbox{ where } \sum n_k \sigma_k \mbox{ is a singular chain in $C_2(
Y_0(\mathfrak{n})_{\rm tr}, \partial Y_0(\mathfrak{n})_{\rm tr}
)$
 }  \}
 \\ \ll   \Norm \mathfrak{n}^{\mathrm{A}}. \end{multline}

Now Gabai's theorem --- used in \S \ref{GHT} --- holds for $H_2 (Y_0 (\mathfrak{n})_{\rm tr} , \partial Y_0 (\mathfrak{n})_{\rm tr} )$: the $2$-cycle $Z$ is homologous 
into a (maybe disconnected) embedded surface 
$$(S, \partial S) \subset (Y_0 (\mathfrak{n})_{\rm tr}, \partial Y_0 (\mathfrak{n})_{\rm tr})$$ 
such that the LHS of \eqref{eqnGT} is $ \sum_{\chi(S_i) <0} -2 \chi(S_i)$, the sum being taken over components $S_i$ of $S$.  
Since $\partial Y_0 (\mathfrak{n})_{\rm tr}$ is incompressible and $Y_0 (\mathfrak{n})_{\rm tr}$ is atoroidal and aspherical 
we may furthermore assume that all components $S_i$ have negative Euler characteristic.

Note that the surface $S$ could a priori have boundary, but since $[S] = [Z]$ belongs to the image of 
$H_2 (Y_0 (\mathfrak{n})_{\rm tr} )$ in $H_2 (Y_0 (\mathfrak{n})_{\rm tr} , \partial Y_0 (\mathfrak{n})_{\rm tr} )$, the image of $[S]$ in $H_1 (\partial Y_0 (\mathfrak{n})_{\rm tr})$ 
by the boundary operator in the long exact sequence associated to the pair $(Y_0 (\mathfrak{n})_{\rm tr} , \partial Y_0 (\mathfrak{n})_{\rm tr})$ is trivial. 

 We can close $S$  using discs or annuli on the boundary tori, because $\partial S$ intersects each boundary torus
 in a union of simple closed curves $\gamma_j$. One first closes each $\gamma_j$ which is null-homotopic by a disc;
 and the remaining $\gamma_j$ must be be parallel and all define, up to sign, the same primitive class in homology; we can 
 close them in pairs by annuli.

 Let $f$ be the total number of discs adjoined
 when closing the boundary curves.  
The closing process has only increased the total Euler characteristic of $S$ by $f$, so we arrive now at a closed surface $S'$
 with Euler characteristic
 $$\chi(S') = \chi(S) + f  = \sum_{\chi(S_i) < 0} \chi(S_i) +  f.$$
 Finally, we may remove from $S'$ all components that are either tori or spheres,
 because both cases must have trivial class inside $H_2 (Y_0 (\mathfrak{n})_{\rm tr} , \partial Y_0 (\mathfrak{n})_{\rm tr} )$.
 Removing the tori components does not change the Euler characteristic but removing
 the sphere components decreases it and therefore increase the complexity. This is the last issue we have to deal with.

Each component $S_i'$ of $S'$ that is a sphere meets $S$ along spheres with $\geq 3$ boundary components. 
So each such $S_i'$ corresponds to a component of $S_i^*$ of $S$
with $\chi(S_i^*)  \leq -1$, and distinct $i$'s give rise to distinct components. So 
$$ \sum_{\chi(S_i ')=2} \chi(S_i') \leq  \sum_{\chi(S_i) <0} -2\chi(S_i).$$

%

%

%

%


%

Therefore, the total Euler characteristic of all sphere components of $S'$
is at most $\sum_{\chi(S_i) < 0} -2\chi(S_i)$, and removing these and tori
gives a closed surface $S''$ with Euler characteristic 
$$ \chi(S')  \geq \sum_{\chi(S_i) < 0} 3 \chi(S_i) +f \geq \sum_{\chi(S_i) < 0} 3 \chi(S_i)$$
where $S''$ still represents $Z \in H_2(  Y_0 (\mathfrak{n})_{\rm tr} , \partial Y_0 (\mathfrak{n})_{\rm tr} )$.
   
   This bounds the complexity of the ($1$-dimensional) image of $H_2 (Y_0 (\mathfrak{n}) )$ in $H_{2, \bm} (Y_0 (\mathfrak{n} ))$.
Finally, since the homology classes of the cusps are represented by surfaces of genus $1$, the Conjecture follows.

 \subsection{Modular symbols} \label{ss:MS} 
 We henceforth suppose we are in the situation of \S \ref{sec:arithmanifolds} with $\G = \G_1 (=\mathrm{Res}_{F/\Q} \PGL_2)$;
in what follows, we will usually think of $\G$ as $\PGL_2$ over $F$, rather than the scalar-restricted group to $\Q$. 
 
 Let $\alpha, \beta \in \mathbf{P}^1(F)$ and  $g_f \in \G(\adele_{f})/K_f$. Then the geodesic from $\alpha$ to $\beta$
 (considered as elements of $\mathbf{P}^1(\C)$, the boundary of $\mathbf{H}^3$),
 translated by $g_f$, 
 defines a class in 
 $H_{1,\bm}(X(K))$ that we denote by $\langle \alpha, \beta ; g_f \rangle$. 
 Evidently these satisfy the relation
 $$\langle \alpha, \beta ; g_f \rangle + \langle \beta, \gamma;  g_f \rangle + \langle \gamma, \alpha ; g_f \rangle = 0$$
 the left-hand side being the (translate by $g_f$ of the) boundary of the Borel--Moore chain defined
 by the ideal triangle  with vertices at $\alpha, \beta, \gamma$. Note that $\langle \alpha, \beta; g_f \rangle = \langle \gamma \alpha, \gamma \beta; \gamma g_f \rangle$
 for $\gamma \in \PGL_2(F)$. 
  
\medskip

For a finite place $v$ of $F$, let $F_v, \OO_v, q_v$ denote the completion of $F$ at $v$, the ring of integers of $F_v$ 
and the cardinality of the residue field of $F_v$, respectively. By the {\em valuation} at $v$  of the triple $\langle \alpha, \beta; g_f \rangle$ 
we shall mean the distance between:
\begin{itemize}
\item[-] 
the geodesic from $\alpha_v , \beta_v \in \mathbf{P}^1(F_v)$ inside
 the Bruhat-Tits tree of $\G(F_v)$, and
 \item[-] the point in that tree
 defined by $g_f \OO_v^2$.
 \end{itemize}
 i.e., the minimum distance
 between a vertex on this geodesic and the vertex whose stabilizer is
 $\Ad(g_f) \PGL_2(\OO_v)$.

  Let $n_v$ be the valuation of the symbol $\langle \alpha, \beta; g_f \rangle$ at $v$. We define 
the   {\em conductor} of the symbol to be $\mathfrak{f} = \prod_{v} \mathfrak{q}_v^{n_v}$, 
 where $\mathfrak{q}_v$ is the prime ideal associated to the place $v$; and 
the  {\em denominator} of the symbol $\langle \alpha, \beta ; g_f \rangle$ is then defined 
\begin{equation} \label{denomdef} \denom(\langle \alpha, \beta; g_f \rangle) = |(\OO/\mathfrak{f})^{\times}| =  \prod_{v: n_v \geq 1} \left( q_v^{n_v-1} (q_v-1) \right), \end{equation}  
where $q_v$ is the norm of $\mathfrak{q}_v$. 
We sometimes write this as the Euler $\varphi$ function $\varphi(\mathfrak{f})$.

 Let $\T$ be the stabilizer of $\alpha, \beta$ in $\PGL_2$;  it is isomorphic to the multiplicative group $\T \simeq \mathbb{G}_m$
 and the isomorphism is unique up to sign. 
 Then
 $\T(\OO_v) \cap  \Ad(g_f) \PGL_2(\OO_v)  $ corresponds to the subgroup 
 $1+\mathfrak{q}_v^{n_v} \subset \mathbb{G}_m(F_v)=F_v^{\times}$ if $n_v \geq 1$, and
 otherwise to the maximal compact subgroup of $F_v^{\times}$. (For example, 
 to see the latter statement, note that $\T(\OO_v)$ fixes exactly the geodesic from $\alpha$ to $\beta$
 inside the building of $\PGL_2(F_v)$.)
 
   In particular,  any finite order character $\psi$
of $\T(\adele_F)/\T(F)\simeq \adele_F^{\times}/F^*$ that is trivial on $\T(\adele_F) \cap  \Ad(g_f) \PGL_2(\widehat{\OO}) $
has conductor dividing  $\mathfrak{f}$ and order dividing $h_F \varphi(\mathfrak{f})$, where
$$ h_F = \mbox{ order of narrow class group $C_F$ of $F$.}$$
  More generally,
if $\psi$ is trivial on  $\T(\adele_F) \cap  \Ad(g_f)  K_0(\mathfrak{n})  $, with $\mathfrak{n}$ a squarefree ideal, 
then -- by a similar argument -- the conductor of $\psi$ divides $\mathfrak{n} \mathfrak{f}$ and its order divides
$h_F \varphi(\mathfrak{n} \mathfrak{f})$, in particular, its order divides 
\begin{equation} \label{orderupperbound} h_F \varphi(\mathfrak{f}) \cdot \mathrm{Norm}(\mathfrak{n}) \cdot  \varphi(\mathfrak{n}). \end{equation}

Note that another way to present our arguments would be to use a stronger version of ``conductor''
designed so that it takes account of level structure at $\mathfrak{n}$. 
 This leads to a more elaborate version of \S \ref{denomavoid} but simplifies other parts of the argument,
 because the factors of $\mathfrak{n}$ are no longer present in \eqref{orderupperbound}. 
See \S \ref{fix} for comments on that.  

  \subsection{Denominator avoidance and its proof} \label{denomavoid} 

\begin{lem*} Fix any integer $M$.  Let $p$ be a prime number. 
 If $p >5$ (resp. $p \leq 5$) any class in $H_{1, \bm}(Y(K), \Z)$
 is represented as a sum of symbols $\langle \alpha, \beta, g_f \rangle$, 
 each of which has  conductor relatively prime to $Mp$ and denominator indivisible by $p$ (resp. divisible by at most $p^A$, for an absolute constant $A$). 
 \end{lem*}
\begin{proof} 
 This is a slight sharpening of results in \cite[\S 6.7.5]{CV}. In fact, 
there is a slight error in \cite{CV} which does not deal properly with the case
 when $g_f \notin \PGL_2(\OO_v)$; the argument below in any case fixes that error.  
 
 As in \cite{CV} the Borel--Moore homology is generated by $\langle 0, \infty; g_f \rangle$ for varied $g_f$ (in the classical case, this goes back to Manin, and the proof is the same here). 
  Set $A_p = 1$ for $p >5$ and $A_p=3$ for $p \leq 5$.

One writes  
$$ \langle 0, \infty; g_f \rangle = \langle 0, x ; g_f \rangle + \langle x, \infty; g_f \rangle$$
for a suitable $x \in \mathbf{P}^1(F)$. 

First of all, if $g_v \in \PGL_2(\OO_v)$, and the prime ideal $\mathfrak{q}_v$ associated to $v$  divides the conductor of either 
$\langle 0, x ; g_f \rangle$ or $\langle x, \infty; g_f \rangle$, then $v(x) \neq 0$. 
 
Now suppose that $v$ belongs to the set of finite set  $\mathcal{B}$ of places such that $g_v \notin \PGL_2(\OO_v)$. In the Bruhat-Tits tree of $\G(F_v)$ consider  the subtree rooted at $[g_v \OO_v^2]$ which consists of the half-geodesics that  intersect the geodesic from $0$ to $\infty$ at most in the vertex $g_v \OO_v^2$.  Its 
boundary at infinity defines an open subset $S_v \subset \mathbf{P}^1(F_v)$,
and the conductors of both $\langle 0, x; g_f \rangle$ and $\langle x, \infty; g_f \rangle$
are prime to $\mathfrak{q}_v$ if $x$ belongs to this subset. 

Being open, $S_v$ contains a subset $S_v'$ of the form
\begin{equation} \label{Svpdef} S_ v' = \varpi_v^{n_v} \beta_v (1+ \varpi_v^{m_v} \OO_v).\end{equation} 
where $n_v $ is an integer, $m_v$ is an integer $\geq 1$, $\varpi_v$ is a uniformizer, and
$\beta_v \in \OO_v^{\times}$.  Write $n_v^+ = \max(n_v,0)$ and $n_v^- = \max(-n_v, 0)$ and set 
$$ \mathfrak{n}_0 =   \prod_{v \in \mathcal{B}} \mathfrak{q}_v^{m_v}, \ \ \mathfrak{a}_1 = \prod_{v \in \mathcal{B}} \mathfrak{q}_v^{n_v^+}, \ \ \ 
\mathfrak{a}_2 =  \prod_{v \in \mathcal{B}} \mathfrak{q}_v^{n_v^-}.$$

We say a prime ideal $\mathfrak{p}$ is {\it good} if it is prime to $Mp$,
 its norm is not congruent to $1$ modulo $p^{A_p}$, and it does not lie in the set $\mathcal{B}$.

%

Now we claim that we may always find $x=\frac{a_1 b_1}{a_2 b_2}$ with the following properties:

\begin{itemize}
\item[(i)]   $a_1, a_2$ have the prime factorization 
$$(a_i)=\mathfrak{a}_i \cdot \mathfrak{a}_i',$$ 
where the $\mathfrak{a}_i'$ are
good prime ideals. In particular, $v(\frac{a_1}{a_2}) = n_v$ for every $v \in \mathcal{B}$. 

 \item[(ii)] \begin{equation} \label{fngs}  \frac{b_1}{b_2} \in  \left( \frac{a_2}{a_1} \varpi_v^{n_v} \right) \beta_v  \left(1+ \varpi_v^{m_v}  \OO_v\right)  \end{equation}
 for every $v \in \mathcal{B}$ (note that this forces $x \in S_v'$ for every $v \in \mathcal{B}$).  
 
\item[(iii)]   $b_1, b_2$ generate principal good  prime ideals $\mathfrak{b}_1, \mathfrak{b}_2$;

\end{itemize} 
Given such $a_i, b_i$ we are done:
Because of \eqref{fngs} and \eqref{Svpdef}, the conductor of $\langle 0, x; g_f \rangle$ and $\langle x, \infty; g_f \rangle$
is not divisible by $\mathfrak{q}_v$ if $v \in \mathcal{B}$. Otherwise, if $v \notin \mathcal{B}$, 
then $g_v \in \PGL_2(\OO_v)$.  In that case, $\mathfrak{q}_v$
divides the conductor of either symbol only when $v(x) \neq 0$. In other words,
the only primes dividing the conductor will be 
primes in the set $\{\mathfrak{a}_1', \mathfrak{a}_2', \mathfrak{b}_1, \mathfrak{b}_2\}$.
Any prime $\mathfrak{q}$ in this set is prime to $Mp$,
so that the conductor is prime to $Mp$. Also, for any prime $\mathfrak{q}$
in this set,  $\Norm \mathfrak{q} - 1$
 is not divisible by $p^{A_p}$. Thus the denominator of either symbol
 is divisible at most by $p^{2(A_p-1)}$.


We first find $a_1, a_2$ to satisfy  (i).  We  then find $b_1, b_2$ to satisfy (ii), (iii).
  
 For (i), we apply the Chebotarev density theorem to the homomorphism $\mathrm{Gal}(\bar{F} /F) \rightarrow C_F \times (\Z/p^{A_p} \Z)$
 arising from the Hilbert class field (for the $C_F =\mbox{ class group}$ factor) and from the extension
 $F(\mu_{p^{A_p}})\supset F$  (for the $(\Z/p^{A_p} \Z)^{\times}$).  
 Now the kernel of $\Gal \rightarrow C_F$ does not project trivially to the second factor; 
 considering inertia shows that the image has size at least $\frac{p^{A_p-1}(p-1)}{2}  > 1$. 
  The Chebotarev density theorem now shows that there are infinitely many prime ideals $\mathfrak{p}$
  whose image in $C_F$ is the same class as $\mathfrak{a}_1^{-1}$ (or $\mathfrak{a}_2^{-1}$), and whose image
  in $(\Z/p^{A_p} \Z)^\times$ is nontrivial. Now take $a_1$ to be a generator for the principal ideal $\mathfrak{p} \mathfrak{a}_1$, 
  where the norm of $\mathfrak{p}$ is taken sufficiently large to guarantee that $\mathfrak{p}$ is prime to $Mp \mathcal{B}$.
Similarly for $a_2$. 
  
  Now, once we have found $a_1, a_2$, then condition (ii) amounts to the following: 
  for a certain class $\lambda \in (\OO_v/\mathfrak{n}_0)^{\times}$ defined by the right-hand side of \eqref{fngs}, we want to have
\begin{equation} \label{ii2} \frac{b_1}{b_2} \equiv \lambda \mbox{ modulo $\mathfrak{n}_0$.}
   \end{equation}  

To get \eqref{ii2} and (iii) is another application of Chebotarev:  
Write ${\frakn}_0 = {\frakn}_1 {\frakn}_2$
where ${\frakn}_1$ is prime-to-$p$ and ${\frakn}_2$
is divisible only by primes above $p$. Choose  $\bar{b}_1, \bar{b}_2  \in (\OO_F/p^{A_p} {\frakn}_2)^{\times}$ such that 
 $\bar{b}_1 \equiv \lambda \bar{b}_2 \ \mbox{mod}  {\frakn}_2$ and the norms of $\bar{b}_1, \bar{b}_2$
 (under the map $$\OO_F/p^{A_p} {\frakn}_2 \rightarrow \OO_F/p^{A_p}\stackrel{\mathrm{N}}{\rightarrow} \Z/p^{A_p} \Z$$ are not congruent to $1$. This can be done, for the
 image of the norm map
  $ (\OO_F/ p^{A_p})^{\times} \rightarrow (\Z/p^{A_p})^{\times}$
has size strictly larger than $2$. 
Now take for $b_1$ a lift of 
 $$\left( \mbox{$\lambda$ mod $\mathfrak{n}_1$}\right)  \times \bar{b}_1 \in (\OO_F/{\frakn}_1) ^{\times} \times  \left( \OO_F/p^{A_p}{\frakn}_2\right)^{\times}
 \simeq (\OO_F/{\frakn}_1 {\frakn}_2 p^{A_p})^{\times}$$
 to a generator $\pi$ of a principal prime ideal,
 and take $b_2$ similarly to be a lift $\pi'$ of $1 \times \bar{b}_2$; these lifts can be done in infinitely many ways, so certainly the prime ideals can be taken prime to $Mp \mathcal{B}$. 
 Moreover, the norm of $(b_1)$ equals the norm of $\pi$ (note this is automatically positive) and thus  is not congruent
 to $1$ modulo $p^{A_p}$. Similarly for $(b_2)$. 

\end{proof}

\subsection{} \label{fix} 
This section is not necessary for the proof. It is rather a commentary on how parts of the proof could be simplified
at the cost of expanding the prior subsection. 

A complication in the later proof arises at various points because of primes dividing $\mathfrak{n}$.
For example, we have to explicitly evaluate some local integrals  (\eqref{explicit}), 
we cannot assume that the conductors of $E, \psi$ are relatively prime in Proposition \ref{BSD}, and so on. 
We outline here a refined version of the prior Lemma that would allow us to avoid these points.

Suppose for finitely many places $V$ we specify a geodesic segment $\ell_v \ (v \in V)$ of length $1$
inside the Bruhat-Tits tree of $\PGL_2(F_v)$ containing $\OO_v^2$ (i.e., $\OO_v^2$ and one adjacent vertex). 
  Now define
the  valuation at $v \in V$ of a triple $\langle \alpha, \beta; g_f \rangle$ to be the distance between the geodesic from $\alpha_v$ to $\beta_v$
and the set of vertices of $g_v \ell_v$, i.e.
$$ \mbox{valuation at $v$} = \max(\mbox{ distance between $P$ and $[\alpha_v, \beta_v$] , for $P \in g_v \ell_v$}).$$ 
Thus the valuation is $0$ if and only if the segment $g_v \ell_v$ is contained 
in the geodesic from $\alpha_v$ to $\beta_v$.  At places outside $V$, the valuation is defined as before.

Then {\em with this refined notion  the same statement as in the Lemma still holds.}

The proof, however, is slightly more involved: In the proof above,  take $\mathcal{B}$
to consist of all places in $V$ together with all places where $g_v \notin \PGL_2(\OO_v)$. The problem
is that the set of $x$ such  that $\langle 0, x; g_f \rangle$ and $\langle x, \infty; g_f \rangle$
both have conductor indivisible by $v$, for $v \in V$, need not contain an open subset of $\mathbf{P}^1(F_v)$.  
 The problem arises when $ g_v \ell_v \subset [0,\infty]$. 
 
Call a modular symbol $\langle \alpha, \beta; g_f \rangle$  {\em good} if, for every
$v \in V$, the segment
 $g_v \ell_v$ is not contained in $[\alpha, \beta]$ for every $v \in V$. Thus, what the proof still gives is:
 \begin{equation}  \label{save} \mbox{A good modular symbol is the sum of two modular symbols with the desired properties} \end{equation}
where ``desired properties'' refers to the relevant divisibility statements for conductor and denominator.
 
 Now if a modular symbol -- without loss $\langle 0, \infty;  g_f \rangle$ is {\em not} good, then, for every $v \in V$, the set of $x \in \mathbf{P}^1(F_v)$ such that:
 \begin{itemize} 
\item $\langle 0,x; g_f \rangle$  has $v$-valuation $0$, and 
\item $\langle x, \infty; g_f \rangle$ is good
\end{itemize}
is open and nonempty. The above argument then works to show that we can write
$$\langle 0, \infty; g_f \rangle = \langle 0, x; g_f \rangle + \langle x, \infty; g_f \rangle$$
where $\langle 0, x; g_f \rangle$ has the desired divisibility properties,
and $\langle x, \infty; g_f \rangle$ is good. 
Then we are done by \eqref{save}.

\subsection{The proof of \eqref{Endgoal} assuming equivariant BSD} 
 
 Fix in what follows a symbol $\langle \alpha, \beta; g_f \rangle$
 with conductor $\mathfrak{f}$ and denominator $D=\varphi(\mathfrak{f})$; without loss of generality
 we can suppose $\alpha=0, \beta=\infty$.   We write
 $N$ for the norm of $\mathfrak{f}$. Also we can factorize $D = \prod_v D_v$ over places $v$ of $F$. 
 Finally we write $N_E = \mathrm{Norm}(\mathfrak{n})$ for the absolute conductor of the elliptic curve $E$.

\subsubsection{Normalizations}

Fix an additive character $\theta$ of $\adele_F/F$: for definiteness
we take the composition of the standard character of $\adele_{\Q}/\Q$
with the trace. Fix the measure on $\adele_F$ that is self-dual with respect to $\theta$,
and similarly on each $F_v$. 

For a function $\varphi$ on $\G(\Q) \backslash \G(\adele) \simeq
\PGL_2(F) \backslash \PGL_2(\adele_F)$ we define the Whittaker function $W_{\varphi}$
by the rule
$$W_{\varphi}(g) = \int_{x \in \adele_F} \theta(x) \varphi( \left( \begin{array}{cc} 1 & x \\ 0 & 1\end{array} \right) g)dx.$$

 Let $X = \left( \begin{array}{cc} 1/2 & 0 \\ 0 & -1/2 \end{array}\right)$
in the Lie algebra of the diagonal torus $\torus$  of $\mathfrak{pgl}_2$;
we will also think of it as
 an element in the Lie algebra of $\mathfrak{pgl}_2$. 
 
 For $y \in F$ or $F_v$, we set $a(y) = \left(  \begin{array}{cc} y & 0 \\ 0 & 1 \end{array} \right)$. 

Write $U_{\infty}= \torus(F_{\infty}) \cap K_{\infty}$. It is a maximal compact subgroup of $\torus(F_{\infty})$.

 On every $\torus(F_v)$, for $v$ finite, choose the measure $\mu_v$ which assigns
 the maximal compact subgroup mass $1$. 
 On the $1$-dimensional Lie group $\torus(F_{\infty})/U_{\infty}$ we put the measure 
 that is dual to the vector field $\underline{X}$ defined by $X \in \Lie(\torus)$, in other words, induced by a differential form dual to $\underline{X}$. 
 Finally, on $\torus(F_{\infty})$ itself, take the Haar measure which projects to the measure just defined on
 $\torus(F_{\infty})/U_{\infty}$. 

 The product measure $\mu = \prod_{v} \mu_v$  has been chosen to have the following property:  
  if $\nu$ is a $1$-form on the quotient $\torus(F) \backslash \torus(\adele_F)/ K_{\infty} U$
  for some open compact $U \subset \torus(\adele_{F,f})$, we have  
\begin{equation} \label{normalisation}
 \int_{\torus(F) \backslash \torus(\adele_F)/ U_{\infty} U} \nu =  \frac{1}{\mathrm{vol} (U)}
  \int_{\torus(F) \backslash \torus(\adele_F)} \langle \underline{X}, \nu \rangle d \mu. 
\end{equation}
Here is how to interpret the right-hand side:  $X$ defines a vector field
$\underline{X}$ on 
$ \torus(F) \backslash \torus(\adele_F)/ U_{\infty} U$;
pairing with $\nu$ gives a function, which we then pull back to $\torus(F) \backslash \torus(\adele_F)$
and integrate against the measure we have just described.   The volume $\vol(U)$ is measured
with respect to the measure $\prod \mu_v$ over finite $v$. 
Finally, the left-hand side requires an orientation to make sense; we orient
so that $\underline{X}$ is positive.

To prove \eqref{normalisation}, note  the $\nu$-integral is a sum of integrals over components. 
Each component is a quotient of $\torus(F_{\infty})/U_{\infty}$. 
 On each such components, the integral is (by definition) obtained by pushing forward the measure
 $\langle \underline{X}, \nu \rangle \mu_{\infty}$ to this quotient, and integrating. 
One also computes the right-hand side to induce the same measure on each component.

\subsubsection{Normalization of $T(X)$}  \label{TXdef}
 Let $ T \in \Hom_{K_{\infty}}(\mathfrak{g}/\mathfrak{k}, \pi)^{   K} $;
 here $\pi$ is the unique cohomological representation of level $\mathfrak{n}$
 as per our assumptions  (\S \ref{assump}) and, as in the previous section, $\mathfrak{g}$ and $\mathfrak{k}$
 are the Lie algebra of the groups $\G(F_{\infty})$ and its maximal compact subgroup. 
  Now $T$ defines a differential form on $Y_0(\mathfrak{n})$, which we call simply $\omega$.  
Put $T_X := T(X) \in \pi$; in our case it will be a factorizable vector $\bigotimes f_v$.

We normalize $T$ by requiring that the 
\begin{equation} \label{Wvdef}  \mbox{ Whittaker function $W_{T(X)}$ of $T(X)$ } = 
 \prod_v W_v,\end{equation}  where 
 $W_v$ is the new vector  of \cite{JPS} (in particular, $W_v(e)=1$ when $\theta_v$ is unramified) and  at $\infty$ we normalize by the requirement 
$$ \int_{ F_{\infty}^*} W_{\infty}(\begin{array}{cc}  y &0\\ 0 & 1 \end{array} ) dy= 1$$
where $dy$ is chosen to correspond to the measure on $\mathbf{A}(F_{\infty})$ fixed above
(a simple computation is necessary to check this is possible, since the integral might, {\em a priori}, always equal 0). 
By Rankin-Selberg and standard estimates,   we check that \begin{equation} \label{RS} \Norm(\mathfrak{n})^{-\varepsilon}   \ll   \frac{  \langle \omega, \omega \rangle_{L^2(Y_0(\mathfrak{n}))}  }{\vol(Y_0(\mathfrak{n}))}   \ll \Norm(\mathfrak{n})^{\varepsilon}. \end{equation} 
Indeed, all we need is polynomial bounds of this form, with lower bound $\Norm(\mathfrak{n})^{-A}$
and upper bound $ \Norm(\mathfrak{n})^A$ for a constant $A$ depending only on $F$; such bounds are given in \cite[eq. (10) and Theorem 5]{Brumley};
for the
case of $F=\Q$ the sharper lower bound is due to Hoffstein and Lockhart \cite{HL}, and that also contains references
for the sharper upper bound.

 \subsubsection{Adelic torus orbits versus modular symbols}
 We want to express the integral of $\omega$ over a modular  symbol $\langle 0, \infty; g_f \rangle$ 
 in terms of an adelic integral, similar to what was done in \eqref{pairing}. {\em We will assume
 that the conductor of $\langle 0, \infty; g_f \rangle$ is relatively prime to $\mathfrak{n}$.} 
 
Let $U = \torus(\adele_{F,f}) \cap g_f K g_f^{-1}$. 
Consider now the map 
 $$\torus(F) \backslash \torus(\adele_F)/U_{\infty} U \rightarrow Y(K)$$
 defined by  $t \mapsto  t g_f$. Its image can be regarded as a finite union of modular symbols
$\langle 0, \infty; t g_f \rangle$ where $t$ varies through representatives in $\torus(\adele_{F,f})$ for the group $Q  =  \torus(\adele_{F,f})/\torus(F) U. $
There is an exact sequence:
$$\mu_F/\mu_F \cap g_f K g_f^{-1} \rightarrow  \underbrace{ \torus(\widehat{\OO})/ U}_{ \mathrm{ size } \ =\mathrm{vol}(U)^{-1}} \rightarrow
 Q \rightarrow \mbox{ class group.}$$ 
where $\mu_F$ is the group of roots of unity and we regard it as a subgroup of $\torus(F) \simeq F^{\times}$ via $\mu_F \subset F^{\times}$. Call $w_F'$ the size of the group on the far left-hand side. So
 $ |Q| w_F'=       h_F  \vol(U) ^{-1}$.
%
Any character $\psi$ of $Q$ extends to a character of $[\torus ] = \torus(F) \backslash \torus(\adele_F)$ which is trivial at infinity and
it follows from \eqref{normalisation} that 
$$  \sum_{q \in Q} \psi (q) \int_{\langle 0, \infty; qg_f \rangle } \omega = \frac{ 1}{\vol(U)}  \cdot  \int_{\torus(F) \backslash \torus(\adele_F)} \psi (t) \langle (g_f)^*  \omega,  \underline{X}
 \rangle(t)   d\mu(t) $$ 
 where the right-hand side is intepreted in the same way as in \eqref{normalisation}; recall that $\omega$ has been defined in \S \ref{TXdef}.
`Fourier analysis' on the finite group $Q$ then gives:  
\begin{eqnarray}\nonumber \int_{\langle 0, \infty; g_f \rangle } \omega& =&  
\frac{1}{|Q|}\sum_{\psi \in \widehat{Q}}   \frac{ 1}{\mathrm{vol}(U)}  \cdot  \int_{[\torus]} 
  \psi (t)  \   \langle (g_f)^*  \omega,  \underline{X}
 \rangle(t)   d\mu(t) 
 \\ \nonumber &=& \frac{w_F'}{h_F}
\sum_{\psi \in \widehat{Q}}  \int_{[\torus]} 
   \psi(t)     \   \langle  (g_f)^* \omega,  \underline{X}
 \rangle(t) d\mu(t) \\ \label{buge0} 
  &=& \frac{w_F'}{h_F  } \sum_{\psi \in \widehat{Q} }  \int_{[\torus]} 
T_X(tg_f) \psi(t) d\mu
  \\   \label{buge} &=&    \frac{w_F'}{h_F  } \sum_{\psi}  L(\frac{1}{2}, \pi \times \psi)  \cdot \prod_{v}  \frac{I_v}{L_v(\frac{1}{2}, \pi \times \psi)}. \end{eqnarray}

%
      
Here $$I_v :=   \int_{y \in F_v^{\times}} W_v(a(y) g_v) \psi_v(y) dy,$$ where $g_v$ is the component at $v$ of $g_f$;
  $W_v$ is in \eqref{Wvdef}; and measures are as normalized earlier.     We have used  at step \eqref{buge} unfolding,  as in the theory of Hecke integrals \cite[\S 3.5]{Bump}.
  
   Let $S$ be the set of archimedean places, together with all places where 
  the conductor  of the symbol $\langle 0, \infty;  g_f \rangle $ is not $1$.  
  Let $S'$ be the set of finite places dividing $\mathfrak{n}$. 
  Because of our assumption, $S$ and $S'$ are disjoint. 
  
   Note that if $v \notin S \cup S'$
we have $g_v \in \torus(F_v) \cdot \PGL_2(\OO_v)$.
So  $\psi_v$ must be unramified for $I_v$ to be nonzero. 
       By choice of $W_v$ we have $I_v =  u_v \cdot L_v(\frac{1}{2}, \pi_v \times \psi_v)$
        whenever $v \notin S \cup S'$, where $u_v$ is an algebraic unit.

        For finite $v \in S$, 
the values of $W_v^{g_v}$ at least lie in $\overline{\Z}[\frac{1}{q_v}]$,
as one verifies by explicit computation. 
On  the other hand, the function $y \mapsto  W_v(a(y) g_v)$ is now constant on
each coset of $1+\mathfrak{q}_v^{n_v}$, where $n_v$ is the local conductor -- see discussion
just before \S \ref{denomavoid}; for the integral to be nonzero, then $\psi_v(y)$ must be
identically $1$ on $1+\mathfrak{q}_v^{n_v}$ and constant on each of its cosets. Each
of these cosets has measure $D_v^{-1}$, where $D_v$ is the local denominator.   

  Now  $J_v(s) := \frac{   \int_{F_v^*} W_v(a(y)g_v) \psi_v(y) |y|^s dy}{L(s+1/2, \pi_v \times \psi_v)}$ 
can be rewritten as $\int_{F_v^*} f(y) \psi_v(y) |y|^s dy$ where $f$ is a certain sum of translates of $W_v$. \footnote{Namely, 
write  $L(s+1/2, \pi_v \times \psi_v)^{-1}$ as $(1- \alpha_v \psi_v(\varpi_v) q_v^{-s}) (1-\beta_v  \psi_v(\varpi_v) q_v^{-s})$, 
where $\varpi_v$ is a uniformizer and $\alpha_v, \beta_v$ could be $0$, and then take
  $f(y)= (1- \alpha_v T) (1-\beta_v T) W_v(a(y)g_v)$, where $T$ is the operation which translates a function by $\varpi_v$.}
  But $J_v(s)$ is a polynomial in $q_v^{-s}$ (this again by the theory of Hecke integrals) and so $f_v(y)$ is compactly supported.
  Also, $\alpha_v, \beta_v \in \overline{\Z}[1/q_v]$. 
So $J_v(0)$ is actually a finite sum of elements, each lying in $\overline{\Z}[\frac{1}{q_v}] \cdot D_v^{-1}$.  This shows that 
   \begin{equation} \label{Ivdenominator}   D_v \cdot I_v \in   L_v(\frac{1}{2}, \pi_v \times \psi_v)  \overline{\Z}[\frac{1}{N }] \ (v \in S)\end{equation} 
 where $\overline{\Z}$ is the ring of algebraic integers,  $N$ is the  norm of the conductor of $\langle 0, \infty; g_f \rangle$, and $D_v$ the contribution of $v$ to the denominator of
 $\langle 0, \infty; g_f \rangle$.

Now for $v \in S'$.  Although in fact exactly the same reasoning that was just applied to $v\in S$ also applies to $v \in S'$, we will argue separately because we actually want a slightly more precise result for $v \in S'$, i.e. the set of primes dividing $\mathfrak{n}$,  with better denominator control.
Because $S \cap S' = \emptyset$ we have $g_v \in \torus(F_v) \cdot \PGL_2(\OO_v)$ for each $v \in S'$.   In particular, we may suppose that $g_v \in \PGL_2(\OO_v)$ while only modifying
  the value of $I_v$ by an algebraic unit.  By a direct computation with Steinberg representations we find
  that in fact , for $k_v \in \PGL_2(\OO_v)$ and $W_v$ the new vector for a Steinberg representation $\pi_v$, we have 
\begin{equation} \label{expl} \int W_v(a(y) k_v) \psi_v(y) d^* y \in   \frac{1}{q_v(q_v-1)} L(\frac{1}{2}, \pi_v \times \psi_v) \cdot  \overline{\Z} \end{equation}
 This is a matter of explicit computation, as we now detail:
 
 \begin{itemize}
 \item[(i)]    If $k_v $ belongs to $K_0(\mathfrak{n})$ this is clear. 
 \item[(ii)] Otherwise we can write
  $k_v =  \left( \begin{array}{cc} 1 & x \\ 0 & 1\end{array}  \right) \cdot w \cdot   k_v'$
  where $x \in \OO_v, k_v' \in K_0(\mathfrak{n})$ and $w=\left( \begin{array}{cc}  0 & 1\\ -1&0 \end{array}  \right)$. In that case
   we can rewrite the integral as $ \int W_v^{w_v}(a(y)) \theta_v(xy)  \psi_v(y) dy$. 
   The function $W_v^{w_v}(a(y))$ is supported on $v(y) \geq -1$ and its values are algebraic units when $v(y)=-1$
   (see  \cite[(11.14)]{SparseEqui}).    Moreover, it is invariant on each coset $y \OO_v^*$. 
   
   If $\psi_v$ is ramified 
    the only contribution comes from $v(y) = -1$, since on any other coset $y \OO_v^*$ with $v(y) \geq 0$
    both $W_v^{w_v}(a(y))$ and $\theta_v(xy)$ remain constant on that coset.   The integral  over $v(y)=-1$ then amounts to a Gauss sum; it belongs to
    $\frac{1}{q_v-1} \overline{\Z}$. 
    
    On the other hand, if $\psi$ is unramified, the value of the integral is -- by explicit computation 
    --  $\pm q_v^{-1} \cdot L_v(1/2, \pi_v \times \psi_v) + \frac{u}{q_v-1}$ where $u$ is an algebraic unit. 
    In fact, the term $\frac{u}{q_v-1}$ comes from $v(y)=-1$, and the remaining term $\pm q_v^{-1} \cdot L_v(1/2, \pi_v \times \psi_v) $
    comes from the contribution of $v(y) \geq 0$. 
    \end{itemize}

 We deduce that, {\em if the conductor  of $\langle 0, \infty; g_f \rangle$ is relatively prime to $\mathfrak{n}$}, 
\begin{equation} \label{Awai} \int_{\langle 0, \infty; g_f \rangle } \omega = \frac{1}{h_F   D \Norm(\mathfrak{n}) \varphi(\mathfrak{n}) } \sum_{\psi}  L(\frac{1}{2}, \pi \times \psi) \cdot a_{\psi},  \ \ a_{\psi} \in  \overline{\Z}[\frac{1}{N }]. \end{equation}  
 where $D=\prod D_v$ is the denominator of $\langle 0, \infty; g_f \rangle$,  
  every character $\psi$ that occurs on the right-hand side has conductor dividing $\mathfrak{n} \mathfrak{f}$, with $\mathfrak{f}$ the conductor of $\langle 0, \infty; g_f \rangle$;
  and $N = \mathrm{Norm}(\mathfrak{f})$.  Note also that the order of $\psi$ is bounded, as in 
  as in \eqref{orderupperbound}.
%


We will now apply equivariant BSD. We first normalize a period $\Omega_E$. Let $\Omega^1_{\mathcal{E}}$ be the $\OO_F$-submodule
of differential $1$-forms on $E$ which extend to a N{\'e}ron model; it's an $\OO_F$-module of rank $1$.  It will
be slightly more convenient for us to deal with $\Omega^1_{\mathcal{E}} \mathfrak{d}_F^{-1} := \Omega^1_{\mathcal{E}} 
\otimes \mathfrak{d}_F^{-1}$, where $\mathfrak{d}_F$ is the different of $F/\Z$.  (The reason why this is more convenient 
will become clear before \eqref{Formiso}.) We regard
it as a submodule of the $F$-vector space $\Omega^1$ of all differential $1$-forms. 
 Pick a $\Z$-basis $\xi_1,\xi_2$ for $\Omega^1_{\mathcal{E}} \mathfrak{d}_F^{-1}$ (we'll only use $\xi_2$ later). 
Now put \begin{equation} \label{Cperiod} \Omega_E = \left| \frac{1}{[\Omega^1_{\mathcal{E}} \mathfrak{d}_F^{-1}: \OO_F \xi_1]}   \int_{E(\C)} \xi_1 \wedge \overline{ \xi_1}\right|.\end{equation}  
This is independent of the choice of $\xi_1$.  For later usage, note the following: If $a = [\Omega^1_{\mathcal{E}} \mathfrak{d}_F^{-1}: \OO_F \xi_1]$,
then we have $\frac{\sqrt{-D_F/4}}{a} =\mathrm{Im}(\xi_2/\xi_1)$, by an area computation.

\begin{prop}\label{BSD} (Equivariant BSD conjecture; see \S \ref{eqBSDsec} for full discussion.)  Assume the equivariant BSD conjecture, in the formulation
\eqref{eqBSD}.
Let $E$ be a non-CM elliptic curve over the imaginary quadratic field $F$ of conductor $\mathfrak{n}_E$. Let $\Omega_E$ be
as in \eqref{Cperiod}.  
Let $\psi$ be  a character of $\adele_F^{\times}/F^{\times} F_{\infty}^{\times}$ of finite order $d$ and conductor $\mathfrak{n}_{\psi}$.
Suppose that $E$ has semistable reduction at all primes dividing $(\mathfrak{n}_E, \mathfrak{n}_{\psi})$.  Then
  we have 
 \begin{equation} \label{BSDconq} (d \Norm(\mathfrak{n}_{\psi}))^{2} \cdot  L(\frac{1}{2}, \pi \times \psi) \in  \overline{\Z}[\frac{1}{N_{\psi}' }] \cdot 
\frac{\Omega_E}{ | E(F_{\psi})_{\tors}|^2}, \ \ \   \end{equation}
where  $E(F_\psi)_{\tors}$ denotes the torsion subgroup of the points of $F$
over the abelian extension $F_{\psi}$ corresponding to $\psi$, and $N_{\psi}' = \prod_{\mathfrak{p}^2|\mathfrak{n}_{\psi}} \Norm(\mathfrak{p})$. \end{prop}
 
 Note that the elliptic curve $E$ in our context does have semistable reduction at every place, because its conductor is squarefree,
 so we can freely apply this result.

 \subsection{Proof of Theorem \ref{alltogether}} \label{finalproof}
 We now collect together what we have shown, in order to complete the proof.
 As above, $\omega$ is a differential $1$-form of level $\mathfrak{n}$ belonging to the automorphic representation $\pi$.

Fix a prime $\mathfrak{l}$ of $\overline{\Z}$ above a prime $\ell$ of $\Z$. 
 Let $\overline{\Z}_{\mathfrak{l}}$ consist of algebraic integers with valuation $\geq 0$ at $\mathfrak{l}$. 
 Also, let  $F_{\ell}$ denote the largest abelian extension of $F$ that is unramified  at all primes above $\ell$
 if $\ell$ is relatively prime to $\mathfrak{n}$. Otherwise, let $F_{\ell}$ be the largest abelian extension of $F$
 that is at worst tamely ramified at primes of $F$ above $\ell$. 
%
Begin with $\int_{\gamma} \omega$ for arbitrary $\gamma \in H_{1,BM}$, and use the Lemma 
 of \S \ref{denomavoid}  to write $\gamma$ as a sum of symbols $\langle 0, \infty; g_f \rangle$,
 where the conductor of each symbol is relatively prime to $N_E \ell$, and the denominator
 is prime to $\ell$ (or divisible by at most $\ell^A$ if $\ell \leq 5$).

 Now \eqref{Awai} writes $\int_{\gamma} \omega$ as a sum of $L$-values
 $L(\frac{1}{2}, E \times \psi)$, where the $\psi$s which occur have conductor 
 dividing $\mathfrak{n} \cdot \mathfrak{f}$, where $\mathfrak{f}$ is prime to $N_E \ell$.  
 In particular, for any prime $\mathfrak{l}$ above $l$, 
 the square $\mathfrak{l}^2$ doesn't divide the conductor of $\psi$.  
\footnote{Note that, using  \S \ref{fix}, the situation can be simplified in the following way:  Take $\ell_v$ of \S \ref{fix}
 to be the set of vertices fixed by $K_0(\mathfrak{n})$. Then \S \ref{fix}
 allows us to write $\gamma$ as a sum of symbols 
  $\langle 0, \infty; g_f \rangle$,
 where the ``refined'' conductor of each symbol is relatively prime to $N_E \ell$, and with
controlled denominator as above. Now, \eqref{Awai} writes
$\int_{\gamma} \omega$ as a sum of $L$-values
 $L(\frac{1}{2}, E \times \psi)$, where the $\psi$s which occur have conductor 
 relatively prime to $N_E \ell$ also. 
  In particular, we can assume that $\psi$ and $E$ have relatively prime conductor -- simplifying our later discussion.
   
  The reason is the following:
 For any symbol $\langle 0, \infty; g_f \rangle$, refined valuation $0$ actually means
 that $g_v K_0(\mathfrak{n})_v g_v^{-1}$  contains the maximal compact subgroup of $\mathbf{A} (F_v)$. 
 In particular -- 
  looking
 above \eqref{buge} -- if $v$ is any place of ``refined'' valuation $0$,  the vector $W(a(y) g_v)$ is actually invariant by $y \in \OO_v^*$,
and then $\psi_v$ must actually be unramified for the local integral $I_v$ to be nonzero. So, in the above reasoning,
the only $\psi$s that occur have conductor  relatively prime to $N_E \ell$, because the modular symbols
which occur had ``refined'' conductor relatively prime to $N_E\ell$.  }

   Combine  \eqref{orderupperbound}, 
 \eqref{BSDconq} and   \eqref{Awai}  to arrive at: 
$$  \int_{\gamma} \omega    \in   \frac{1} 
{(30 h_F  \Norm(\mathfrak{n}) \varphi(\mathfrak{n})  )^B} \cdot \frac{\Omega_E}{  | E(F_{\ell})_{\tors}|^2} \cdot  \overline{\Z}_{\mathfrak{l}},$$
for some absolute constant $B$.

Set $$M = \prod_{\ell} \# E(F_{\ell})[\ell^{\infty}] \mbox{ and } M' =(30 h_F \Norm(\mathfrak{n}) \varphi(\mathfrak{n})  )^B ;$$ 
$M$ is finite because $E(F^{\ab})_{\tors}$ is finite -- that is
a simple consequence of Serre's open image theorem (see e.g. \cite{S04}), using the fact that $E$ does not have CM.
Then 
$$  \int_{\gamma} \omega    \in   \frac{1} 
{M'} \cdot \frac{\Omega_E}{ M^2 } \cdot  \overline{\Z}_{\mathfrak{l}},$$
and beause this is true for all $\mathfrak{l}$ we get 
$$  \int_{\gamma} \omega    \in   \frac{1} 
{M'} \cdot \frac{\Omega_E} { M^2} \cdot  \overline{\Z}. $$


 Thus, if we set
 $$\omega' =   \Omega_E^{-1} M' M^2  \cdot \omega$$
 the form $\omega'$ has integral periods, i.e. $\int_{\gamma} \omega' \in \Z$ for all $\gamma$. 
Our desired result (Theorem \ref{alltogether}) follows from \eqref{RS} together with the bounds:
\begin{equation} \label{moo2} \Omega_E^{-1} \mbox{ and } M  \ll_F A \left( \Norm \mathfrak{n}\right)^B
 \end{equation}
for absolute constants $A, B$. 

We now explain how to check \eqref{moo2}.

For $\Omega_E$, one uses the relationship
with the Faltings height, together with Szpiro's conjecture and Frey's conjecture, see \cite[F.3.2]{HS}.
As commented in that reference, these conjectures are, up to the exact value of the constants involved, equivalent to the ABC conjecture. 
Then, up to constant factors, $\Omega_E^{-1/2}$
 coincides with the exponential of the Faltings height; conjecture \cite[F.3.2]{HS} now yields
 \eqref{moo2}, using also the result stated in \cite[Exercise F.5(c)]{HS}. 

Now let us examine  $M$. 

\begin{itemize}
\item For $\ell > 3$ and $\ell$ relatively prime to $\mathfrak{n}$, 
$E(F_{\ell})[\ell]$ must be either trivial or cyclic, because   
if $E(F_{\ell})[\ell]$ were all of $E[\ell]$, that means that inertia groups  $I_{\mathfrak{l}} \subset \Gal(\bar{F}/F)$
for any prime $\mathfrak{l}$ of $F$ above $\ell$ would act trivially on $E[\ell]$.
But this is never true,
because the determinant of this action is the mod $\ell$ cyclotomic character.
which is nontrivial.
\item  For $\ell =2,3$ or $\ell$ dividing $\mathfrak{n}$, 
 we see similarly that $E(F_{\ell})[\ell^3]$  cannot be of
 $E[\ell^3]$.  
 \end{itemize}
 Write $Q = (6 \mathrm{Norm}(\mathfrak{n}))^2$. 
 
 So  $Q \cdot E(F_{\ell})[\ell^{\infty}]$
is cyclic for every $\ell$.   Write 
 $K $ for the subgroup of $E(F^{\ab})_{\tors}$
 generated by all the $Q \cdot E(F_{\ell})[\ell^{\infty}]$. As we have just seen,
 $K$ is a cyclic subgroup of order $\geq M/Q^2$, and it is stable by the Galois group 
 of $\bar{F}/F$. Consider the isogeny $\varphi: E \rightarrow E' := E/K$. 
  Masser-W{\"u}stholz (see e.g. the main theorem of
 \cite{MW})  give an isogeny
 $\varphi': E' \rightarrow E$ in the reverse direction, whose degree is  bounded by a polynomial in
 the Faltings height of $E$. The composite isogeny $\theta = \varphi' \circ \varphi: E \rightarrow E$
 must be multiplication by an integer $r$,  because $E$ does not have CM, and also $\# K $ divides $r$
 because $\theta(K) = 0$ and $K$ is cyclic. 
From $$(\# K)  \cdot \deg \varphi' = r^2$$
 we get the desired bound: $M \leq Q^2 \deg \varphi'$.


%
%
%
%

\section{The equivariant conjecture of Birch and Swinnerton--Dyer} \label{eqBSDsec} 
 
In the previous section we used   Proposition \ref{BSD}  which says (see that section for notation):
\begin{quote}
Assume equivariant BSD. Let $E$ be an elliptic curve over the imaginary quadratic field $F$ of conductor $\mathfrak{n}_E$.
Let $\psi$ be  a character of $\adele_F^{\times}/F^{\times} F_{\infty}^{\times}$ of finite order $d$ and conductor $\mathfrak{n}_{\psi}$.
We assume that $E$ has semistable reduction at every prime dividing $(\mathfrak{n}_{\psi}, \mathfrak{n}_E)$. 
Put $N_{\psi}'=\prod_{\mathfrak{p}^2|\mathfrak{n}_{\psi}} \Norm(\mathfrak{p})$. 
 \begin{equation} \label{BSDconq2} (d \mathrm{Norm}(\mathfrak{n}_{\psi}))^2 \cdot  L(\frac{1}{2}, \pi \times \psi) \in  \overline{\Z}[\frac{1}{N_{\psi}' }] \cdot 
\frac{\Omega_E}{ | \# E(F_{\psi})_{\tors}|^2}, \ \ \   \end{equation}
where $N_E, N_{\psi}$ are the respective norms of $\mathfrak{n}_E, \mathfrak{n}_{\psi}$; and $F_{\psi}$ is the abelian extension
determined by $F$, and $\Omega_E$ is the period normalized as in \eqref{Cperiod}.     \end{quote} 

This was quoted as a consequence of the ``equivariant Birch/Swinnerton-Dyer conjecture.''  
Unfortunately, there is no standardized form of such a conjecture in the literature, to our knowledge, in the generality we need it. That is why we have written the current section \S \ref{eqBSDsec}, to spell out exactly
what we mean and how it gives rise to \eqref{BSDconq2}.  
We have chosen to directly formulate an equivariant BSD conjecture in \eqref{eqBSD} 
in a way that directly mirrors the formulation given by  Gross \cite{Gross} for CM elliptic curves.  
In principle, this should be routinely verifiable to be equivalent
to the equivariant Tamagawa number conjecture of \cite[\S 4]{ETNC}, although we did not attempt to verify the details. 
 In summary, when we say ``equivariant BSD'' in this paper, we {\em mean} the conjecture
 that is formulated in \eqref{eqBSD} below; and we anticipate, but have not verified,  that this can be verified to be compatible
 with \cite{ETNC} in a routine fashion.  
  
Here is the basic idea.  To understand $L(\frac{1}{2}, E \times \psi)$ as below
one needs to understand the $L$-function of $E$ over a certain abelian extension $F_{\psi}$, 
i.e. the $L$-function of abelian variety $\mathrm{Res}_{F_{\psi}/F} E$, but {\em equivariantly}
for the action of the Galois group $G$ of $F_{\psi}$ over $F$. One difficulty encountered is that $\Z[G]$
is not a Dedekind ring.  This issue comes up in other work on the subject \cite{Bley}. Of course, our goal is much less precise,
since we may lose arbitrary denominators at $N_{\psi}'$ and also some denominator at $d$.  In any case we deal with this by instead passing to an abelian subvariety
of $A$ which admits an action of a Dedekind quotient of $\Z[G]$.  As a simple example, if $\psi$ is a quadratic character,
one can analyze $L(\frac{1}{2}, E \times \psi)$ as the $L$-function of a quadratic twist of $E$, rather than
working with $E$ over the quadratic extension defined by $\psi$. 

In the actual derivation we will try to write formulas that are as explicit as possible. 
We write for short $\Z' = \Z[\frac{1}{  N_{\psi}'}]$
and if $M$ is a $\Z$-module we sometimes write $M'$ for $M \otimes_{\Z} \Z'$. 

%
%
%

%
%

\subsection{Basic setup }

Choose a prime $\ell$ that doesn't divide $  N_{\psi}'$
 and extend the valuation at $\ell$ to a valuation of $\overline{\Q}$. 
 We will prove  \eqref{BSDconq2} ``at $\ell$,'' i.e. verify that the $\ell$-adic valuation
 of the ratio $\frac{\mathrm{LHS}}{\mathrm{RHS}}$ behaves as predicted.

 We regard $F$ as a subfield of $\C$, i.e. we choose a fixed embedding $\iota$ of $F$ into $\C$. 
 When we write $E(\C)$ we understand it as the complex points of $E$ considered as a complex variety via $\iota$.
 
  Let $F_{\psi}$ be the abelian field extension of $F$ determined, according to class field theory, by the kernel of $\psi$.
   Note that $F_{\psi}$ is tamely ramified above $F$ at all primes above $\ell$,
 because any prime $\mathfrak{l}$ above $\ell$ divides $\mathfrak{n}_{\psi}$ with multiplicity $\leq 1$.

  Let $\mu$ be the cokernel of $\psi$ (so that $\psi$ 
gives an isomorphism of $\mu$ with  a cyclic subgroup of $\C^{\times}$); thus
 $\Gal(F_{\psi}/F) \simeq \mu$. 
  We   fix an extension of $\iota$ to an embedding $\sigma_1: F_{\psi} \rightarrow \C$;
  for $\alpha \in \mu$ we put $\sigma_{\alpha} = \sigma_1 \circ \alpha$. 
   
   \subsection{Background on cyclotomic rings} \label{Cycback}

The size of $\mu$ is $d$, i.e. the order of $\psi$; 
and let $R = \Z[\mu]$ be the group algebra of $\mu$, so that $\psi$
gives an algebra homomorphism $\psi: R \rightarrow \C$. 
For $a \in R$ we will sometimes write $a^{\psi}$ instead of $\psi(a)$;
we will also use this notation for $a \in R_{\R} := R \otimes \R$
(i.e. $a^{\psi}$ is the value, in $\C$, of the real-linear extension $\psi: R_{\R} \rightarrow \C$).

Choose a generator $\zeta$ for $\mu$. 
Let $\phi_d \in \Z[x]$ be the $d$th cyclotomc polynomial 
and $\theta_d = \frac{x^d-1}{\phi_d} \in \Z[x]$. 
Let $\Phi_d = \phi_d(\zeta)$ and $\Theta_d =\theta_d(\zeta)$
be the elements of $R$ obtained by evaluating these at $\zeta$. 
Note that $\Theta_d \Phi_d = 0$ in $R$.

Set  $$S = R/(\Phi_d).$$ Then $S$ is a Dedekind ring, isomorphic to the ring of integers
in the $d$th cyclotomic field, and the homomorphism $\psi: R \rightarrow \C$ then factors through $S$. 
 Note that by differentiating
\begin{equation}  \label{diffdiv} \mbox{ image in $S$ of  $(\Theta_d \cdot \phi_d'(\zeta))$} =  \frac{d}{dx} (x^d-1) |_{x=\zeta} = d \zeta^{-1}, \end{equation}
where this equality is in $S$. 
This shows that $d$ is divisible, in $S$,  by the product of $\Theta_d$
and $\phi_d'(\zeta)$. Note that the image of $\phi_d'(\zeta)$
in $S$ is exactly the different of $S$ over $\ZZ$.

Consider the abelian category $\Stor$ of finite $S$-modules: modules that are finite as abelian groups.
Then the rule $S/\mathfrak{n} \mapsto \mathfrak{n}$ gives an isomorphism
\begin{equation} \label{moduleinequality} K_0(\Stor) \simeq \{ \mbox{fractional ideals of $S$.} \}. \end{equation}  We use $[X]$ to denote the 
 (fractional) ideal corresponding to a torsion $S$-module $X$, and write
 $[X] \geq [Y]$ if the ideal for $X$ is divisible by the ideal for $Y$. 
 We write $[X] \geq_{\ell} [Y]$ if this holds ``at $\ell$,'' i.e.
 the valuation at any prime $\mathfrak{l}$ above $\ell$
 for $[X]$ is $\geq$ the same for $[Y]$. 
 
 If $[X]$ corresponds to a {\em principal} ideal, we will say that $X$ is ``virtually principal.'' By an abuse of notation, we may regard then $[X]$ as an element of $ S_{\Q}^*/S^*$, namely,
 a generator for that principal ideal.  Here we have written $S_{\Q}$ as an abbreviation for $(S \otimes \Q)$. 
 
Involutions: We denote by $x \mapsto x^*$ the involution of $R$ that is induced by inversion on $\mu$.
This descends to the canonical complex conjugation on the CM-field $S$. 
Later we also consider
the ``complex conjugation''   $x \mapsto \bar{x}$ on $R_{\C} := R \otimes_{\Z} \C$  arising
from the conjugation of $\C/\R$.

Given $x \in F_{\psi}$ we define $[x] \in R_{\C}$ by 
$$ [x] = \sum_{\alpha} \sigma_{\alpha}(x) \alpha^{-1}.$$
The morphism  $x \mapsto [x]$ is actually equivariant for the action
of $\OO_F[\mu]$, which acts on $F_{\psi}$ by linear extension of the $\mu$-action
and acts on $R_{\C}$ via the map $\OO_F[\mu] \rightarrow R_{\C}$
 induced from the natural embedding $\mu \rightarrow R$ and 
 the inclusion $\iota: \OO_F \hookrightarrow \C$. 
 

%
%

%

 \subsection{The abelian varieties and their  N{\'e}ron models}

We put  $$ A = \mathrm{Res}_{F_{\psi}/\mathbf{Q}} E.$$
 Then $A$ is a $2d$-dimensional abelian variety which admits an action of $\mu \simeq \Gal(F_{\psi}/F)$ and so also of the algebra $R$. 
 Consider the $2 \varphi(d)$-dimensional abelian variety
 $B$ which is given by the connected component of the kernel of $\Phi_d $ acting on $A$:
 $$B = \mathrm{ker}( \Phi_d: A \rightarrow A)^0.$$
 Then the action of $R$ on $B$ factors through $S$. 
 Also 
  $\Theta_d $ gives a surjection of abelian varieties $A \rightarrow B$. 
%

Denote by $\mathcal{E}$  the N{\'e}ron model of $E$ over $\OO_F$, 
and by $\mathcal{B}$ the N{\'e}ron model of $B$ (now over $\Z$)
and finally $\mathcal{A}$ that of $A$ (also over $\Z$).  

Denote by $\mathrm{Lie}(\mathcal{E})$ the tangent space to $\mathcal{E}$ above the identity section,
and $\Omega^1_{\mathcal{E}}$ its $\OO_F$-linear dual; these are both locally free $\OO_F$-modules of rank one. 
We use similar notation for $\mathcal{A}$ and $\mathcal{B}$; in that case, they are free $\Z$-modules
of rank $2d$ and $2\varphi(d)$ respectively. 
Thus, e.g. $\mathrm{Lie}(\mathcal{E})$ is the set of $\Spec \left( \OO_F[\varepsilon]/\varepsilon^2 \right) $-valued
points of $\mathcal{E}$ that extend the identity section $\Spec \OO_F \rightarrow \mathcal{E}$.


The connected N{\'e}ron model of $E \otimes_F F_{\psi}$ over $\OO_{F_{\psi}}$
 coincides  with the base-change of  the connected N{\'e}ron model for $\mathcal{E}$.   
Indeed, the  universal property gives a map from $\mathcal{E} \otimes_{\Z} \OO_{F_{\psi}}$
 {\em to} this N{\'e}ron model, and this is an open immersion. We check this after localizing at each prime $\mathfrak{p}$:
 \begin{itemize}
\item[-] If  a prime $\mathfrak{p}$ of $F$   doesn't divide $\mathfrak{n}_{\psi}$,
 this is the commutation of N{\'e}ron models with unramified base change \cite[Theorem 1, Chapter 7]{Neronmodels}.
 \item[-]  If  a prime $\mathfrak{p}$ of $F$   {\em does} divide $\mathfrak{n}_{\psi}$ then
 by assumption $E$ has semistable reduction at $\mathfrak{p}$, and the result is known \cite[Prop 3, Chapter 7]{Neronmodels}. 
 \end{itemize}

Now 
it is known (\cite[Proposition 4.1]{Edixhoven}) that
$\mathcal{A}$ is the restriction of scalars $\Res_{\OO_{F_{\psi}}/\Z} $ for the
N{\'e}ron model of $E$ over $\OO_{F_{\psi}}$.
  Using the description of Lie algebra noted above and the defining property of restriction of scalars we see that
$$\mathrm{Lie}(\mathcal{A}) =    \left( \mathrm{Lie}(\mathcal{E}) \otimes_{\OO_{F}} \OO_{F, \psi} \right)$$

We obtain the respective $\Omega^1$ spaces by dualizing.   
  Now the $\Z$-dual of a locally-free $\OO_F$-module $M$ is isomorphic
  to $\Hom_{\OO_F}(M, \OO_F) \otimes_{\OO_F} \mathfrak{d}_F^{-1}$, where $\mathfrak{d}_F$ denotes the different.
The pairing $x \in M, y \otimes \delta \in \Hom_{\OO_F}(M, \OO_F) \otimes_{\OO_F} \mathfrak{d}_F^{-1} \mapsto
\mathrm{trace}_{F/\Z}(  \langle x,y \rangle \delta )$ induces this isomorphism. 
Thus --    \begin{eqnarray}\nonumber  \Omega^1_{\mathcal{A}/\Z} =\Hom_{\Z}(\mathrm{Lie}(\mathcal{A}), \Z)   
=  \Hom_{\Z}( \mathrm{Lie}(\mathcal{E}) \otimes_{\OO_{F}} \OO_{F_ {\psi}} , \Z)  \\ =\nonumber
\left( \Hom_{\OO_{F_{\psi}}}( \mathrm{Lie}(\mathcal{E}) \otimes_{\OO_{F}} \OO_{F_{\psi}} , \OO_{F_{\psi}})  \otimes_{\OO_{F_{\psi}}} \mathfrak{d}_{F_{\psi}}^{-1}  \right)   
\\ =  \left( \Hom_{\OO_{F}}( \mathrm{Lie}(\mathcal{E}),  \OO_{F_{\psi}})  \otimes_{\OO_{F_{\psi}}} \mathfrak{d}_{F_{\psi}}^{-1}  \right) 
\\ \label{Formiso2} =  
    \left( \Omega^1_{\mathcal{E}}  
 \otimes_{\OO_{F}}  \mathfrak{d}_{F_{\psi}}^{-1} \right)
  \\ \label{Formiso} =  
    \left( \Omega^1_{\mathcal{E}}   \mathfrak{d}_F^{-1}
 \otimes_{\OO_{F}}  \mathfrak{d}_{F_{\psi}/F}^{-1} \right) \end{eqnarray}
 
 To be more precise,  there is a natural map $\Omega^1_{\mathcal{E}} \otimes_{\OO_F} \mathfrak{d}_{F_{\psi}}^{-1}
 \rightarrow \Omega^1_A$, and the assertion is that the image consists of $1$-forms on $A$ which  extend to the N{\'e}ron model. 
 In the last equations, $\mathfrak{d}_{F_{\psi}/F}$ denotes the relative different
 and we used transitivity of the different, which means $\mathfrak{d}_{F_{\psi}} \simeq \mathfrak{d}_{F_{\psi}/F} \otimes_{\OO_F} \mathfrak{d}_F$.

Our next order of business is to get some understanding of $\Omega^1_{\mathcal{B}}$.
There's a natural morphism $\Omega^1_{\mathcal{A}} \rightarrow \Omega^1_{\mathcal{B}}$
induced by $B \hookrightarrow A$,
and we want to put an upper bound on the size of the cokernel.  To do so
we examine the morphism $A \rightarrow B$ given by ``multiplication by $\Theta_d$.'' 
The composite $B \rightarrow A \rightarrow B$ is given by multiplication by $\Theta_d$ on $B$; that shows us that
$$ \Theta_d  \Omega^1_{\mathcal{B}} \subset \mbox{ image of $\Omega^1_{\mathcal{A}}$}.$$   
  
   Note that $\Omega^1_{\mathcal{B}}$ is a locally free $S$-module of rank $2$. (In fact, it is a free $\Z$-module
   of rank $2 \varphi(d)$, and if we tensor $\otimes_{\Z} \C$  we get   a free
   $S \otimes \C$ module of rank $2$.     From there we see that $\Omega^1_{\mathcal{B}}$
   is contained with finite index in a free $S$-module, which easily implies it is locally free.)

So in the exact sequence
\begin{equation} \label{OmSr} \Omega^1_{\mathcal{A}} \rightarrow \Omega^1_{\mathcal{B}} \rightarrow C,\end{equation}
   the cokernel $C$  has the property that $[C] \leq  2[S/\Theta_d]$;
   in particular (see \eqref{diffdiv}) $[C] \leq [S/d^2]$. 

%
%
%
%
%
%
%
%
%
%
%
%
%

    \subsection{The homology of $A(\C)$.} \label{Before} 
     Note that $$A(\C) = E(F_{\psi } \otimes_{\Q} \C) = \bigoplus_{\sigma: F_{\psi} \rightarrow \C} E^{\sigma}(\C).$$
     The set of $\sigma$s which occur is the set $\sigma_{\alpha} \ (\alpha \in \mu)$
defined earlier,     together with their conjugates $\overline{\sigma_{\alpha}} \ (\alpha \in \mu)$.

            Choose generators $\gamma_1, \gamma_2$ for $H_1(E^{\sigma_1}(\C))$
 and let $\overline{\gamma_1}, \overline{\gamma_2}$
 be their images under the antiholomorphc map $E^{\sigma_1} \rightarrow E^{\overline{\sigma_1}}$.  
  A free $R$-basis for $H_1(A(\C), \Z)$ is given by $ \gamma_1,   \gamma_2,  \overline{\gamma_1},
  \overline{\gamma_2}$.  The complex conjugation of $\C/\R$ 
  induces an antiholomorphic involution  of $A(\C)$; that involution switches $\gamma_i$ and $\overline{\gamma_i}$. 
Later we will also set  $\delta_i = \gamma_i + \overline{\gamma_i}$.  
 
Now $H_1(B(\C), \Z)$ is given by the kernel of $\Phi_d$ acting on $H_1(A(\C), \Z)$. 
Since the latter is free, as $R$-module on  $ \gamma_1,   \gamma_2,  \overline{\gamma_1},
  \overline{\gamma_2}$, it follows that $H_1(B(\C), \Z)$
  is free as $S$-module on the same generators multiplied by $\Theta_d$. 
  
  (In fact, the kernel of $\Phi_d : R \rightarrow R$ is just $R \Theta_d$, 
  and is free of rank $1$ as an $S$-module:  regard $R = \Z[x]/(x^d-1)$; if the class of $f(x)$ is killed by $\phi_d$, 
   then $(x^d-1)$ divides $f(x) \cdot \phi_d$, so that $\theta_d$ divides $f$.)

 \subsection{Torsion subgroups}
 
 Later we will need to understand the torsion subgroups of both $B(\Q)$ and $\hat{B}(\Q)$ where $\hat{B}$ is the dual abelian variety.
 
    Clearly $B(\Q)_{\tors} \subset A(\Q)_{\tors} = E(F_{\psi})_{\tors}$. 
    To bound torsion in $\hat{B}$ note that 
    we have a map $ \Theta_d : A \rightarrow B$ and thus also a dual map $\widehat{\Theta_d}: \hat{B} \rightarrow \hat{A}$.
 We compute the kernel of $\widehat{\Theta_d}$ over $\C$: it is dual, as an abelian group, to the cokernel of
 $$\Theta_d: H_1(A(\C), \Z) \rightarrow H_1(B(\C),\Z),$$
 but this is trivial, as we have seen. 
  Thus also $\hat{B}(\Q)_{\tors}$ is isomorphic to a subgroup of $\hat{A}(\Q)_{\tors} = E(F_{\psi})_{\tors}$
 ($A$ carries a principal polarization and so is isomorphic to $\hat{A}$). 

  \subsection{Integration}

By integration we get a mapping
\begin{equation} \label{Intmap} f:   \Omega^1_{\mathcal{B}}  \otimes \R  \stackrel{f}{\longrightarrow}   \left( H^1(B(\C), \R)_+ \right), \end{equation} 
where on the right hand side the subscript $+$ denotes {\em coinvariants}
of complex conjugation  considered
as an antiholomorphic involution of $B(\C)$. We can regard
the right-hand side as the $\R$-dual of $H_1(B(\C), \R)^+$, the conjugation-invariants on homology,  and then the map is $\omega \mapsto \int_{\gamma} \omega$
for $\gamma \in H_1(B(\C), \R)^+$. 
  This
is an isomorphism of free $\R$-modules, both of rank $r= 2 \varphi(d)$. 

Both sides have integral structures:  On the left-hand side $\Omega^1_{\mathcal{B}}$.
On the right-hand side we put the integral structure that is the image of $H^1(B(\C), \Z)$. 
Thus we can compute the ``period determinant'' of \eqref{Intmap}, well-defined up to sign.  That determinant is given by the volume
\begin{equation} \label{Brvo}\Lambda= \pm \int_{B(\R)^{\circ}}  | \omega_1 \wedge \dots \omega_{r}|, \end{equation} 
where $\omega_i$ is an integral basis for $\Omega^1_{\mathcal{B}}$.  

%
%
%
%
%
%
%

     The usual BSD conjecture  \cite{BSD} for $B$ says 
 \begin{equation} \label{UsualBSD} L( \frac{1}{2},  B) =   \pm \frac{\Sha_B \ \cdot \  \prod_{v} c_v(B) }{B(\Q) \cdot  \hat{B}(\Q) } 
\cdot  \Lambda
\end{equation} 
where we {\em allow ourselves to write a finite group in place of its order}, and if $B(\Q)$ is infinite
we understand the right-hand side as $0$. 
Note that every term on the right is a finite $S$-module (e.g. $c_v(B)$ is the local component group of the N{\'e}ron model, 
and $S$ acts on it too.) Also note that (the way we have set things up) the archimedean component
groups $c_{\infty}(B) \simeq B(\R)/B(\R)^{\circ}$ also counts. 

As preparation for the equivariant version, we phrase this a little differently. Suppose that we give ourselves finite index  subgroups $\mathcal{H} \subset H^1_+$
and $\mathcal{W} \subset \Omega^1_{\mathcal{B}}$ of the respective integral structures. We can
form the period determinant $\Lambda'$ with respect to $\mathcal{H}$ and $\mathcal{W}$, 
i.e. $f_*(\det \mathcal{W}) = \Lambda' \cdot \det \mathcal{H}$, where 
e.g. $\det \mathcal{H}$ denotes the element of the top exterior power
of $\mathcal{H} \otimes \R$ determined by the lattice $\mathcal{H}$. 
Since $[H^1_+: \mathcal{H}] \cdot \det(H^1_+) = \det(\mathcal{H})$
and similarly for $\mathcal{W}$, we deduce the following variant form
of BSD: 
 \begin{equation} \label{UsualBSD2} L( \frac{1}{2},  B) =  \pm  \frac{\Sha_B \ \cdot \  \prod_{v} c_v(B) }{B(\Q) \cdot  \hat{B}(\Q) } 
 \cdot \frac{[H^1_+:\mathcal{H}]}{[\Omega^1_{\mathcal{B}}: \mathcal{W}]}
\cdot  \Lambda'
\end{equation} 
%

%

    \subsection{Statement of the conjecture}
In order to make the equivariant conjecture we need to break
up the right hand side of \eqref{UsualBSD}
in a way that corresponds to the factorization $L(\frac{1}{2}, B) = \prod_{\chi} L(\frac{1}{2}, E \times \chi)$,
where the product is taken over all powers $\chi =\psi^i$ with $i \in (\Z/d)^*$. 

First of all choose integral elements $e_1,e_2 \subset H^1(B(\C), \R)_+$ so that the $S_{\R}$-module generated
by $e_1, e_2$ is free, and similarly  choose $\nu_1, \nu_2 \in \Omega^1_{\mathcal{B}}$. Then  \eqref{UsualBSD2}  says
 \begin{equation} \label{UsualBSD3} L( \frac{1}{2},  B) =   \left(  \frac{\Sha_B \ \cdot \  \prod_{v} c_v(B) }{B(\Q) \hat{B}(\Q) }
 \frac{H^1_+/S e_1 + S e_2}{\Omega^1_{\mathcal{B}} / S \nu_1 + S \nu_2  } \right)
\cdot  \Lambda'
\end{equation}  
where $\Lambda'$ is the period  determinant taken relative to the integral lattices $S e_1+Se_2$ and $S \nu_1+ S \nu_2$. 
Note that all the finite groups inside the brackets on the right-hand side are actually $S$-modules. We will next examine how to refine
each term on the right to an element of $S_{\R}^{\times}$, so that we recover \eqref{UsualBSD2} by taking norms.

Firstly let us examine $\Lambda'$. The map  \eqref{Intmap} is an {\em isomorphism} of free $S_{\R} = (S \otimes_{\Q} \R)$ modules of rank $2$. 
We will obtain an element of $S_{\R}^{\times}$
by comparing generators for $\wedge^2_{S_{\R}} LHS $ and $\wedge^2_{S_{\R}} RHS$:
We have 
\begin{equation} \label{lambdadef1} \mbox{ image of $\nu_1 \wedge_{S_{\R}} \nu_2$} =  \lambda (e_{1  } \wedge_{S_{\R}} e_{2 }) \ \ \mbox{ some $\lambda \in S_{\R}$} \end{equation}
 and $\lambda \in S_{\R}$ is the desired element; its norm is equal to $\Lambda'$.  
 A more explicit way to think about this is the following: 
 There are elements $\alpha, \beta, \gamma, \delta \in S_{\R}$ so that 
the period map  \eqref{Intmap} is given by 
 $$ f \left( \begin{array}{c} \nu_1 \\ \nu_2 \end{array} \right) = 
  \left( \begin{array}{cc}  \alpha & \beta \\ \gamma & \delta  . \end{array}\right) \cdot
   \left( \begin{array}{c} e_1 \\e_2 \end{array} \right).$$
 Then simply $\lambda = \alpha \delta-\beta \gamma \in S_{\R}$; also the norm of $\lambda$ is $\Lambda'$ as before.

We can now state the equivariant BSD conjecture.  The $\R$-linear extension of $\psi: S \rightarrow \C$
gives $\psi: S_{\R} \rightarrow \C$. 
We then  allow ourselves to denote $\psi(a)$ also by $a^{\psi}$. Then:
\begin{equation} \label{eqBSD} L( \frac{1}{2}, E \times \psi )   =   \left( \left[ \frac{\Sha_B \ \cdot \  \prod_{v} c_v(B) }{B(\Q) \hat{B}(\Q)} 
\cdot  \frac{(H^1_+/S e_1 + S e_2)}{ ( \Omega^1_{\mathcal{B}} / S \nu_1 + S \nu_2  )}    \right]  . \lambda. \right)^{\psi}  \mbox{ modulo } \psi(S^{\times}).
\end{equation} 
 part of the conjecture is that the square-bracketed term is a virtually principal $S$-module (see \S \ref{Cycback}) so that
 it gives an element of $S_{\Q}^{\times}/S^{\times}$ according to our conventions.  As before, we regard the right-hand side as $0$
 if $B(\Q)$ is infinite. Also, ``modulo $\psi(S^{\times})$'' means that
 the ratio of the two sides belongs to $\psi(S^{\times})$. 
 
As remarked previously, it is likely one can derive this from the equivariant Tamagawa number conjecture \cite{ETNC},  although we did not verify the details of this process. 
For the purpose of this paper, the phrase ``equivariant BSD'' refers to the formulation \eqref{eqBSD} above.

 Taking the corresponding conjecture with $\psi$ replaced by $\psi^i$, and taking product over $i \in (\Z/d)^*$,
 recovers the original BSD conjecture for $B$ -- at least up to algebraic units. 

\subsection{Explication}
Assume equivariant BSD.  Now since we are interested only in proving Proposition \ref{BSD}
we may suppose that $L(\frac{1}{2}, E \times \psi) \neq 0$;
then also (by \eqref{eqBSD}) we have that $B(\Q)$ is finite
and so $L(\frac{1}{2}, B) \neq 0$. We assume these in what follows.

Next let us explicitly choose $\nu_1, \nu_2, e_1,e_2$ as in the discussion
above  \eqref{UsualBSD3}.  The
inclusion $B \hookrightarrow A$ induces $H^1(A(\C)) \rightarrow H^1(B(\C))$. 
It's enough to produce forms $\nu_1, \nu_2, e_1, e_2$ on $A$
so that  the $R$-modules spanned by $\nu_1, \nu_2, e_1, e_2$ are free, 
and then we pull them back to $B$.   Then the $S$-modules spanned by $\nu_1, \nu_2$
and by $e_1, e_2$ are also free.  In order to compute the number $\lambda$
as above, it will be enough to do the corresponding computation on $A$
and then pull back to $B$.

 \begin{itemize}
 \item Choice of $e_i$:

 Recall that $H_1(A(\C), \Z)$ is free as $R$-module on basis $\gamma_1, \gamma_2, \overline{\gamma_1}, \overline{\gamma_2}$.

 We let $x_1, x_2, y_1, y_2  \in H^1(A(\C), \Z)$ be the dual basis (to the basis for $H_1(A(\C), \Z)$
 as $\Z$-module obtained by applying $\mu$ to $\gamma_1, \gamma_2,  \overline{\gamma_1}, \overline{\gamma_2}$).  
 In other words, for any $\alpha \in \mu$, 
 $$ \langle x_1, \alpha(\gamma_i ) \rangle = \begin{cases} 1, i=1, \alpha=\mathrm{id} \\ 0, else \end{cases}; \langle x_1, \overline{\alpha(\gamma_i)} \rangle = 0,$$
and $x_2$ is similarly dual to $\gamma_2$, $y_1$ to $\overline{\gamma_1}$, $y_2$ to $\overline{\gamma_2}$. 
 
 Then $H^1(A(\C), \Z)$
 is a free $R$-module on $x_1, x_2, y_1, y_2$. 
 Also, the images of $x_1, y_1$ in $H^1_+$ coincide, as do $x_2, y_2$; 
 and
 $$H^1(A(\C), \Z)_+ \mbox{ is a free $R$-module on $x_1, x_2$,}$$
 where we abuse notation by writing $x_1$ also for its image in $H^1_+$. 
%
%

 \item  Choice of $\nu_i$: 
 
 
 We have an isomorphism, from \eqref{Formiso}, 
  $$\left(   \Omega^1_{\mathcal{E}} \mathfrak{d}_F^{-1} \otimes_{\OO_F} \mathfrak{d}_{F_{\psi }/F}^{-1}\right) \stackrel{\sim}{\longrightarrow} \Omega^1_{\mathcal{A}/\Z}.$$ 
 
 Now choose a $\Z$-basis $\xi_1, \xi_2$ for $\Omega^1_{\mathcal{E}}\mathfrak{d}_F^{-1}$
 and take $\nu_i $ to be (the image of) $\xi_i \otimes x$, where $x \in \mathfrak{d}_{F_{\psi}/F}^{-1}$
is chosen to have the property that 
 $$ [ \mathfrak{d}_{F_{\psi}/F}^{-1}: \sum_{\alpha \in \mu}  \OO_F x^{\alpha}  ]$$
 is prime-to-$\ell$ (see next paragraph for why this is possible.) In particular,  
\begin{equation} \label{primetol}  \Omega^1_{\mathcal{A}}/\left( R \nu_1 + R \nu_2  \right ) \mbox{ is prime to $\ell$.} \end{equation}  
 As for why we can choose such an $x$:
We want to show that $\mathfrak{d}_{F_{\psi}/F}^{-1} \otimes_{\Z} \Z_{\ell}$
 has a ``normal basis'' over $\OO_F \otimes_{\Z} \Z_{\ell}$, i.e.
 there is $x \in \mathfrak{d}_{F_{\psi}/F}^{-1} \otimes_{\Z} \Z_{\ell}$ so that $\alpha  x \ (\alpha \in \mu)$
 spans as $\OO_F \otimes \Z_{\ell}$-module. This comes down to the fact that
 $F_{\psi}/F$ is tamely ramified at primes above $\ell$, and 
 Galois-stable ideals in tamely ramified extensions have (locally) normal bases: 
 
 In other words, let $\mathfrak{l}_i$ be the primes of $F$ above $\ell$,
 and let $\lambda_{ij}$ be the primes of $F_{\psi}$ above $\mathfrak{l}_i$. 
 Let $m_{ij}$ be the valuation of $\mathfrak{d}_{F_{\psi}/F}^{-1}$ at $\lambda_{ij}$. 
 We are asking that
 $$ \prod_{i,j}  \lambda_{ij}^{m_{ij}}   \OO_{F_{\psi, \lambda_{ij}}} \mbox{ have a normal basis over } \prod \OO_{F, \mathfrak{l}_i}.$$
 It is enough that $\prod_j  \lambda_{ij}^{m_{ij}}   \OO_{F_{\psi, \lambda_{ij}}} $ have a normal basis over $\OO_{F, \mathfrak{l}_i}$
 for each $i$ separately; say $i=1$. Next, because the Galois group permutes the various $\lambda_{1j}$, 
 and $m_{1j}$ doesn't depend on $j$, 
 it is enough to show that $\lambda_{11}^{m_{11}} \OO_{F_{\psi}, \lambda_{11}}$ has a normal basis over 
 over $\OO_{F, \mathfrak{l}_1}$. But that is a theorem of S. Ullom \cite[Theorem 1]{Ullom} because $F_{\psi}/F$ is tamely ramified. 
 
 Finally, we note for later use that in fact 
 \begin{equation} \label{xint} \mathrm{Norm}(\mathfrak{n}_{\psi}) x \in \OO_{F_{\psi}} \otimes \Z_{\ell}\end{equation} i.e.
 it is an algebraic integer above $\ell$.  Here $\Norm(\mathfrak{n}_{\psi}) \in \Z$ is the absolute ideal norm from ideals of $F$.
 
  In fact, it's enough to see that $\mathfrak{n}_{\psi} x \subset \OO_{F_{\psi}} \otimes \Z_{\ell}$, i.e.
 $\mathfrak{n}_{\psi} \mathfrak{d}_{F_{\psi}/F}^{-1} \subset \OO_{F_{\psi}} \otimes \Z_{\ell}$. 
 But, if $L/K$ is a tamely ramified extension of global fields
 then $\prod_{\mathfrak{q}} \mathfrak{q} \cdot \mathfrak{d}_{L/K}^{-1} \subset \OO_L$,
 where the product is over ramified primes $\mathfrak{q}$. (We want just the ``version above $\ell$'' of this.)  One reduces immediately 
 to the case of a tamely ramified extension of local fields, say with ramification index 
 $e$ and residue field degree $f$. In that case, 
 we can check the inclusion by taking norms of both sides;
 the valuation of the norm of $\mathfrak{q}$ is $ef$ and the valuation
 of the norm of $\mathfrak{d}_{L/K}$ is $e-1$; clearly $ef \geq e-1$. 
 
\end{itemize}

Recall that we may choose $\xi_1, \xi_2$
 in such a way that \begin{equation} \label{adef} \mathrm{Im}(\frac{\xi_2}{\xi_1}) =  \frac{ \sqrt{-\Delta_F/4}}{a} \end{equation} 
 where $a=[\Omega^1_{\mathcal{E}} \mathfrak{d}_F^{-1}: \OO_F \xi_1]$ is as in  \eqref{Cperiod}.    
 
%
%

 \subsection{Exterior product computations}

We compute $\lambda \in S_{\R}$ as in \eqref{lambdadef1} -- with respect to 
the images of $e_1, e_2, \nu_, \nu_2$ under the natural maps induced by $B \hookrightarrow A$ -- 
 by computing its analog $\widetilde{\lambda} \in R_{\R}$ computed ``on $A$'':
There exists $\widetilde{\lambda} \in R_{\R}$ such that
\begin{equation} \label{lambdadef2} \mbox{ image of $\nu_1 \wedge_{R_{\R}} \nu_2$} = \widetilde{ \lambda} (e_{1  } \wedge_{R_{\R}} e_{2 }) \ \ \mbox{ some $\widetilde{\lambda} \in R_{\R}$} \end{equation}
(where everything is computed on the abelian variety $A$). Then  the desired $\lambda \in S_{\R}$ is simply the image of $\widetilde{\lambda} \in R_{\R}$ under the natural map. 

In what follows, we write simply $\nu_1 \wedge \nu_2$ instead of $\nu_1 \wedge_{R_{\R}} \nu_2$. 
 
Recall that we take $e_1=x_1, e_2=x_2$ (see above).  

Put $\delta_i = \gamma_i + \overline{\gamma_i}$.
Regard integration on $\delta_i$ as being functionals  $H^1(A(\C), \R) \rightarrow \R$;
they factor through $H^1_+$. 
By  averaging them over $\mu$ we get $R$-linear functionals:
we  define $\Delta_i: H^1(A(\C), \R)_+ \rightarrow R_{\R}$ by the rule
$$\Delta_i (\omega) = \sum \alpha  \int_{\alpha(\delta_i)} \omega,  \ \ (i=1, 2; \omega \in H^1(A(\C),\R).),$$
which is now $R$-linear.
We now pair both sides of \eqref{lambdadef2}  with $\Delta_1 \wedge \Delta_2$ in order to compute $\lambda$. 

 Firstly,
 $$\langle x_1 , \Delta_1 \rangle = \sum_{\alpha} \alpha \int_{\alpha(\delta_1)}  x_1   = 1  \in R,$$
 and similarly
 $$ \langle x_2, \Delta_2 \rangle =1; \  \langle x_1, \Delta_2 \rangle = \langle x_2, \Delta_1 \rangle = 0.$$ 
Therefore,  \begin{equation} \label{x1x2}   \langle x_1 \wedge x_2 ,   \Delta_1 \wedge \Delta_2 \rangle  =  x_1(\Delta_1) x_2(\Delta_2) - x_2(\Delta_1) x_1(\Delta_2) = 1. \end{equation}  
 
 Next, compute
 $\langle \nu_1 \wedge \nu_2, \Delta_1 \wedge \Delta_2 \rangle$; it equals
\begin{eqnarray*} && \nu_1(\Delta_1) \nu_2(\Delta_2) -  \nu_2(\Delta_1) \nu_1(\Delta_2)    
  = \\  &&
\sum_{\alpha, \beta}   \alpha^{-1} \beta^{-1}  \left(  \langle \nu_1, \alpha^{-1}(\delta_1) \rangle \langle \nu_2, \beta^{-1}(\delta_2) \rangle
  -  \langle \nu_2,  \alpha^{-1}(\delta_1)  \rangle \langle \nu_1, \beta^{-1}(\delta_2)  \rangle \right), \end{eqnarray*} 
 wher $\int_{ \alpha^{-1}(\delta_1)} \nu_1$ has been abbreviated $  \langle \nu_1, \alpha_1^{-1}(\delta_1)  \rangle$ and so on. 
In turn
 \footnote{Write $\xi_2/\xi_1$ for the element $t \in F$ with $t \xi_1=\xi_2$. We used the following simple fact at \eqref{footnotestep}, with $q=
 \frac{\sigma_{\beta}(x) \xi_2 }{ \sigma_{\alpha}(x) \xi_1}$,   $$\Re(z_1) \Re(q z_2) - \Re(z_2) \Re(q z_1) = 
 \mathrm{Im}(q) \cdot \Re(z_1 \overline{z_2}).$$ } \begin{eqnarray} \label{Urkl}    \frac{1}{4} \left(  \langle \nu_1, {\alpha^{-1}(\delta_1)} \rangle \langle \nu_2, \beta^{-1}(\delta_2)  \rangle
 - \langle \nu_2,  \alpha_1^{-1} (\delta_1)  \rangle \langle \nu_1 , \beta^{-1}(\delta_2)  \rangle \right)  \\ \nonumber
 = \left( \mathrm{Re}(\int_{\gamma_1}  \sigma_{\alpha}(x) \xi_1) \mathrm{Re}(\int_{\gamma_2}  \sigma_{\beta}(x) \xi_2)
 -\mathrm{Re}(\int_{\gamma_2}  \sigma_{\alpha}(x) \xi_1) \mathrm{Re}(\int_{\gamma_1} \sigma_{\beta}(x) \xi_2) \right) 
 \\ = \label{footnotestep}  \mathrm{Im} \left( \frac{\sigma_{\beta}(x)}{ \sigma_{\alpha}(x)} \frac{\xi_2}{\xi_1} \right) \cdot \left( \mbox{ area of lattice of $E$ with respect to $1$-form $\sigma_\alpha(x) \xi_1$} \right)
\\ =   \frac{i}{2}  \left(  \int_{E(\C)}  \xi_1 \wedge   \overline{\xi_1}  \right) \cdot \mathrm{Im}   \left( \sigma_{\beta}(x) \overline{\sigma_{\alpha}(x)} \frac{\xi_2}{\xi_1} \right) 
    \end{eqnarray} 
and $\langle \nu_1 \wedge \nu_2, \Delta_1 \wedge \Delta_2 \rangle$ is obtained by summing this expression multiplied by $4 \alpha^{-1} \beta^{-1}$, over $\alpha$ and $\beta$.

%
 
  Note now that if we modify $\xi_2$ by a real multiple of $\xi_1$ the answer is unchanged (the contribution of $\alpha, \beta$
  and of $\beta, \alpha$ cancel in the summation).  
   Thus we may take $\xi_2 = \frac{1}{a} \sqrt{\Delta_F/4} \cdot \xi_1$,  where $a$ as in  \eqref{adef},  and then  from the definition of $\Omega_E$
   we see that:
\begin{eqnarray*} \langle  \nu_1 \wedge \nu_2 , \Delta_1 \wedge \Delta_2 \rangle & =&   \sum_{\alpha, \beta} \sqrt{-\Delta_F} \Omega_E \cdot \mathrm{Re} \left( \sigma_{\beta}(x) \overline{\sigma_{\alpha}(x)} \right)  \alpha^{-1} \beta^{-1}  \\ &  =&   \sqrt{-\Delta_F} [x][\overline{x}] \Omega_E \in R_{\R}. \end{eqnarray*}

(Recall that $[x] = \sum_{\alpha} \sigma_{\alpha}(x) \alpha^{-1}$ and $\overline{[x]} = \sum_{\alpha} \overline{\sigma_{\alpha}(x)} \cdot \alpha^{-1}$;
these belong to $R_{\C}$ but their product $[x] \overline{[x]}$ belongs to $R_{\R}$.) 
Comparing with \eqref{x1x2} and \eqref{lambdadef2} we see that
$$\widetilde{ \lambda} =  \sqrt{-\Delta_F} [x][\overline{x}] \Omega_E \in R_{\R}.$$
    We can now rewrite equivariant BSD from the form \eqref{eqBSD}. We have seen that $H^1_+$ is free on $e_1, e_2$, so
 \begin{equation} \label{eqBSD21}  \frac{ L( \frac{1}{2}, E \times \psi )  }{\Omega_E \sqrt{-\Delta_F}} =  \left(      [x]  \overline{[x]} \cdot \left[ \frac{\Sha_B\ \cdot \  \prod_{v} c_v(B) }{B(\Q) \hat{B}(\Q)}  
\cdot   \frac{1}  {\left(   \Omega^1_{\mathcal{B}} / S \nu_1+S\nu_2 \right) }\right] \right)^{\psi} \mbox{ mod $\psi(S^{\times})$.} \end{equation} 
   

   \subsection{Conclusion}
   We are almost finished.  First  note that for $M$ a finite $S$-module, we always have $[M] \leq [S/\# M]$;
   this is simply the fact that an ideal of $S$ divides its norm.  As we have mentioned, we can suppose that $B(\Q)$ and $\hat{B}(\Q)$ are finite. 
 So, examining the denominator of \eqref{eqBSD21},
\begin{eqnarray*}
[\hat{B}(\Q)_{\tors}] + [B(\Q)_{\tors}] +  [\Omega^1_{\mathcal{B}}/S\nu_1 + S \nu_2] &\leq&  2 [E(F_{\psi})_{\tors}]   + 
 \\  & & + [\Omega^1_{\mathcal{B}}/\mbox{image of }\Omega^1_{\mathcal{A}}]   + [\mbox{image of $\Omega^1_{\mathcal{A}}$}/R\nu_1 + R\nu_2] 
   \\ 
  & \leq_{\ell} & [S/ K] 
       \end{eqnarray*} 
     Here $K =  \# E(F_{\psi})_{\tors}^2 \cdot d^{2} \cdot  \left(   \# \Omega^1_{\mathcal{A}}/R\nu_1+R\nu_2 \right)$ and 
       $\leq_{\ell}$ means the equality holds ``at $\ell$,'' as mentioned before. 
This last inequality follows from \eqref{OmSr}. 
Since by \eqref{xint} we have that $ \Norm(\mathfrak{n}_{\psi})^2 \cdot \psi([x]\overline{[x]})$ is an algebraic integer,  we have proved 
   $$  \Norm(\mathfrak{n}_{\psi})^2 K\cdot \frac{L( \frac{1}{2}, E \times \psi ) }{  \Omega_E \sqrt{-\Delta_F}} \mbox{ has valuation $\geq 0$ at $\ell$,}$$
which finishes the proof,  because   $ \# \Omega^1_{\mathcal{A}}/R\nu_1+R\nu_2$ is prime-to-$\ell$ by \eqref{primetol}.

\section{Numerical computations} \label{Haluk}

\subsection{How much of the cohomology is base-change at higher levels ? }  \label{numerics}

Let $F$ be an imaginary quadratic field with its ring of integers $\mathcal{O}_F$. Given a positive integer $N$, 
consider the congruence subgroup $\Gamma_0((N))$ of level $(N)$ inside the associated Bianchi group $\SL_2(\mathcal{O}_F)$. 
Let $H^2_{bc}(\Gamma_0((N)), \C)^{new}$ denote the subspace of 
$H^2(\Gamma_0((N)), \C)^{new}$ which corresponds to Bianchi modular forms that are base-change of classical elliptic newforms and their twists 
(see Section \ref{BC-definition}). We are interested in the following question: \\
\begin{center}
{\em How much of $H^2(\Gamma_0((N)), \C)^{new}$ is exhausted by $H^2_{bc}(\Gamma_0((N)), \C)^{new}$ ?}  
\end{center}

Note that this question was investigated in \cite{rahm-sengun} for $N=1$ and more general coefficient modules. 
As part of an ongoing project \cite{sengun-tsaknias}, Panagiotis Tsaknias and the second author (M.H.\c{S}.) computed the dimension of 
$H^2_{bc}(\Gamma_0((N)),\C)^{new}$
for the following special case: \\

\begin{center}
{\em $F$ is ramified at a unique prime $p>2$, $N$ is square-free and prime to $p$.}
\end{center}

Over the fields $F=\Q(\sqrt{-d})$ with $d=3,7,11$, we have collected data to investigate the above question. 
For efficiency reasons, we computed $H_1(\Gamma_0((N)), \mathbb{F}_\ell)$ for six primes $\ell$ lying between $50$ and $100$ 
and took the minimum of the dimensions we got from these six mod $\ell$ computations. By the Universal Coefficients Theorem,  
this minimum is an upper-bound on the dimension of $H_1(\Gamma_0((N)), \C)$. However in practice, this upper-bound is very likely to give the actual dimension. 

We focused on three classes of ideals $(N) \triangleleft \mathcal{O}_F$: \\
\begin{itemize}
\item $N=p$ with $p$ rational prime that is {\em inert} in $F$ (Table 1). \\
\indent Here there are no oldforms.  Thus the base-change dimension formula, together with the number cusps, provides a 
lower bound for the dimension of $H_1(\Gamma_0((N)), \C)$. If this lower-bound agrees with our upper-bound 
coming from the mod $\ell$ computations, then we know for sure that the whole (co)homology is exhausted by 
base-change. As a result, the zero entries in ``non-BC" column of Table 1 are provenly correct. 

We also directly computed the char. 0 dimensions (the scope was smaller of course). The nonzero 
entries in the ``non-BC" columns of Table 1 which are in bold are proven to be correct as a result of these char. 0 computations. \\

\item $N=p$ with $p$ rational prime that is {\em split} in $F$  (Table 2). \\
\indent Here they may be oldforms and we can compute size of the oldforms using the data computed \cite{SengunExpMath}. 
Now the base-change dimension formula, the dimension of the old part, together with the number cusps, provide a 
lower bound for the dimension of $H_1(\Gamma_0((N)), \C)$. As above, if this lower-bound agrees with our upper-bound 
coming from the mod $\ell$ computations, then we know for sure that the whole (co)homology is exhausted by 
base-change. As a result, the zero entries in the ``non-BC" columns of Table 2 are provenly correct. 
 
We also directly computed the char. 0 dimensions for Table 2. The nonzero 
entries in the ``non-BC" columns of Table 2 which are in bold are proven to be correct as a result of these char. 0 computations. \\

\item $N=pq$ with $p,q$ rational primes that are {\em inert} in $F$ (Table 3). \\
\indent To compute the size of the oldforms, one can use the data computed for Table 1 (note that 
we only have to refer to entries of Table 1 which are provenly correct). As before, 
the zero entries in the ``non-BC" columns of Table 3 are provenly correct. 

We also directly computed the char. 0 dimensions for Table 3. The nonzero 
entries in the ``non-BC" columns of Table 3 which are in bold are proven to be correct as a result of these char. 0 computations. \\

\end{itemize}

Of course, there can be non-base change classes in the oldforms part, but this is not common. 
As we mentioned before, in Case (1), there are no oldforms. 
In Case (2),  extensive computations in \cite{SengunExpMath} show that $\%90$ 
of the time the cuspidal cohomology of $\Gamma_0(\p)$ (with trivial coefficients) vanishes, 
where $(p)=\p \bar{\p}$. So usually, we do not have oldforms in Case (2). But when we do, 
they are completely non-base change. For Case (3), there will be lots of old forms, however 
with little non-base change classes amongst them (which can detected via Table 1).  \\

In the Tables 1,2,3, the columns labeled ``new" denote the dimension of the new subspace and the columns labeled ``non-BC" denote the 
dimension of the dimension of the complement of the base-change subspace inside the new subspace.

\begin{longtable}{ccc|ccc|ccc}

\multicolumn{3}{c}{$d=3$}&  \multicolumn{3}{c}{$d=7$} & \multicolumn{3}{c}{$d=11$}\\ \hline
$p$ & {\text new} & {\text non-BC} &  $p$ & {\text new} & {\text non-BC} & $p$ & {\text new} & {\text non-BC} \\ \hline 
5	&	0	&	0	&	3	&	0	&	0	&	7	&	3	&	0	\\  
11	&	2	&	0	&	5	&	1	&	0	&	13	&	5	&	0	\\  
17	&	2	&	0	&	13	&	5	&	{\bf 2}&	17	&	9	&{\bf 2}	\\  
23	&	4	&	0	&	17	&	5	&	0	&	19	&	9	&	0	\\  
29	&	4	&	0	&	19	&	5	&	0	&	29	&	13	&	0	\\  
41	&	6	&	0	&	31	&	9	&	0	&	41	&	21	&{\bf 2}	\\  
47	&	10	&	{\bf 2}&	41	&	13	&	0	&	43	&	23	&{\bf 2}	\\  
53	&	8	&	0	&	47	&	15	&	0	&	61	&	29	&	0	\\  
59	&	12	&	{\bf 2}&	59	&	19	&	0	&	73	&	35	&	0	\\  
71	&	12	&	0	&	61	&	19	&	0	&	79	&	39	&	0	\\  
83	&	14	&	0	&	73	&	23	&	0	&	83	&	41	&	0	\\  
89	&	16	&	{\bf 2}&	83	&	29	&	{\bf 2}&	101	&	49	&	0	\\  
101	&	16	&	0	&	89	&	31	&	{\bf 2}&	107	&	55	&{\bf 2}	\\  
107	&	18	&	0	&	97	&	33	&	{\bf 2}&	109	&	53	&	0	\\  
113	&	22	&	{\bf 4}&	101	&	35	&	{\bf 2}&	127	&	63	&	0	\\  
131	&	22	&	0	&	103	&	33	&	0	&	131	&	65	&	0	\\  
137	&	22	&	0	&	131	&	43	&	0	&	139	&	69	&	0	\\  
149	&	26	&	{\bf 2}&	139	&	45	&	0	&	149	&	73	&	0	\\  
167	&	28	&	0	&	157	&	51	&	0	&	151	&	75	&	0	\\  
173	&	28	&	0	&	167	&	55	&	0	&	167	&	83	&	0	\\  
179	&	34	&	{\bf 4}&	173	&	57	&	0	&	173	&	85	&	0	\\  
191	&	34	&	{\bf 2}&	181	&	59	&	0	&	193	&	99	&{\bf 4} \\  
197	&	32	&	0	&	199	&	65	&	0	&	197	&	97	&	0	\\  
227	&	46	&	{\bf 8}&	223	&	73	&	0	&	211	&	105	&	0	\\  
233	&	38	&	0	&	227	&	75	&	0	&	227	&	113	&	0	\\  
239	&	40	&	0	&	229	&	75	&	0	&	233	&	121	&	6	\\  
251	&	42	&	0	&	241	&	79	&	0	&	239	&	119	&	0	\\  
257	&	42	&	0	&	251	&	87	&	{\bf 4}&	241	&	119	&	0	\\  
263	&	44	&	0	&	257	&	85	&	0	&	263	&	131	&	0	\\  
269	&	44	&	0	&	269	&	89	&	0	&	271	&	135	&	0	\\  
281	&	48	&	{\bf 2}&	271	&	89	&	0	&	277	&	137	&	0	\\  
293	&	48	&	0	&	283	&	93	&	0	&	281	&	139	&	0	\\  
311	&	54	&	{\bf 2}&	293	&	97	&	0	&	283	&	141	&	0	\\  
317	&	52	&	0	&	307	&	101	&	0	&	293	&	145	&	0	\\  
347	&	58	&	0	&	311	&	103	&	0	&	307	&	153	&	0	\\  
353	&	60	&	{\bf 2}&	313	&	103	&	0	&	337	&	167	&	0	\\  
359	&	62	&	{\bf 2}&	349	&	115	&	0	&	347	&	173	&	0	\\  
383	&	64	&	0	&	353	&	117	&	0	&	349	&	173	&	0	\\  
389	&	64	&	0	&	367	&	121	&	0	&	359	&	181	&	2	\\  
401	&	66	&	0	&	383	&	127	&	0	&	373	&	185	&	0	\\  
419	&	70	&	0	&	397	&	131	&	0	&	409	&	203	&	0	\\  
431	&	72	&	0	&	409	&	135	&	0	&	431	&	215	&	0	\\  
443	&	80	&	6	&	419	&	139	&	0	&	439	&	219	&	0	\\  
449	&	74	&	0	&	433	&	147	&	4	&	457	&	227	&	0	\\  
461	&	76	&	0	&	439	&	145	&	0	&	461	&	229	&	0	\\  
	&		&		&	461	&	153	&	0	&	479	&	239	&	0	\\  
	&		&		&	467	&	155	&	0	&	491	&	245	&	0	\\  
	&		&		&	479	&	159	&	0	&	503	&	251	&	0	\\  
	&		&		&	503	&	167	&	0	&	523	&	261	&	0	\\  
	&		&		&	509	&	169	&	0	&	541	&	269	&	0	\\  
	&		&		&	521	&	173	&	0	&	547	&	273	&	0	\\  
	&		&		&	523	&	173	&	0	&	557	&	277	&	0	\\  
	&		&		&	563	&	187	&	0	&	563	&	281	&	0	\\  
	&		&		&	577	&	191	&	0	&	569	&	283	&	0	\\  
	&		&		&	587	&	195	&	0	&	571	&	285	&	0	\\  
	&		&		&	593	&	197	&	0	&	593	&	295	&	0	\\  
	&		&		&	601	&	199	&	0	&	601	&	301	&	2	\\  
	&		&		&	607	&	201	&	0	&	607	&	303	&	0	\\  
	&		&		&	619	&	205	&	0	&	613	&	305	&	0	\\ 
\multicolumn{9}{c}{} \\
\caption{Level is $(p)$ with $p$ rational prime, {\em inert} in $F$}
\end{longtable}

\begin{longtable}{ccc|ccc|ccc}

\multicolumn{3}{c}{$d=3$}&  \multicolumn{3}{c}{$d=7$} & \multicolumn{3}{c}{$d=11$}\\ \hline
$p$ & {\text new} & {\text non-BC} &  $p$ & {\text new} & {\text non-BC} & $p$ & {\text new} & {\text non-BC} \\ \hline 
7	&	1	&	0	&	11	&	3	&	0	&	3	&	1	&	0	\\ 
13	&	1	&	0	&	23	&	7	&	0	&	5	&	1	&	0	\\ 
19	&	3	&	0	&	29	&	9	&	0	&	23	&	13	&{\bf 2}	\\ 
31	&	5	&	0	&	37	&	15	&	{\bf 4}&	31	&	17	&{\bf 2}	\\ 
37	&	5	&	0	&	43	&	13	&	0	&	37	&	17	&	0	\\ 
43	&	7	&	0	&	53	&	17	&	0	&	47	&	35	&{\bf 12}	\\ 
61	&	11	&	{\bf 2}&	67	&	21	&	0	&	53	&	25	&	0	\\ 
67	&	11	&	0	&	71	&	23	&	0	&	59	&	29	&	0	\\ 
73	&	11	&	0	&	79	&	25	&	0	&	67	&	33	&	0	\\ 
79	&	15	&	{\bf 2}&	107	&	37	&	{\bf 2}&	71	&	35	&	0	\\ 
97	&	15	&	0	&	109	&	35	&	0	&	89	&	43	&	0	\\ 
103	&	19	&	{\bf 2}&	113	&	37	&	0	&	97	&	47	&	0	\\ 
109	&	17	&	0	&	127	&	41	&	0	&	103	&	51	&	0	\\ 
127	&	29	&	{\bf 8}&	137	&	45	&	0	&	113	&	55	&	0	\\ 
139	&	25	&	{\bf 2}&	149	&	49	&	0	&	137	&	67	&	0	\\ 
151	&	27	&	{\bf 2}&	151	&	49	&	0	&	157	&	77	&	0	\\ 
157	&	27	&	{\bf 2}&	163	&	53	&	0	&	163	&	81	&	0	\\ 
163	&	29	&	{\bf 2}&	179	&	59	&	0	&	179	&	89	&	0	\\ 
181	&	31	&	{\bf 2}&	191	&	65	&	{\bf 2}&	181	&	89	&	0	\\ 
193	&	31	&	0	&	193	&	65	&	{\bf 2}&	191	&	97	&	2	\\ 
199	&	35	&	{\bf 2}&	197	&	67	&	{\bf 2}&	199	&	99	&	0	\\ 
211	&	37	&	{\bf 2}&	211	&	69	&	0	&	223	&	111	&	0	\\ 
223	&	37	&	0	&	233	&	77	&	0	&	229	&	115	&	2	\\ 
229	&	37	&	0	&	239	&	79	&	0	&	251	&	125	&	0	\\ 
241	&	47	&	{\bf 8}&	263	&	87	&	0	&	257	&	129	&	2	\\ 
271	&	47	&	{\bf 2}&	277	&	91	&	0	&	269	&	133	&	0	\\ 
277	&	45	&	0	&	281	&	97	&	4	&	311	&	157	&	2	\\ 
283	&	47	&	0	&	317	&	105	&	0	&	313	&	155	&	0	\\ 
307	&	53	&	{\bf 2}&	331	&	111	&	2	&	317	&	157	&	0	\\ 
313	&	51	&	0	&	337	&	111	&	0	&	331	&	165	&	0	\\ 
331	&	57	&	{\bf 2}&	347	&	115	&	0	&	353	&	175	&	0	\\ 
337	&	57	&	{\bf 2}&	359	&	119	&	0	&	367	&	183	&	0	\\ 
349	&	59	&	{\bf 2}&	373	&	123	&	0	&	379	&	189	&	0	\\ 
367	&	61	&	0	&	379	&	125	&	0	&	383	&	191	&	0	\\ 
373	&	61	&	0	&	389	&	129	&	0	&	389	&	193	&	0	\\ 
379	&	63	&	0	&	401	&	133	&	0	&	397	&	201	&	4	\\ 
397	&	67	&	2	&	421	&	139	&	0	&	401	&	199	&	0	\\ 
409	&	69	&	2	&	431	&	145	&	2	&	419	&	209	&	0	\\ 
421	&	69	&	0	&	443	&	147	&	0	&	421	&	209	&	0	\\ 
433	&	71	&	0	&	449	&	149	&	0	&	433	&	215	&	0	\\ 
	&		&		&	457	&	151	&	0	&	443	&	221	&	0	\\ 
	&		&		&	463	&	153	&	0	&	449	&	223	&	0	\\ 
	&		&		&	487	&	161	&	0	&	463	&	231	&	0	\\ 
	&		&		&	491	&	163	&	0	&	467	&	233	&	0	\\ 
	&		&		&	499	&	167	&	2	&	487	&	243	&	0	\\ 
	&		&		&		&		&		&	499	&	251	&	2	\\ 
	&		&		&		&		&		&	509	&	255	&	2	\\ 
	&		&		&		&		&		&	521	&	259	&	0	\\ 
	&		&		&		&		&		&	577	&	289	&	2	\\ 
	&		&		&		&		&		&	587	&	293	&	0	\\ 
\multicolumn{9}{c}{} \\
\caption{Here the level is $(p)$ with $p$ rational prime, {\em split} in $F$}
\end{longtable}

\begin{longtable}{ccc|ccc|ccc}

\multicolumn{3}{c}{$d=3$}&  \multicolumn{3}{c}{$d=7$} & \multicolumn{3}{c}{$d=11$}\\ \hline
$pq$ & {\text new} & {\text non-BC} &  $pq$ & {\text new} & {\text non-BC} & $pq$ & {\text new} & {\text non-BC} \\ \hline 
55	&	7	&	0	&	15	&	3	&	0	&	91	&	41	&	{\bf 4}   \\ 
85	&	11	&	0	&	39	&	11	&	{\bf 2}&	119	&	59	&	{\bf 10} \\ 
115	&	15	&	0	&	51	&	11	&	0	&	133	&	53	&	0	\\ 
145	&	19	&	0	&	57	&	13	&	0	&	203	&	85	&	0	\\ 
187	&	27	&	0	&	65	&	17	&	0	&	221	&	97	&	0	\\ 
205	&	31	&	{\bf 4}&	85	&	21	&	0	&	247	&	109	&	0	\\ 
235	&	33	&	{\bf 2}&	93	&	25	&	{\bf 4}&	287	&	121	&	0	\\ 
253	&	35	&	0	&	95	&	25	&	0	&	301	&	127	&	2	\\ 
265	&	35	&	0	&	123	&	27	&	0	&	323	&	145	&	0	\\ 
295	&	39	&	0	&	141	&	31	&	0	&	377	&	171	&	2	\\ 
319	&	47	&	0	&	155	&	43	&	{\bf 2}&	427	&	181	&	0	\\ 
355	&	47	&	0	&	177	&	39	&	0	&	493	&	225	&	0	\\ 
391	&	59	&	0	&	183	&	43	&	{\bf 2}&	511	&	219	&	2	\\ 
415	&	55	&	0	&	205	&	55	&	{\bf 2}&	533	&	243	&	2	\\ 
445	&	61	&	{\bf  2}&	219	&	49	&	0	&	551	&	253	&	0	\\ 
451	&	67	&	0	&	221	&	65	&	0	&	553	&	233	&	0	\\ 
493	&	75	&	0	&	235	&	61	&	0	&	559	&	253	&	0	\\ 
505	&	67	&	0	&	247	&	73	&	0	&	581	&	245	&	0	\\ 
517	&	75	&	0	&	249	&	55	&	0	&		&		&		\\ 
	&		&		&	267	&	59	&	0	&		&		&		\\ 
	&		&		&	291	&	65	&	0	&		&		&		\\ 
	&		&		&	295	&	81	&	{\bf 4}&		&		&		\\ 
	&		&		&	303	&	67	&	0	&		&		&		\\ 
	&		&		&	305	&	83	&	{\bf 2}&		&		&		\\ 
	&		&		&	309	&	69	&	0	&		&		&		\\ 
	&		&		&	323	&	97	&	0	&		&		&		\\ 
	&		&		&	365	&	97	&	0	&		&		&		\\ 
	&		&		&	393	&	87	&	0	&		&		&		\\ 
	&		&		&	403	&	121	&	0	&		&		&		\\ 
	&		&		&	415	&	111	&	2	&		&		&		\\ 
	&		&		&	417	&	93	&	0	&		&		&		\\ 
	&		&		&	445	&	117	&	0	&		&		&		\\ 
	&		&		&	471	&	105	&	0	&		&		&		\\ 
	&		&		&	485	&	131	&	2	&		&		&		\\ 
	&		&		&	501	&	113	&	2	&		&		&		\\ 
	&		&		&	505	&	133	&	0	&		&		&		\\ 
	&		&		&	515	&	139	&	2	&		&		&		\\ 
	&		&		&	519	&	117	&	2	&		&		&		\\ 
	&		&		&	527	&	161	&	0	&		&		&		\\ 
	&		&		&	533	&	161	&	0	&		&		&		\\ 
	&		&		&	543	&	121	&	0	&		&		&		\\ 
	&		&		&	589	&	181	&	0	&		&		&		\\ 
	&		&		&	597	&	133	&	0	&		&		&		\\ 
	&		&		&	611	&	185	&	0	&		&		&		\\ 
	&		&		&	655	&	173	&	0	&		&		&		\\ 
\multicolumn{9}{c}{} \\
\caption{Here the level is $(pq)$ with $p,q$ rational primes both {\em inert} in $F$}
\end{longtable}

\subsection{Cases with one-dimensional cuspidal cohomology}  \label{numerics2}
As mentioned in the Introduction, experiments in \cite{SengunExpMath} show that for the five Euclidean imaginary quadratic fields $F$, 
the cuspidal part of $H_1(\Gamma_0(\mathfrak{p}),\C)$, for $\Gamma_0(\mathfrak{p}) \leq \PSL_2(\mathcal{O}_F)$ with residue degree one 
prime level $\mathfrak{p}$ of norm $\leq 45000$, vanishes roughly $\%90$ of the time. In the remaining non-vanishing cases, 
the dimension is observed to be one in the majority of cases (see Table 16 of \cite{SengunExpMath} for details). \\

In a new experiment, we computed the dimension of the cuspidal part of  $H_1(\Gamma_0(\mathfrak{n}),\C)$, for $\Gamma_0(\mathfrak{n}) \leq 
\PGL_2(\mathcal{O}_F)$ with {\em all} \ levels  $\mathfrak{n}$ of norm $\leq 10000$ for the fields $F=\Q(\sqrt{-1}), \Q(\sqrt{-3})$. 
Again we see that a significant proportion of the non-vanishing cases have dimension exactly one. The distribution of the levels according to the dimension 
is given in Table 4. 

\begin{longtable}{ccc} 
\textrm{dim} & $\Q(\sqrt{-1})$ & $\Q(\sqrt{-3})$ \\ \hline
0 & 4170 & 3516 \\ 
1 & 614 & 526 \\ 
2 & 734 & 642 \\ 
3 & 341 & 266 \\ 
4 & 402 & 327  \\ 
5 & 183 & 168 \\ 
6 $\leq$ & 1409  & 1104 \\ 
total: & 7853 & 6048 \\ 
\multicolumn{3}{c}{} \\
\caption{distribution of subgroups $\Gamma_0(\mathfrak{n})$ with $\textrm{Norm}(\mathfrak{n}) \leq 10000$.}
\end{longtable}

\subsection{Growth of Regulators of Hyperbolic Tetrahedral Groups}  \label{numerics3}
In this section, we report on our numerical experiments related to the growth of regulators in the case of hyperbolic tetrahedral groups. Here we deal with combinatorial regulators rather than analytic ones. One may however prove that they both have either subexponential or exponential growth with respect to the index \cite{Luck}.

\subsubsection{Tetrahedral Groups}

A hyperbolic tetrahedral group is the index two subgroup consisting of orientation-preserving isometries in the discrete group generated 
by reflections in the faces of a hyperbolic tetrahedron whose dihedral angles are submultiples of $\pi$. It is well-known that 
there are 32 hyperbolic tetrahedral groups; 9 of them are cocompact. Among the 9 cocompact ones, only one is non-arithmetic (see \cite{sengun-tetrahedral}). \\

Let $\Delta$ be one of the 9 compact tetrahedra mentioned above, sitting in hyperbolic 3-space $\H^3$ and $\Gamma$ be the associated 
hyperbolic tetrahedral group. Let $\Sigma$ be a fundamental domain for $\Gamma$, viewed as a 3-dimensional simplicial complex. We can take $\Sigma$ to be the union $\Delta \cup \Delta^*$ where $\Delta^*$ is the copy of $\Delta$ obtained by reflecting $\Delta$ 
along one of its faces. Let $\T$ denote the triangulation of $\H^3$ obtained from $\Delta$, viewed as an infinite, locally finite simplicial complex with a cocompact cellular action of $\Gamma$ so that $\Gamma \backslash \T = \Sigma$. \\

Let $H$ be a finite index subgroup of $\Gamma$ and consider the vector space $M=\R[H \backslash \Gamma]$ 
with the natural $\Gamma$-action. It is well-known that the $\Gamma$-equivarient cohomology of $\T$ equals to the usual cohomology of $\Gamma$:
$$H^*(\Gamma,M) \simeq H^*_\Gamma(\T,M).$$
Put an inner product on $M$ by declaring the basis $H\backslash \Gamma$ to be 
orthonormal. We define the combinatorial Laplacians $\{ \Delta_i \}_i$ on the  $\Gamma$-equivarient cochain complex 
$\{ C^i(\T,M)^\Gamma \}_i$ (which computes the RHS) using the inner product 
$$\langle f,g \rangle_i^\Gamma := \sum_{\sigma \in \Sigma_i} \dfrac{1}{|\Gamma(\sigma)|} \langle f(\tilde{\sigma}),g(\tilde{\sigma})\rangle$$ 
where $\tilde{\sigma}$ is a lift of $\sigma$ in $\T_i$. See, for example,  \cite[Section 2.]{jordan-livne} for details. \\
	
For $i=1,2$, let $r_i$ denote volume of $H^i(\Gamma, \Z[H \backslash \Gamma])$ inside 
$H^i(\Gamma, \R[H \backslash \Gamma])$ with respect to the above inner product. 
According to our Proposition \ref{P1}, asymptotically $r_i$ behaves like the regulator $R_i$. 
For computational efficiency, we will compute another quantity, which, asymptotically speaking, gives us the desired information.  
Let $\tilde{r}_i$ denote the volume (w.r.t. to the same inner product) of the subspace of harmonic i-cochains, that is, 
the kernel of $\Delta_i$. Then it is not hard to see that
$$ \tilde{r}_i \geq r_i  \geq \dfrac{1}{\tilde{r}_i}.$$

We computed $\tilde{r}_i$ for prime level $\Gamma_0$-type subgroups $H$ (see \cite{sengun-tetrahedral}) of two cocompact hyperbolic tetrahedral groups 
$T6$ and $T8$, which, in the notation of \cite{sengun-tetrahedral}, can be identified as $T(4,3,2; 4,3,2)$ and $T(5,3,2; 4,3,2)$. 
While $T6$ is arithmetic, $T8$ is non-arithmetic. The data we collected are depicted in Tables 5 and 6 respectively. In the tables, 
the first column shows the index of the subgroup $H$ inside the tetrahedral group, the second column shows the dimension of the cohomology 
$H_1(H,\R)$. For the cases where this dimension is zero, the space of harmonic cochains is trivial and thus these cases are not included in the tables.

\begin{table}
\begin{tabular}{rr| rl | rl}

$[ T6 : H]$ &  rank &   \multicolumn{1}{c}{Log($\tilde{r}_1$)}  &  Log($\tilde{r}_1) \ / \ [ T6 : H]$ &    \multicolumn{1}{c}{Log($\tilde{r}_2$)}  &  Log($\tilde{r}_2) \ / \ [ T6 : H]$ \\ \hline
122    &  1      &  5.4161004    & 0.04439426559   &    4.4755991     & 0.03668523930 \\ 
170    &  5      & 28.1536040   & 0.1656094354   &    44.7684568     & 0.2633438638  \\ 
290    &  5      & 36.2058878   & 0.1248478891   &    30.5155222     &  0.1052259389 \\ 
362    &  7      & 45.4762539   & 0.1256250109   &    44.4415985     &  0.1227668467 \\ 
458    &  1      & 9.26712597   & 0.02023389951 &      6.4856782     &  0.01416086959\\ 
674    &  1      & 7.6487642     & 0.01134831485  &     8.3297444     &  0.01235867125 \\ 
962    & 11     & 78.0538394   & 0.08113704729  &   79.9289185     &  0.08308619394 \\ 
1034   &  2     & 17.8191345   & 0.01723320555  &   21.3528238     & 0.02065070001 \\ 
1370   &  1     & 7.0476963     & 0.00514430392 &   7.52250931      &  0.005490882711 \\  
1682   & 15    & 105.2828487 & 0.06259384583   & 116.4427193    & 0.06922872726 \\ 
1850   &  2     & 20.1698091   & 0.01090259953  &   20.7570214     &  0.01122001160 \\ 
2210   & 15    & 109.5835840 & 0.04958533211  &    &\\ 
2522   &  2     & 19.5918702   & 0.00776838630  &  &\\  
\multicolumn{6}{c}{} \\
\end{tabular}
\caption{Data for projective subgroups $H$ of $T6$}
\end{table}

\begin{table}
\begin{tabular}{rr | rl | rl}

$[ T8 : H]$ &  rank &  \multicolumn{1}{c}{ Log($\tilde{r}_1$)}  &  Log($\tilde{r}_1) \ / \ [ T8 : H]$ & \multicolumn{1}{c}{ Log($\tilde{r}_2$)}  &  Log($\tilde{r}_2) \ / \ [ T8 : H]$ \\ \hline
42   &      1  &   5.67652517    & 0.1351553613   &       8.25960528   & 0.1966572686  \\ 
82   &      4  &   15.39756182  & 0.1877751441   &     23.67090665     & 0.2886695933  \\ 
962  &     9   &  60.87067153  & 0.06327512633  &    79.69735698     & 0.08284548543  \\ 
1682 &   13  &  90.93035031   & 0.05406085036   &  &  \\ 
2402  &   4   &  26.74244163   & 0.01113340617   &   & \\ 
\multicolumn{6}{c}{} \\
\end{tabular}
\caption{Data for projective subgroups $H$ of $T8$}
\end{table}

\subsubsection{Arithmetic versus non-arithmetic}
As discussed at the end of the introduction (also see \cite[Chapter 9]{BV}), subexponential growth of the regulator with respect to the volume 
might be related to {\em arithmeticity}. Unfortunately, the scope of the data we collected here on the growth of the regulator is too limited to infer anything 
on this speculation.  However, the experiments in \cite{Dunfield, sengun-tetrahedral}, which inspect the growth of torsion, all suggest that 
if $M_0$ is non-arithmetic, then for a sequence $(M_i \rightarrow M_0)_{i \in \N}$ of finite covers of $M_0$ which is BS-converging to $\H^3$, the sequence
$$\frac{ \log \# H_1(M_i, \Z)_{\tors} }{V_i}$$ 
does {\em not} necessarily converge to $1 / (6 \pi)$ (the convergence is broken at covers with positive Betti numbers). 
If we believe that analytic torsion converges in this general setting then it must be that the regulator does not disappear in the limit, giving support 
to the above speculation.


\bibliography{bibli}

\providecommand{\bysame}{\leavevmode\hbox to3em{\hrulefill}\thinspace}
\providecommand{\MR}{\relax\ifhmode\unskip\space\fi MR }
\providecommand{\MRhref}[2]{%
  \href{http://www.ams.org/mathscinet-getitem?mr=#1}{#2}
}
\providecommand{\href}[2]{#2}
\begin{thebibliography}{10}

\bibitem{ABBGNRS}
M.~{Abert}, N.~{Bergeron}, I.~{Biringer}, T.~{Gelander}, N.~{Nikolov},
  J.~{Raimbault}, and I.~{Samet}, \emph{{On the growth of $L^2$-invariants for
  sequences of lattices in {L}ie groups}}, ArXiv e-prints (2012).

\bibitem{Aizenbud}
A.~Aizenbud, D.~Gourevitch, S.~Rallis, and G.~Schiffmann, \emph{Multiplicity
  one theorems}, Ann. of Math. (2) \textbf{172} (2010), no.~2, 1407--1434.
  \MR{2680495 (2011g:22024)}

\bibitem{AP06}
U.~K. Anandavardhanan and D.~Prasad, \emph{On the {${\rm SL}(2)$} period
  integral}, Amer. J. Math. \textbf{128} (2006), no.~6, 1429--1453. \MR{2275907
  (2008b:22014)}

\bibitem{PV}
\bysame, \emph{A local-global question in automorphic forms}, Compos. Math.
  \textbf{149} (2013), no.~6, 959--995. \MR{3077658}

\bibitem{BO}
Y.~Benoist and H.~Oh, \emph{Effective equidistribution of {$S$}-integral points
  on symmetric varieties}, Ann. Inst. Fourier (2010), no.~62, 1889--1942.

\bibitem{BLLS}
N.~{Bergeron}, P.~{Linnell}, W.~{L{\"u}ck}, and R.~{Sauer}, \emph{{On the
  growth of {B}etti numbers in $p$-adic analytic towers}}, ArXiv e-prints
  (2012), to appear in GGD.

\bibitem{BV}
N.~Bergeron and A.~Venkatesh, \emph{The asymptotic growth of torsion homology
  for arithmetic groups}, J. Inst. Math. Jussieu \textbf{12} (2013), no.~2,
  391--447. \MR{3028790}

\bibitem{BSD}
B.~J. Birch and H.~P.~F. Swinnerton-Dyer, \emph{Notes on elliptic curves.
  {II}}, J. Reine Angew. Math. \textbf{218} (1965), 79--108. \MR{0179168 (31
  \#3419)}

\bibitem{Bley}
W.~Bley, \emph{Numerical evidence for the equivariant {B}irch and
  {S}winnerton-{D}yer conjecture}, Exp. Math. \textbf{20} (2011), no.~4,
  426--456. \MR{2859900}

\bibitem{BW}
A.~Borel and N.~Wallach, \emph{Continuous cohomology, discrete subgroups, and
  representations of reductive groups}, second ed., Mathematical Surveys and
  Monographs, vol.~67, American Mathematical Society, Providence, RI, 2000.
  \MR{1721403 (2000j:22015)}

\bibitem{Neronmodels}
L{\"u}tkebohmert~W. Bosch, S. and M.~Raynaud, \emph{N\'eron models}, Ergebnisse
  der Mathematik und ihrer Grenzgebiete (3) [Results in Mathematics and Related
  Areas (3)], vol.~21, Springer-Verlag, Berlin, 1990. \MR{1045822 (91i:14034)}

\bibitem{Dunfield}
J.~{Brock} and N.~{Dunfield}, \emph{{Injectivity radii of hyperbolic integer
  homology 3-spheres}}, ArXiv e-prints (2013).

\bibitem{Brumley}
F.~Brumley, \emph{Effective multiplicity one on {${\rm GL}_N$} and narrow
  zero-free regions for {R}ankin-{S}elberg {$L$}-functions}, Amer. J. Math.
  \textbf{128} (2006), no.~6, 1455--1474. \MR{2275908 (2007h:11062)}

\bibitem{Bump}
D.l Bump, \emph{Automorphic forms and representations}, Cambridge Studies in
  Advanced Mathematics, vol.~55, Cambridge University Press, Cambridge, 1997.
  \MR{1431508 (97k:11080)}

\bibitem{Bushnell}
C.~J. Bushnell, \emph{Hereditary orders, {G}auss sums and supercuspidal
  representations of {${\rm GL}_N$}}, J. Reine Angew. Math. \textbf{375/376}
  (1987), 184--210. \MR{882297 (88e:22024)}

\bibitem{CE}
F.~Calegari and M.~Emerton, \emph{Bounds for multiplicities of unitary
  representations of cohomological type in spaces of cusp forms}, Ann. of Math.
  (2) \textbf{170} (2009), no.~3, 1437--1446. \MR{2600878 (2011c:22032)}

\bibitem{CV}
F.~{Calegari} and A.~{Venkatesh}, \emph{{A torsion {J}acquet--{L}anglands
  correspondence}}, ArXiv e-prints (2012).

\bibitem{Casselman}
B.~Casselman, \emph{The asymptotic behavior of matrix coefficients}, fragment
  of the notes {\em Introduction to admissible representations of $p$-adic
  groups} available on his web page.

\bibitem{Cheeger}
J.~Cheeger, \emph{Analytic torsion and the heat equation}, Ann. of Math. (2)
  \textbf{109} (1979), no.~2, 259--322.

\bibitem{Edixhoven}
B.~Edixhoven, \emph{N\'eron models and tame ramification}, Compositio Math.
  \textbf{81} (1992), no.~3, 291--306. \MR{1149171 (93a:14041)}

\bibitem{FLM}
T.~Finis, E.~Lapid, and W.~M{\"u}ller, \emph{On the degrees of matrix
  coefficients of intertwining operators}, Pacific J. Math. \textbf{260}
  (2012), no.~2, 433--456. \MR{3001800}

\bibitem{ETNC}
M.~Flach, \emph{The equivariant {T}amagawa number conjecture: a survey},
  Stark's conjectures: recent work and new directions, Contemp. Math., vol.
  358, Amer. Math. Soc., Providence, RI, 2004, With an appendix by C. Greither,
  pp.~79--125. \MR{2088713 (2005k:11132)}

\bibitem{Flicker}
Y.~Z. Flicker, \emph{Twisted tensors and {E}uler products}, Bull. Soc. Math.
  France \textbf{116} (1988), no.~3, 295--313. \MR{984899 (89m:11049)}

\bibitem{Flicker1}
\bysame, \emph{On distinguished representations}, J. Reine Angew. Math.
  \textbf{418} (1991), 139--172. \MR{1111204 (92i:22019)}

\bibitem{Gabai}
D.~Gabai, \emph{Foliations and the topology of {$3$}-manifolds}, J.
  Differential Geom. \textbf{18} (1983), no.~3, 445--503. \MR{723813
  (86a:57009)}

\bibitem{Goldfeld}
D.~Goldfeld, \emph{Modular forms, elliptic curves and the {$ABC$}-conjecture},
  A panorama of number theory or the view from {B}aker's garden ({Z}\"urich,
  1999), Cambridge Univ. Press, Cambridge, 2002, pp.~128--147. \MR{1975449
  (2004f:11048)}

\bibitem{Gross}
B.~H. Gross, \emph{On the conjecture of {B}irch and {S}winnerton-{D}yer for
  elliptic curves with complex multiplication}, Number theory related to
  {F}ermat's last theorem ({C}ambridge, {M}ass., 1981), Progr. Math., vol.~26,
  Birkh\"auser Boston, Mass., 1982, pp.~219--236. \MR{685298 (84e:14020)}

\bibitem{HLR}
G.~Harder, R.~P. Langlands, and M.~Rapoport, \emph{Algebraische {Z}yklen auf
  {H}ilbert-{B}lumenthal-{F}l\"achen}, J. Reine Angew. Math. \textbf{366}
  (1986), 53--120. \MR{833013 (87k:11066)}

\bibitem{HS}
M.~Hindry and J.~H. Silverman, \emph{Diophantine geometry}, Graduate Texts in
  Mathematics, vol. 201, Springer-Verlag, New York, 2000, An introduction.
  \MR{1745599 (2001e:11058)}

\bibitem{HL}
J.~Hoffstein and P.~Lockhart, \emph{Coefficients of {M}aass forms and the
  {S}iegel zero}, Ann. of Math. (2) \textbf{140} (1994), no.~1, 161--181, With
  an appendix by Dorian Goldfeld, Hoffstein and Daniel Lieman. \MR{1289494
  (95m:11048)}

\bibitem{GUE}
G.~Steil J.~Bolte and F.~Steiner, \emph{Arithmetical chaos and violation of
  universality in energy level statistics}, Physical Review Letters (1992),
  no.~69, 2188--2191.

\bibitem{JPS}
H.~Jacquet, I.~I. Piatetski-Shapiro, and J.~Shalika, \emph{Conducteur des
  repr\'esentations du groupe lin\'eaire}, Math. Ann. \textbf{256} (1981),
  no.~2, 199--214. \MR{620708 (83c:22025)}

\bibitem{JPS2}
H.~Jacquet and J.~A. Shalika, \emph{On {E}uler products and the classification
  of automorphic forms. {II}}, Amer. J. Math. \textbf{103} (1981), no.~4,
  777--815. \MR{623137 (82m:10050b)}

\bibitem{jordan-livne}
B.~W. Jordan and R.~Livn\'e, \emph{Integral {H}odge theory and congruences
  between modular forms}, Duke Math. J. \textbf{80} (1995), no.~2, 419--484,
  {\'e}.

\bibitem{KT}
S-I Kato and K.~Takano, \emph{Subrepresentation theorem for {$p$}-adic
  symmetric spaces}, Int. Math. Res. Not. IMRN (2008), no.~11, Art. ID rnn028,
  40. \MR{2428854 (2009i:22021)}

\bibitem{Lagier}
N.~Lagier, \emph{Terme constant de fonctions sur un espace sym\'etrique
  r\'eductif {$p$}-adique}, J. Funct. Anal. \textbf{254} (2008), no.~4,
  1088--1145. \MR{2381204 (2009d:22013)}

\bibitem{RL}
J.~Lansky and A.~Raghuram, \emph{On the correspondence of representations
  between {${\rm GL}(n)$} and division algebras}, Proc. Amer. Math. Soc.
  \textbf{131} (2003), no.~5, 1641--1648. \MR{1950297 (2003m:22021)}

\bibitem{Luck}
W.~L\"uck, personal communication.

\bibitem{Marshall}
S.~Marshall, \emph{Bounds for the multiplicities of cohomological automorphic
  forms on {${\rm GL}_2$}}, Ann. of Math. (2) \textbf{175} (2012), no.~3,
  1629--1651. \MR{2912713}

\bibitem{MW}
D.~Masser and G.~W{\"u}stholz, \emph{Isogeny estimates for abelian varieties,
  and finiteness theorems}, Ann. of Math. (2) \textbf{137} (1993), no.~3,
  459--472. \MR{1217345 (95d:11074)}

\bibitem{Millson}
J.~J. Millson, \emph{On the first {B}etti number of a constant negatively
  curved manifold}, Ann. of Math. (2) \textbf{104} (1976), no.~2, 235--247.
  \MR{0422501 (54 \#10488)}

\bibitem{Moreno}
C.~J. Moreno, \emph{The strong multiplicity one theorem for {${\rm GL}_{n}$}},
  Bull. Amer. Math. Soc. (N.S.) \textbf{11} (1984), no.~1, 180--182. \MR{741735
  (85j:11058)}

\bibitem{Mueller}
W.~M\"uller, \emph{Analytic torsion and ${R}$-torsion of {R}iemannian
  manifolds}, Adv. in Math. \textbf{28} (1978), no.~3, 233--305.

\bibitem{DP}
D.~Prasad, \emph{Invariant forms for representations of {${\rm GL}_2$} over a
  local field}, Amer. J. Math. \textbf{114} (1992), no.~6, 1317--1363.
  \MR{1198305 (93m:22011)}

\bibitem{rahm-sengun}
A.~Rahm and M.H. {\c{S}}eng{\"u}n, \emph{On level one cuspidal {B}ianchi
  modular forms}, LMS J. Comput. Math. \textbf{16} (2013), 187--199.

\bibitem{Rattcliffe}
J.~G. Ratcliffe, \emph{Foundations of hyperbolic manifolds}, second ed.,
  Graduate Texts in Mathematics, vol. 149, Springer, New York, 2006.
  \MR{2249478 (2007d:57029)}

\bibitem{Rudnick}
Z.~Rudnick, \emph{A central limit theorem for the spectrum of the modular
  domain}, Ann. Henri Poincar\'e \textbf{6} (2005), no.~5, 863--883.
  \MR{2219860 (2007b:11074)}

\bibitem{Sakellaridis}
Y.~Sakellaridis, \emph{On the unramified spectrum of spherical varieties over
  {$p$}-adic fields}, Compos. Math. \textbf{144} (2008), no.~4, 978--1016.
  \MR{2441254 (2010d:22026)}

\bibitem{SV}
Y.~Sakellaridis and A.~Venkatesh, \emph{Periods and harmonic analysis on
  spherical varieties}, preprint {\tt http://arxiv.org/abs/1203.0039},.

\bibitem{Sarnak}
P.~Sarnak, \emph{Letter to{ Z}.{R}udnick on multiplicities of eigenvalues for
  the modular surface}, available on
  http://publications.ias.edu/sarnak/paper/500.

\bibitem{SengunExpMath}
M.~H. {\c{S}}eng{\"u}n, \emph{On the integral cohomology of {B}ianchi groups},
  Exp. Math. \textbf{20} (2011), no.~4, 487--505. \MR{2859903}

\bibitem{sengun-tetrahedral}
\bysame, \emph{On the torsion homology of non-arithmetic hyperbolic tetrahedral
  groups}, Int. J. Number Theory \textbf{8} (2012), no.~2, 311--320.

\bibitem{sengun-tsaknias}
M.~H. {\c{S}}eng{\"u}n and P.~Tsaknias, \emph{Even {G}alois representations and
  torsion homology of {B}ianchi groups}, in progress.

\bibitem{S04}
J.~H. Silverman, \emph{A lower bound for the canonical height on elliptic
  curves over abelian extensions}, J. Number Theory \textbf{104} (2004), no.~2,
  353--372. \MR{2029512 (2004k:11106)}

\bibitem{Ullom}
S.~Ullom, \emph{Integral normal bases in {G}alois extensions of local fields},
  Nagoya Math. J. \textbf{39} (1970), 141--148. \MR{0263790 (41 \#8390)}

\bibitem{SparseEqui}
A.~Venkatesh, \emph{Sparse equidistribution problems, period bounds and
  subconvexity}, Ann. of Math. (2) \textbf{172} (2010), no.~2, 989--1094.
  \MR{2680486 (2012k:11061)}

\end{thebibliography}

\bibliographystyle{amsplain}

\end{document}